\documentclass[12pt]{amsart}
\usepackage{xypic}
\usepackage[all]{xy}
\usepackage{url}
\usepackage{amssymb}

\let\<\langle
\let\>\rangle
\def\figdir/{/u/levy/MSRI/Book54/main/}

\newcommand\inv{^{-1}}

\newcommand\ff{\mathfrak F}

\newcommand\G{\Gamma}

\newcommand\fL{\mathfrak L}
\newcommand\g{\gamma}
\newcommand\Ra{\mathbb R}\newcommand\Ta{\mathbb T}\newcommand\Pa{\mathbb P}
\newcommand\Ca{\mathbb C}
\newcommand\Za{\mathbb Z}
\newcommand\Qa{\mathbb Q}

\newcommand\Ha{\mathbb H}
\newcommand\sw{\mathcal W}

 \DeclareMathOperator{\Isom}{Isom}

\DeclareMathOperator{\Diff}{Diff} 
\DeclareMathOperator{\Aff}{Aff}

\DeclareMathOperator{\Homeo}{Homeo}
 
\DeclareMathOperator{\Hom}{Hom} \DeclareMathOperator{\Aut}{Aut}
 
\DeclareMathOperator{\CAT}{CAT} 

\DeclareMathOperator{\disp}{disp}\DeclareMathOperator{\Out}{Out}

\newtheorem{theorem}{Theorem}[section]
\newtheorem{proposition}[theorem]{Proposition}
\newtheorem{lemma}[theorem]{Lemma}

\newtheorem{defn}[theorem]{Definition}
\newtheorem{example}[theorem]{Example}
\newtheorem{question}[theorem]{Question}
\newtheorem{conjecture}[theorem]{Conjecture}
\newtheorem{problem}[theorem]{Problem}

\newtheorem{remarks}[theorem]{Remarks}
\newtheorem{remark}[theorem]{Remark}
\newtheorem{principle}[theorem]{Prinicple}

\begin{document}

\title[Groups acting on manifolds]{Groups acting on manifolds: around the Zimmer program}
\author{David Fisher}
\address{Department of Mathematics\\
Rawles Hall\\
Indiana University\\
Bloomington, IN 47405}
\email{fisherdm@indiana.edu}

\thanks{Author partially supported by NSF grants DMS-0541917 and 0643546, a fellowship from the Radcliffe Institute
for Advanced Studies and visiting positions at \'{E}cole Polytechnique, Palaiseau and Universit\'{e} Paris-Nord, Villateneuse.}

\dedicatory{To Robert Zimmer on the occasion of his 60th birthday.}

\begin{abstract}
This paper is a survey on the {\em Zimmer program}.  In its broadest
form, this program seeks an understanding of actions of large groups
on compact manifolds.  The goals of this survey are $(1)$
to put in context the original questions and conjectures of Zimmer
and Gromov that motivated the program, $(2)$ to indicate the current
state of the art on as many of these conjectures and questions as
possible and $(3)$ to indicate a wide variety of open problems and
directions of research.
\end{abstract}

\maketitle

\tableofcontents

\section{Prologue}
\label{section:prologue}

Traditionally, the study of dynamical systems is concerned with actions of
$\mathbb R$ or $\mathbb Z$ on manifolds.  I.e. with flows and
diffeomorphisms.  It is natural to consider instead dynamical
systems defined by actions of larger discrete or continuous groups.
For non-discrete groups, the Hilbert-Smith conjecture is relevant, since
the conjecture states that non-discrete, locally compact totally disconnected groups do not
act continuously by homeomorphisms on manifolds. For smooth actions known
results suffice to rule out actions of totally disconnected groups, the Hilbert-Smith conjecture has been proven for H\"{o}lder diffeomorphisms \cite{Maleshich, Repovus}. For infinite discrete
groups, the whole universe is open. One might consider the ``generalized
Zimmer program" to be the study of homomorphisms $\rho:\Gamma{\rightarrow}\Diff(M)$
where $\Gamma$ is a finitely generated group and $M$ is a compact manifold.
In this survey much emphasis
will be on a program proposed by Zimmer. Here one considers a very special class
of both Lie groups and
discrete groups, namely semi-simple Lie groups of higher real rank and
their lattices.

Zimmer's program is motivated by several of Zimmer's own
theorems and observations, many of which we will discuss below. But
in broadest strokes, the motivation is simpler.  Given a group whose
linear and unitary representations are very rigid or constrained,
might it also be true that the group's representations into
$\Diff^{\infty}(M)$ are also very rigid or constrained at least when
$M$ is a compact manifold? For the higher rank lattices considered by
Zimmer, Margulis' superrigidity theorems classified finite dimensional
representations and the groups enjoy property $(T)$ of Kazhdan, which makes
unitary representations quite rigid.

This motivation also stems from an analogy
between semi-simple Lie groups and diffeomorphism groups.  When
$M$ is a compact manifold, not only is $\Diff^{\infty}(M)$ an
infinite dimensional Lie group, but its connected component is
simple. Simplicity of the connected component of
$\Diff^{\infty}(M)$ was proven by Thurston using results of
Epstein and Herman \cite{Thurston-BAMS,Epstein-simple,Herman-torus}. Herman had used Epstein's
work to see that the connected component of
$\Diff^{\infty}(\Ta^n)$ is simple and Thurston's proof of the
general case uses this. See also further work on the topic by
Banyaga and Mather \cite{Banyaga-structure,Mather-BAMS,Mather-commutatorI,Mather-commutatorII,Mather-ICM}, as well as Banyaga's
book \cite{Banyaga-book}.

A major motivation for Zimmer's program came from his cocycle
superrigidity theorem, discussed in detail below in Section \ref{section:csr}.
One can think of a homomorphism $\rho:\Gamma{\rightarrow}\Diff(M)$ as
defining a {\em virtual homomorphism} from $\Gamma$ to $GL(\dim(M), \Ra)$.
The notion of {\em virtual homomorphisms}, now more commonly referred to
in our context as cocycles over group actions, was introduced by Mackey \cite{Mackey-virtual}.

In all of Zimmer's early papers, it was
always assumed that we had a homomorphism $\rho:\Gamma{\rightarrow}\Diff^{\infty}(M, \omega)$ where $\omega$ is
a volume form on $M$. A major motivation for this is that the cocycle superrigidity theorem referred to in the
 last paragraph only applies to cocycles over measure preserving actions. That this hypothesis is necessary was
confirmed by examples of Stuck, who showed that no rigidity could be
hoped for unless one had an invariant volume form or assumed that
$M$ was very low dimensional \cite{Stuck}.  Recent work of Nevo and Zimmer does
explore the non-volume preserving case and will be discussed below,
along with Stuck's examples and some others due to Weinberger in
Section \ref{section:novolume}.

Let $G$ be a semisimple real Lie group, all of whose simple factors
are of real rank at least two.  Let $\Gamma$ in $G$ be a lattice.
The simplest question asked by Zimmer was: can one classify smooth
volume preserving actions of $\Gamma$ on compact manifolds? At the
time of this writing, the answer to this question is still unclear.
There certainly are a wider collection of actions than Zimmer may
have initially suspected and it is also clear that the moduli space of
such actions is not discrete, see \cite{Benveniste-thesis,Fisher-deformations,KatokLewis-global} or section
\ref{section:exoticactions} below.  But there are still few
enough examples known that a classification remains plausible.
And if one assumes some additional conditions on the actions,
then plausible conjectures and striking results
abound.

We now discuss four paradigmatic conjectures. To make this introduction
accessible, these conjectures are all special cases of more general
conjectures stated later in the text.  In particular, we state all
the conjectures for lattices in $SL(n,\mathbb R)$ when $n>2$.  The reader not
familiar with higher rank lattices can consider the examples
of finite index subgroups of $SL(n,\Za)$, with $n>2$, rather
than for general lattices in $SL(n,\Ra)$. We warn the reader that the algebraic structure of $SL(n, \Za)$ and it's finite
 index subgroups makes many results easier for these groups than for other lattices in $SL(n, \Ra)$.
 The conjectures we state concern, respectively,
$(1)$ classification of low dimensional actions, $(2)$ classification of geometric actions
$(3)$ classification of uniformly hyperbolic actions and $(4)$ the topology of manifolds
admitting volume preserving actions.

A motivating conjecture for much recent research is the following:

\begin{conjecture}[Zimmer's conjecture]
\label{conjecture:zimmersconjecturesl}
For any $n>2$, any homomorphism $\rho:SL(n,\mathbb Z){\rightarrow}\Diff(M)$
has finite image if $\dim(M)<n-1$.  The same for any lattice $\G$ in $SL(n\Ra)$.
\end{conjecture}

The dimension bound in the conjecture is clearly sharp, as $SL(n, \mathbb R)$
and all of its subgroups act on $\mathbb P^{n-1}$.  The conjecture is a special
case of a conjecture of Zimmer which concerns actions of higher rank lattices
on low dimensional manifolds which we state as Conjecture \ref{conjecutre:lowdimgen} below.
Recently much attention has focused on these conjectures concerning low dimensional actions.

In our second, geometric, setting (a special case of) a major motivating conjecture is the following:

\begin{conjecture}[Affine actions]
\label{conjecture:rgsslactions}
Let $\Gamma<SL(n,\mathbb R)$ be a lattice.  Then there
is a classification of actions of $\Gamma$ on compact manifolds which
preserve both a volume form and an affine connection.  All such actions
are algebraically defined in a sense to be made precise below.  (See Definition \ref{definition:affine}
below.)
\end{conjecture}

The easiest example of an algebraically defined action is the action of
$SL(n,\Za)$ (or any of it's subgroups) on $\Ta^n$.  We formulate our
definition of algebraically defined action to include all trivial actions,
all isometric actions, and all (skew) products of other actions with these.
The conjecture is a special case of a conjecture stated
below as Conjecture \ref{conjecture:gromovzimmer}.  In addition to
considering more general acting groups, we will consider more
general invariant geometric structures, not just affine connections.
A remark worth making is that, for the geometric structures we
consider, the automorphism group is always a finite dimensional Lie
group.  The question of when the full automorphism group of a
geometric structure is large is well studied from other points of
view, see particularly \cite{Kobayashi-book}.  However, this
question is generally most approachable when the large subgroup is
connected and much less is known about discrete subgroups. In
particular, geometric approaches to this problem tend to use
information about the connected component of the automorphism group
and give much less information about the group of components
particularly if the connected component is trivial.

One is also interested in the possibility of classifying actions under
strong dynamical hypotheses.  The following conjecture is motivated by
work of Feres-Labourie and Goetze-Spatzier \cite{FeresLabourie, GoetzeSpatzier-Duke, GoetzeSpatzier-Annals} and is similar to
conjectures stated in \cite{Gorodnik-problemlist, Hurder-survey}:

\begin{conjecture}
\label{conjecture:hyperbolicclassificationspec}
Let $\Gamma<SL(n,\mathbb R)$ be a lattice.  Then there
is a classification of actions of $\Gamma$ on a compact manifold $M$ which
preserve both a volume form and where one element $\gamma{\in}\G$ acts as
an Anosov diffeomorphism.  All such actions
are algebraically defined in a sense to be made precise below. (Again see Definition \ref{definition:affine}.)
\end{conjecture}

In the setting of this particular conjecture, a proof of an older conjecture of Franks concerning Anosov
diffeomorphisms would imply that any manifold $M$ as in the conjecture was homeomorphic
to an infranilmanifold on which $\gamma$ is conjugate by a homeomorphism to a standard
affine Anosov map.  Even assuming Franks' conjecture, Conjecture \ref{conjecture:hyperbolicclassificationspec}
is open.  One can make a more general conjecture by only assuming that $\gamma$ has some uniformly partially hyperbolic
behavior.  Various versions of this are discussed in \S \ref{section:hyperbolic}, see particularly Conjecture \ref{conjecture:hyperbolicclassification}.

We end this introduction by stating a topological conjecture about {\em all} manifold admitting
smooth volume preserving actions of a simple higher rank algebraic group.  Very special cases of
this conjecture are known and the conjecture is plausible in light of existing examples.  More
precise variants and a version for lattice actions will be stated below in \S \ref{section:arithmeticquotients}.

\begin{conjecture}
\label{conjecture:fundamentalgroup}
Let $G$ be a simple Lie group of real rank at least two.  Assume $G$ has a faithful action preserving
volume on a compact manifold $M$.  Then $\pi_1(M)$ has a finite index subgroup $\Lambda$ such
that $\Lambda$ surjects onto an arithmetic lattice in a Lie group $H$ where $H$ locally contains
$G$.
\end{conjecture}

The conjecture says, more or less, that admitting a $G$ action forces the fundamental group
of $M$ to be large.  Passage to a finite index subgroup and a quotient is necessary, see
subsection \ref{section:exoticactions} and \cite{FisherWhyte-quotient} for more
discussion.  The conjecture might be considered the analogue for group actions of Margulis'
arithmeticity theorem.

This survey is organized on the following lines.  We begin in \S \ref{section:examples}
by describing in some detail the kinds of groups we consider and examples of their
actions on compact manifolds.  In \S \ref{section:avantlalettre}
we digress with a prehistory of motivating results from rigidity theory.
Then in \S \ref{section:lowdimensional}, we discuss conjectures and theorems concerning actions
on ``low dimensional" manifolds. In this section, there are a number of related
results and conjectures concerning groups not covered by Zimmer's original conjectures.
Here low dimensional is in two senses $(1)$
compared to the group as in Conjecture \ref{conjecture:zimmersconjecturesl} and
$(2)$ absolutely small, as in dimension being between $1$ and $4$.  This discussion
is simplified by the fact that many theorems conclude with the non-existence of
actions or at least with all actions factoring through finite quotients.  In section \ref{section:csr}
we further describe Zimmer's motivations for his conjectures by discussing the cocycle
superrigidity theorem and some of its consequences for smooth group actions.
In  \S \ref{section:rgs} and \S \ref{section:hyperbolic}, we discuss, respectively, geometric and
dynamical conditions under which a simple classification might be possible. In Section \ref{section:arithmeticquotients},
we discuss another approach to classifying $G$ and $\G$ actions using topology and representations of
 fundamental groups in order to produce algebraically defined quotients of actions.  Then in
\S \ref{section:exoticactions}, we describe the known ``exotic examples" of the acting groups
we consider.  This constructions reveals some necessary complexity of a high dimensional classification.
We then describe known results and examples of actions not preserving volume  in \S \ref{section:novolume} and
some rather surprising group actions on manifolds in \S \ref{section:moreexamples}.  Finally
we end the survey with a collection of questions about the algebraic and geometric structure of
finitely generated subgroups of $\Diff(M)$.

\subsection*{Some remarks on biases and omissions.} Like any
survey of this kind, this work is informed by its authors biases
and experiences.  There is an additional bias that this paper
emphasizes developments that are close to Zimmer's own work and
conjectures.  In particular, the study of rigidity of group actions often focuses
on the low dimensional setting where all group actions are conjectured
to be finite or trivial.  While Zimmer did substantial work in this
setting, he also proved many results and made many conjectures in more
general settings where any potential classification of group actions is necessarily,
due to the existence of examples, more complicated.

Another omission is that almost nothing will be said here about local
rigidity of groups actions, since the author has recently written another
survey on that topic \cite{Fisher-survey}.  While that survey could already
use updating, that update will appear elsewhere.

For other surveys of Zimmer's program and rigidity of large group actions
the reader is referred to \cite{FeresKatok-survey, Labourie-survey, Zimmer-CBMS}.
The forthcoming book by Witte Morris and Zimmer, \cite{Zimmer-CBMS}, is a particularly
useful introduction to ideas and techniques in Zimmer's own work.
Also of interest are $(1)$ a brief survey of rigidity theory by Spatzier with a more
geometric focus \cite{Spatzier-survey} $(2)$ an older survey also by Spatzier, with a somewhat
broader scope, \cite{Spatzier-alexandria} $(3)$  a recent problem list by Margulis on rigidity theory with
a focus on measure rigidity \cite{Margulis-millenium} and $(4)$ a more recent survey by Lindenstrauss focused
on recent developments in measure rigidity \cite{Lindenstrauss-survey}.  Both of the last two
mentioned surveys are particularly oriented towards connections between rigidity and number theory
which are not mentioned at all in this survey.

Finally, while all mistakes and errors in this survey are the sole responsibility
of its author, I would like to thank many people whose comments lead to improvements.
These include Danny Calegari, Renato Feres, Etienne Ghys, Daniel Groves, Steve Hurder, Jean-Francois Lafont, Karin Melnick, Nicolas Monod,
Dave Witte Morris, Leonid Polterovich, Pierre Py, Andr\'{e}s Navas, Yehuda Shalom, Ralf Spatzier, and an
anonymous referee.

\section{A brief digression: some examples of groups and actions}
\label{section:examples}

In this section we briefly describe some of the groups that will
play important roles in the results discussed here.  The reader
already familiar with semi-simple Lie groups and their lattices
may want to skip to the second subsection where we give
descriptions of group actions.  The following convention is in
force throughout this paper.  For definitions of relevant terms the
reader is referred to the following subsections.

\noindent {\bf Convention:} In this article we will have occasion to
refer to three overlapping classes of lattices in Lie groups which are
slightly different.  Let $G$ be a semisimple Lie group and $\Gamma<G$
a lattice.  We call $\Gamma$ a {\em higher rank lattice} if all simple
factors of $G$ have real rank at least $2$. We call $\Gamma$ a {\em lattice
with $(T)$} if all simple factors of $G$ have property $(T)$. Lastly
we call $\Gamma$ an {\em irreducible higher rank lattice} if  $G$ has
real rank at least $2$ and $\Gamma$ is irreducible.

\subsection{Semi-simple groups and their lattices.}
\label{subsection:groups}

By a simple Lie group, we mean a connected Lie group all of whose
normal subgroups are discrete, though we make the additional
convention that $\Ra$ and $S^1$ are not simple. By a semi-simple
Lie group we mean the quotient of a product of simple Lie groups
by some subgroup of the product of their centers.  Note that with
our conventions, the center of a simple Lie group is discrete and
is in fact the maximal normal subgroup. There is an elaborate
structure theory of semi-simple Lie groups and the groups are
completely classified, see \cite{Helgason-book} or \cite{Knapp-book} for details. Here
we merely describe some examples, all of which are matrix groups.
All connected semisimple Lie groups are discrete central
extensions of matrix groups, so the reader will lose very little
by always thinking of matrix groups.
\begin{enumerate} \item The groups $SL(n,\Ra), SL(n,\Ca)$ and
$SL(n,\Ha)$ of $n$ by $n$ matrices of determinant one over the
real numbers, the complex numbers or the quaternions. \item The
group $SP(2n,\Ra)$ of $2n$ by $2n$ matrices of determinant one
which preserve a real symplectic form on $\Ra^{2n}$.  \item The
groups $SO(p,q), SU(p,q)$ and $SP(p,q)$ of matrices which preserve
inner products of signature $(p,q)$ where the inner product is
real linear on $\Ra^{p+q}$, hermitian on $\Ca^{p+q}$ or
quaternionic hermitian on $\Ha^{p+q}$ respectively.
\end{enumerate}

Let $G$ be a semi-simple Lie group which is a subgroup of
$GL(n,\Ra)$.  We say that $G$ has {\em real rank} $k$ if $G$ has a
$k$ dimensional abelian subgroup which is conjugate to a subgroup
of the real diagonal matrices and no $k+1$ dimensional abelian
subgroups with the same property.  The groups in $(1)$ have rank
$n-1$, the groups in $(2)$ have rank $n$ and the groups in $(3)$
have rank $\min(p,q)$.

Since this article focuses primarily on finitely generated groups,
we are more interested in discrete subgroups of Lie groups than in
the Lie groups themselves.  A discrete subgroup $\G$ in a Lie
group $G$ is called a lattice if $G/{\G}$ has finite Haar measure.
The lattice is called {\em cocompact} or {\em uniform} if $G/{\G}$
is compact and {\em non-uniform} or {\em not cocompact}
otherwise.  If $G=G_1{\times}{\cdots}{\times}G_n$ is a product
then we say a lattice $\G<G$ is {\em irreducible} if its
projection to each $G_i$ is dense.  It is more typical
in the literature to insist that projections to all factors
are dense, but this definition is more practical for our
purposes.  More generally we make the
same definition for an {\em almost direct product}, by which we
mean a direct product $G$ modulo some subgroup of the center
$Z(G)$. Lattices in semi-simple Lie groups can always be
constructed by arithmetic methods, see \cite{Borel} and also
\cite{WitteMorris-book} for more discussion.  In fact, one of the most important
results in the theory of semi-simple Lie groups is that if $G$ is
a semi-simple Lie group without compact factors, then all
irreducible lattices in $G$ are arithmetic unless $G$ is locally
isomorphic to $SO(1,n)$ or $SU(1,n)$.  For $G$ of real rank at
least two, this is Margulis' arithmeticity theorem, which he
deduced from his super-rigidity theorems \cite{Margulis-ICM,Raghunathan-russian,Margulis-book}. For
non-uniform lattices, Margulis had an earlier proof which does not
use the superrigidity theorems, see \cite{Margulis-nonuniformone, Margulis-nonuniformtwo}. This earlier
proof depends on the study of dynamics of unipotent elements on
the space $G/{\Gamma}$, and particularly on what is now known as
the ``non-divergence of unipotent flows". Special cases of the
super-rigidity theorems were then proven for $Sp(1,n)$ and
$F_4^{-20}$ by Corlette and Gromov-Schoen, which sufficed to imply
the statement on arithmeticity given above \cite{Corlette-Annals,GromovSchoen}. As we
will be almost exclusively concerned with arithmetic lattices, we
do not give examples of non-arithmetic lattices here, but refer
the reader to \cite{Margulis-book} and \cite{WitteMorris-book} for more discussion.  A
formal definition of arithmeticity, at least when $G$ is algebraic
is:

\begin{defn}
\label{defn:arithmetic} Let $G$ be a semisimple algebraic Lie
group and $\G<G$ a lattice.  Then $\G$ is arithmetic if there
exists a semi-simple algebraic Lie group $H$
such that
\begin{enumerate}
\item there is a homomorphism $\pi:H^0{\rightarrow}G$ with compact
kernel, \item there is a rational structure on $H$ such that the
projection of the integer points of $H$ to $G$ are commensurable
to $\G$, i.e. $\pi(H(\Za)){\cap}\G$ is of finite index in both
$H(\Za)$ and $\G$.
\end{enumerate}
\end{defn}

We now give some examples of arithmetic lattices.  The simplest is
to take the integer points in a simple (or semi-simple) group $G$
which is a matrix group, e.g. $SL(n,\Za)$ or $Sp(n,\Za)$. This
exact construction always yields lattices, but also always yields
non-uniform lattices.  In fact the lattices one can construct in
this way have very special properties because they will contain
many unipotent matrices.  If a lattice is cocompact, it will
necessarily contain no unipotent matrices.  The standard trick for
understanding the structure of lattices in $G$ which become
integral points after passing to a compact extension is called
{\em change of base}.  For much more discussion see
\cite{Margulis-book,WitteMorris-book,Zimmer-book}.  We give one example to illustrate the process.
Let $G=SO(m,n)$ which we view as the set of matrices in
$SL(n+m,\Ra)$ which preserve the inner product
$$\<v,w\>=\biggl(-\sqrt{2}\sum_{i=1}^mv_iw_i\biggr)
+\biggl(\sum_{i=m+1}^{n+m}v_iw_i\biggr)$$
where $v_i$ and $w_i$ are the $i$th components of $v$ and $w$.
This form, and therefore $G$, are defined over the field
$\Qa(\sqrt{2})$ which has a Galois conjugation $\sigma$ defined by
$\sigma(\sqrt{2})=-\sqrt{2}$.  If we looks at the points
$\G=G(\Za[\sqrt{2}])$, we can define an embedding of $\G$ in
$SO(m,n){\times}SO(m+n)$ by taking $\g$ to $(\g,\sigma(\g))$. It
is straightforward to check that this embedding is discrete.  In
fact, this embeds $\G$ in $H=SO(m,n){\times}SO(m+n)$ as integral
points for the rational structure on $H$ where the rational points
are exactly the points $(\textbf{M},\sigma(\textbf{M}))$ where
$\textbf{M}{\in}G(\Qa(\sqrt{2}))$.  This makes $\G$ a lattice in $H$ and it
is easy to see that $\G$ projects to a lattice in $G$, since $G$
is cocompact in $H$.  What is somewhat harder to verify is that
$\G$ is cocompact in $H$, for which we refer the reader to the
list of references above.

Similar constructions are possible with $SU(m,n)$ or $SP(m,n)$ in
place of $SO(m,n)$ and also with more simple factors and fields
with more Galois automorphisms.  There are also a number of other
constructions of arithmetic lattices using division algebras. See
\cite{PlatonovRapinchuk, WitteMorris-book} for a comprehensive treatment.

We end this section by defining a key property of many semisimple
groups and their lattices.  This is property $(T)$ of Kazhdan, and
was introduced by Kazhdan in \cite{Kazhdan-T} in order to prove that
non-uniform lattices in higher rank semi-simple Lie groups are
finitely generated and have finite abelianization. It has played a
fundamental role in many subsequent developments.  We do not give
Kazhdan's original definition, but one which was shown to be
equivalent by work of Delorme and Guichardet \cite{Delorme,Guichardet}.

\begin{defn}
\label{definition:T} A locally compact group $\G$ has property $(T)$ of Kazhdan if
$H^1(\G,\pi)=0$ for every continuous unitary representation $\pi$
of $\G$ on a Hilbert space. This is equivalent to saying that any
continuous isometric action of $\G$ on a Hilbert space has a fixed
point.
\end{defn}

\begin{remarks}
\begin{enumerate} \item Kazhdan's
definition is that the trivial representation is isolated in the Fell topology on the
unitary dual of $\G$.

\item If a continuous group $G$ has property $(T)$ so does any
lattice in $G$.  This result was proved in \cite{Kazhdan-T}.

\item Any semi-simple Lie group has property $(T)$ if and only if
it has no simple factors locally isomorphic to $SO(1,n)$ or
$SU(1,n)$. For a discussion of this fact and attributions, see
\cite{delaHarpeValette-book}.  For groups with all simple factors of real rank at
least three, this is proven in \cite{Kazhdan-T}.

\item No noncompact amenable group, and in particular no noncompact
abelian group, has property $(T)$.  An easy averaging argument
shows that all compact groups have property $(T)$.
\end{enumerate}
\end{remarks}

\noindent Groups with property $(T)$ play an important role in
many areas of mathematics and computer science.

\subsection{Some actions of groups and lattices.}
\label{subsection:actions}

Here we define and give examples of a general class of actions.
A major impetus in Zimmer's work is determining optimal conditions
for actions to lie in this class.  The class we describe is slightly more
general than the class Zimmer termed ``standard actions" in e.g. \cite{Zimmer-ICM}.
Let $H$ be a
Lie group and $L<H$ a closed subgroup.  Then a diffeomorphism $f$
of $H/L$ is called {\em affine} if there is a diffeomorphism
$\tilde f$ of $H$ such that $f([h])=\tilde f(h)$ where $\tilde f =
A{\circ}\tau_h$ with $A$ an automorphism of $H$ with $A(L)=L$ and
$\tau_h$ is left translation by some $h$ in $H$.  Two obvious
classes of affine diffeomorphisms are left translations on any
homogeneous space and linear automorphisms of tori, or more
generally automorphisms of nilmanifolds.  A group action is called
{\em affine} if every element of the group acts by an affine
diffeomorphism. It is easy to check that the full group of affine
diffeomorphisms $\Aff(H/L)$ is a finite dimensional Lie group and
an affine action of a group $D$ is a homomorphism
$\pi:D{\rightarrow}\Aff(H/L)$. The structure of $\Aff(H/L)$ is
surprisingly complicated in general, it is a quotient of a
subgroup of the group $\Aut(H){\ltimes}H$ where $\Aut(H)$ is a the
group of automorphisms of $H$.  For a more detailed discussion of
this relationship, see \cite[Section 6]{FisherMargulis-lrc}.  While it is not
always the case that any affine action of a group $D$ on $H/L$ can
be described by a homomorphism
$\pi:D{\rightarrow}\Aut(H){\ltimes}H$, this is true for two
important special cases:
\begin{enumerate}
\item $D$ is a connected semi-simple Lie group and $L$ is a
cocompact lattice in $H$, \item $D$ is a lattice in a semi-simple
Lie group $G$ where $G$ has no compact factors and no simple
factors locally isomorphic to $SO(1,n)$ or $SU(1,n)$, and $L$ is a
cocompact lattice in $H$.
\end{enumerate}
\noindent These facts are \cite[Theorem 6.4 and 6.5]{FisherMargulis-lrc} where
affine actions as in $(1)$ and $(2)$ above are classified.

The most obvious examples of affine actions of large groups are of
the following forms, which are frequently referred to as {\em
standard actions}:
\begin{enumerate}
\item Actions of groups by automorphisms of nilmanifolds. I.e. let
$N$ be a simply connected nilpotent group, $\Lambda<N$ a lattice
(which is necessarily cocompact) and assume a finitely generated
group $\G$ acts by automorphisms of $N$ preserving $\Lambda$.  The
most obvious examples of this are when $N=\Ra^n$, $\Lambda=\Za^n$
and $\G<SL(n,\Za)$, in which case we have a linear action of $\G$
on $\Ta^n$.

\item Actions by left translations.  I.e. let $H$ be a Lie group
and $\Lambda<H$ a cocompact lattice and $\G<H$ some subgroup. Then
$\G$ acts on $H/\Lambda$ by left translations.  Note that in this
case $\G$ need not be discrete.

\item Actions by isometries.  Here $K$ is a compact group which
acts by isometries on some compact manifold $M$ and $\G<K$ is a
subgroup.  Note that here $\G$ is either discrete or a discrete
extension of a compact group.
\end{enumerate}

We now briefly define a few more general classes of actions,
which we need to formulate most of the conjectures in this paper. We
first fix some notations. Let $A$ and $D$ be topological groups,
and $B<A$ a closed subgroup. Let
$\rho:D{\times}A/B{\rightarrow}A/B$ be a continuous affine action.

\begin{defn}
\label{definition:affine}
\begin{enumerate}
\item
Let $A,B,D$ and $\rho$
be as above. Let $C$ be a compact group of affine diffeomorphisms
of $A/B$ that commute with the $D$ action. We call the action of
$D$ on $C{\backslash}A/B$ a {\em generalized affine action}.
\item
Let $A$, $B$, $D$ and $\rho$ be as in $1$ above.
Let $M$ be a compact Riemannian manifold and
$\iota:D{\times}A/B{\rightarrow}\Isom(M)$ a $C^1$ cocycle.  We
call the resulting skew product $D$ action on $A/B{\times}M$ a
{\em quasi-affine action}. If $C$ and $D$ are as in $2$, and we have
a smooth cocycle
$\alpha:D{\times}C{\backslash}A/B{\rightarrow}\Isom(M)$ , then we call the resulting skew product action of $D$ on
$C{\backslash}A/B{\times}M$ a {\em generalized quasi-affine
action}.
\end{enumerate}
\end{defn}

Many of the conjectures stated in this paper will end with
the conclusion that all actions satisfying certain hypotheses
are generalized quasi-affine actions.  It is not entirely
clear that generalized quasi-affine actions of higher rank
groups and lattices are much more general than generalized
affine actions. The following discussion is somewhat technical
and might be skipped on first reading.

One can always take a product of a generalized affine
action with a trivial action on any manifold to obtain a generalized
quasi-affine action.  One can also do variants on the following.  Let
$H$ be a semisimple Lie group and $\Lambda<H$ a cocompact lattice.  Let
$\pi:\Lambda{\rightarrow}K$ be any homomorphism of $\Lambda$ into a compact
Lie group.  Let $\rho$ be a generalized affine action of $G$ on $C\backslash H/\Lambda$
and let $M$ be a compact manifold on which $K$ acts.  Then there
is a generalized quasi-affine action of $G$ on $(C \backslash H \times M)/\Lambda$.

\begin{question}
\label{question:quasiaffine}
Is every generalized quasi-affine action of a higher rank simple
Lie group of the type just described?
\end{question}

The question amounts to asking for an understanding of compact
group valued cocycles over quasi-affine actions of higher rank
simple Lie groups.  We leave it to the interested reader to formulate
the analogous question for lattice actions.  For some work in this
direction, see the paper of Witte Morris and Zimmer \cite{MorrisZimmer-compactcocycles}.


\subsection{Induced actions}
\label{subsection:induction}

We end this section by describing briefly the standard
construction of an {\em induced or suspended action}.  This notion
can be seen as a generalization of the construction of a flow
under a function or as an analogue of the more algebraic notion of
inducing a representation.  Given a group $H$, a (usually closed)
subgroup $L$, and an action $\rho$ of $L$ on a space $X$, we can
form the space $(H{\times}X)/L$ where $L$ acts on $H{\times}X$ by
$h\cdot(l,x)=(lh{\inv},\rho(h)x)$.  This space now has a natural
$H$ action by left multiplication on the first coordinate.  Many
properties of the $L$ action on $X$ can be studied more easily in
terms of properties of the $H$ action on $(H{\times}X)/L$.  This
construction is particularly useful when $L$ is a lattice in $H$.

This notion suggests the following principle:

\begin{principle}
\label{lemma:induction}
Let $\Gamma$ be a cocompact lattice in a Lie group $G$.  To classify $\Gamma$
actions on compact manifolds it suffices to classify $G$ actions on compact manifolds.
\end{principle}

The principle is a bit subtle to implement in practice, since we clearly need a sufficiently
detailed classification of $G$ actions to be able to tell which one's arise
as induction of $\G$ actions.  While it is a bit more technical to state and
probably more difficult to use there is an analogous prinicple for non-cocompact lattices.
Here one needs to classify $G$ actions on manifolds  which are not compact
but where the $G$ action preserves a finite volume.  In fact, one needs only to study such
actions on manifolds that are fiber bundles over $G/\Gamma$ with compact
fibers.

The lemma begs the question as to whether or not one should simply always
study $G$ actions.  While in many settings this is useful, it is not always.
In particular, many known results about $\G$ actions require
hypotheses on the $\G$ action where there is no useful way of rephrasing the
property as a property of the induced action. Or, perhaps more awkwardly,
require assumptions on the induced action which cannot be rephrased in
terms of hypotheses on the original $\G$ action.  We will illustrate the
difficulties in employing Priniciple \ref{lemma:induction}
at several points in this paper.

A case where the implications of the principle are particularly clear is a negative
result concerning actions of $SO(1,n)$.  We will make clear by the proof what we
mean by:

\begin{theorem}
\label{theorem:rankoneactions}
Let $G=SO(1,n)$.  Then one cannot classify actions of $G$ on compact manifolds.
\end{theorem}

\begin{proof}
For every $n$ there is at least one lattice $\G<G$ which admits homomorphisms
onto non-abelian free groups, see e.g. \cite{Lubotzky-free}.  And therefore also onto $\Za$.  So we can
take any action of $F_n$ or $\Za$ and induce to a $G$ action.  It is relatively
easy to show that if the induced actions are isomorphic, then so are the actions
they are induced from, see e.g. \cite{Fisher-deformations}.   A classification of $\Za$ actions would amount to a classification
of diffeomorphisms and a classification of $F_n$ actions would involve classifying
all $n$-tuples of diffeomorphisms.  As there is no reasonable classification of diffeomorphisms
there is also no reasonable classification of $n$-tuples of diffeomorphisms.
\end{proof}

We remark that essentially the same theorem holds for actions of  $SU(1,n)$ where homomorphisms to $\Za$ exist for certain lattices \cite{Kazhdan-su}.  Much less is known about free quotients of lattices in $SU(1,n)$, see \cite{Livne-thesis}
for one example. For a surprising local rigidity result for some lattices in $SU(1,n)$,
see \cite[Theorem 1.3]{Fisher-H1}.

\section{Pre-history}
\label{section:avantlalettre}

\subsection{Local and global rigidity of homomorphisms into finite dimensional groups.}\label{subsection:selbergweil}

The earliest work on rigidity theory is a series of works by
Calabi--Vesentini, Selberg, Calabi and Weil, which resulted in the
following:

\begin{theorem}
\label{theorem:rigidityoflattices} Let $G$ be a semi-simple Lie
group and assume that $G$ is not locally isomorphic to
$SL(2,\Ra)$. Let $\G<G$ be an irreducible cocompact lattice, then
the defining embedding of $\G$ in $G$ is locally rigid, i.e. any embedding $\rho$
close to the defining embedding is conjugate to the defining embedding by a
small element of $G$.
\end{theorem}

\begin{remarks}
\begin{enumerate}   \item If $G=SL(2,\Ra)$
the theorem is false and there is a large, well studied space of
deformation of $\G$ in $G$, known as the Teichmueller space. \item
There is an analogue of this theorem for lattices that are not
cocompact.  This result was proven later and has a more
complicated history which we omit here. In this case it is also
necessary to exclude $G$ locally isomorphic to $SL(2,\Ca)$.
\end{enumerate}
\end{remarks}

This theorem was originally proven in special cases by Calabi,
Calabi--Vesentini and Selberg.   In particular, Selberg gives a
proof for cocompact lattices in $SL(n,\Ra)$ for $n\geq{3}$ in
\cite{Selberg}, Calabi--Vesentini give a proof when the associated
symmetric space $X=G/K$ is K\"{a}hler in \cite{CalabiVesentini} and Calabi
gives a proof for $G=SO(1,n)$ where $n{\geq}3$ in \cite{Calabi}.
Shortly afterwards, Weil gave a complete proof of Theorem
\ref{theorem:rigidityoflattices} in \cite{Weil-II,Weil-I}.

In all of the original proofs, the first step was to show that any
perturbation of $\G$ was discrete and therefore a cocompact
lattice. This is shown in special cases in \cite{Calabi,CalabiVesentini,Selberg} and
proven in a somewhat broader context than Theorem
\ref{theorem:rigidityoflattices} in \cite{Weil-I}.

The different proofs of cases of Theorem
\ref{theorem:rigidityoflattices} are also interesting in that
there are two fundamentally different sets of techniques employed
and this dichotomy continues to play a role in the history of
rigidity.   Selberg's proof essentially combines algebraic facts
with a study of the dynamics of iterates of matrices.  He makes
systematic use of the existence of singular directions, or Weyl
chamber walls, in maximal diagonalizable subgroups of $SL(n,\Ra)$.
Exploiting these singular directions is essential to much later
work on rigidity, both of lattices in higher rank groups and of
actions of abelian groups.   It seems possible to generalize
Selberg's proof to the case of $G$ an $\Ra$-split semi-simple Lie
group with rank at least $2$. Selberg's proof, which depended on
asymptotics at infinity of iterates of matrices, inspired Mostow's
explicit use of boundaries in his proof of strong rigidity
\cite{Mostow-personal}. Mostow's work in turn provided inspiration for the use
of boundaries in later work of Margulis, Zimmer and others on
rigidity properties of higher rank groups.

The proofs of Calabi, Calabi--Vesentini and Weil involve studying
variations of geometric structures on the associated locally
symmetric space. The techniques are analytic and use a variational
argument to show that all variations of the geometric structure
are trivial.  This work is a precursor to much work in geometric
analysis studying variations of geometric structures and also
informs later work on proving rigidity/vanishing of harmonic forms
and maps. The dichotomy between approaches based on
algebra/dynamics and approaches that are in the spirit of
geometric analysis  continues through much of the history of
rigidity and the history of rigidity of group actions in
particular.

Shortly after completing this work, Weil discovered a new
criterion for local rigidity \cite{Weil-remark}.  In the context of
Theorem \ref{theorem:rigidityoflattices}, this allows one to avoid
the step of showing that a perturbation of $\G$ remains discrete.
In addition, this result opened the way for understanding local
rigidity of more general representations of discrete groups than the
defining representation.

\begin{theorem}
\label{theorem:weil} Let $\G$ be a finitely generated group, $G$ a
Lie group and $\pi:\G{\rightarrow}G$ a homomorphism.  Then $\pi$
is locally rigid if $H^1(\G,\mathfrak{g})=0$.  Here $\mathfrak{g}$
is the Lie algebra of $G$ and $\G$ acts on $\mathfrak{g}$ by
$Ad_G{\circ}\pi$.
\end{theorem}

\noindent  Weil's proof of this result uses only the implicit
function theorem and elementary properties of the Lie group
exponential map.  The same theorem is true if $G$ is an algebraic
group over any local field of characteristic zero.  In \cite{Weil-remark},
Weil remarks that if $\G<G$ is a cocompact lattice and $G$
satisfies the hypothesis of Theorem
\ref{theorem:rigidityoflattices}, then the vanishing of
$H^1(\G,\mathfrak{g})$ can be deduced from the computations in
\cite{Weil-II}.  The vanishing of $H^1(\G,\mathfrak{g})$ is proven
explicitly by Matsushima and Murakami in \cite{Matsushima-Murakami}.

Motivated by Weil's work and other work of Matsushima, conditions
for vanishing of $H^1(\G,\mathfrak{g})$ were then studied by many
authors.  See particularly \cite{Matsushima-Murakami} and \cite{Raghunathan-1-vanishing}.
The results in these papers imply local rigidity of many linear
representations of lattices.

\subsection{Strong and super rigidity}
\label{subsection:nonlocalrigidity}

In a major and surprising development, it turns out that in many
instances, local rigidity is just the tip of the iceberg and that
much stronger rigidity phenomena exist.  We discuss now major developments
from the 60'S and 70's.

The first remarkable result in this direction is Mostow's rigidity
theorem, see \cite{Mostow-rigidity, Mostow-book} and references there.  Given $G$ as in
Theorem \ref{theorem:rigidityoflattices}, and two irreducible
cocompact lattices $\G_1$ and $\G_2$ in $G$, Mostow proves that
any isomorphism from $\G_1$ to $\G_2$ extends to an isomorphism of
$G$ with itself.  Combined with the principal theorem of
\cite{Weil-I} which shows that a perturbation of a lattice is again a
lattice, this gives a remarkable and different proof of Theorem
\ref{theorem:rigidityoflattices}, and Mostow was motivated by the
desire for a ``more geometric understanding" of Theorem
\ref{theorem:rigidityoflattices} \cite{Mostow-book}. Mostow's theorem is
in fact a good deal stronger, and controls not only homomorphisms
$\G{\rightarrow}G$ near the defining homomorphism, but any
homomorphism into any other simple Lie group $G'$ where the image
is lattice. As mentioned above, Mostow's approach was partially
inspired by Selberg's proof of certain cases of Theorem
\ref{theorem:rigidityoflattices}, \cite{Mostow-personal}.  A key step in
Mostow's proof is the construction of a continuous map between the
geometric boundaries of the symmetric spaces associated to $G$ and
$G'$.  Boundary maps continue to play a key role in many
developments in rigidity theory. A new proof of Mostow rigidity,
at least for $G_i$ of real rank one, was provided by Besson,
Courtois and Gallot.  Their approach is quite different and has
had many other applications concerning rigidity in geometry and
dynamics, see e.g. \cite{BCG-GAFA, BCG-ETDS, ConnellFarb}.

The next remarkable result in this direction is Margulis'
superrigidity theorem. Margulis proved this theorem as a tool to
prove arithmeticity of irreducible uniform lattices in groups of
real rank at least $2$.  For irreducible lattices in semi-simple
Lie groups of real rank at least $2$, the superrigidity theorems
classifies all finite dimensional linear representations.
Margulis' theorem holds for irreducible lattices in semi-simple
Lie groups of real rank at least two. Given a lattice $\G<G$ where
$G$ is simply connected, one precise statement of some of Margulis
results is to say that any linear representation $\sigma$ of $\G$
{\em almost extends} to a linear representation of $G$.  By this
we mean that there is a linear representation $\tilde \sigma$ of
$G$ and a bounded image representation $\bar \sigma$ of $\G$ such
that $\sigma(\g)=\tilde \sigma(\g)\bar \sigma(\g)$ for all $\g$ in
$G$. Margulis' theorems also give an essentially complete
description of the representations $\bar \sigma$, up to some
issues concerning finite image representations.  The proof here is
partially inspired by Mostow's work: a key step is the
construction of a measurable ``boundary map".  However the methods
for producing the boundary map in this case are very dynamical.
Margulis' original proof used Oseledec Multiplicative Ergodic
Theorem.  Later proofs were given by both Furstenberg and Margulis
using the theory of group boundaries as developed by Furstenberg
from his study of random walks on groups \cite{Furstenberg-poisson,Furstenberg-ICM}.
Furstenberg's probabilistic version of boundary theory has had a
profound influence on many subsequent developments in rigidity
theory. For more discussion of Margulis' superrigidity theorem,
see \cite{Margulis-ICM,Margulis-Inventsuperrigid,Margulis-book, Zimmer-book}.

Margulis theorem, by classifying all linear representations, and not
just one's with constrained images, leads one to believe that one
might be able to classify all homomorphisms to other interesting
classes of topological groups.  Zimmer's program is just one
aspect of this theory, in other directions, many authors have
studied homomorphisms to isometry groups of non-positively curved
(and more general) metric spaces.  See e.g.
\cite{Burger-ICM, GelanderKarlssonMargulis, Monod-irreducible}.

\subsection{Harmonic map approaches to rigidity}
\label{subsection:harmonic}
In the 90's, first Corlette and then Gromov-Schoen showed that one
could prove major cases of Margulis' superrigidity theorems also
for lattices in $Sp(1,n)$ and $F_4^{-20}$ \cite{Corlette-Annals, GromovSchoen}.
Corlette considered the case of representations over Archimedean fields and
Gromov and Schoen proved results over other local fields.
These proofs used harmonic maps and proceed in three steps.  First showing a harmonic map exists,
second showing it is smooth, and third using certain special Bochner-type formulas to show that
the harmonic mapping must be a local isometry. Combined with earlier work
of Matsushima, Murakami and Raghunathan, this leads to a complete
classification of linear representations for lattices in these groups \cite{Matsushima-Murakami,Raghunathan-1-vanishing}. It is worth noting that the
use of harmonic map techniques in rigidity theory had been pioneered by Siu,
who used them to prove generalizations of Mostow rigidity for certain
classes of K\"{a}hler manifolds \cite{Siu-Mostow}.  There is also much later
work on applying harmonic map techniques to reprove cases of Margulis' superrigidity
theorem by Jost-Yau and Mok-Siu-Yeung among others, see e.g. \cite{Jost-Yau, Mok-Siu-Yeung}.  We remark in passing
that the general problem of existence of harmonic maps for non-compact, finite
volume locally symmetric spaces has not been solved in general.  Results of Saper
and Jost-Zuo allow one to prove superrigidity for fundamental groups of many such
manifolds, but only while assuming arithmeticity \cite{JostZuo, Saper}.  The use of
harmonic maps in rigidity was inspired by the use of variational techniques and
harmonic forms and functions in work on local rigidity and vanishing of cohomology groups.
The original suggestion to use harmonic maps in this setting appears to go back to
Calabi \cite{Siu-YaleTalk}.

\subsection{A remark on cocycle super-rigidity}
\label{subsection:csrremark}
An important impetus for the study of rigidity of groups acting on manifolds
was Zimmer's proof of his cocycle superrigidity theorem.  We discuss this
important result below in section \ref{section:csr} where we also indicate
some of its applications to group actions.  This theorem is a generalization
of Margulis' superrigidity theorem to the class of {\em virtual homomorphisms}
corresponding to cocycles over measure preserving group actions.

\subsection{Margulis' normal subgroup theorem}
\label{subsection:normalsubgroups}
We end this section by mentioning another result of Margulis that has
had tremendous importance in results concerning group actions on manifolds.
This is the normal subgroups theorem, which says that any normal subgroup
in a higher rank lattice is either finite or of finite index see e.g. \cite{Margulis-book, Zimmer-book}
for more on the proof.  The proof precedes by a remarkable strategy.  Let $N$ be a normal subgroup of $\G$,
we show $\G/N$ is finite by showing that it is amenable and has property $(T)$.
This strategy has been applied in other contexts and is a major tool in
the construction of simple groups with good geometric properties see
\cite{BaderShalom, BurgerMozes-twotrees2, BurgerMozes-twotrees1, CapraceRemy}.  The proof that $\G/N$ is amenable
already involves one step that might rightly be called a theorem about rigidity
of group actions.  Margulis shows that if $G$ is a semisimple group of higher rank,
$P$ is a minimal parabolic, and $\G$ is a lattice then any measurable $\G$ space
$X$ which is a measurable quotient of the $\G$ action on $G/P$ is necessarily of the form
$G/Q$ where $Q$ is a parabolic subgroup containing $P$.  The proof of this result,
sometimes called the projective factors theorem, plays a fundamental role in
work of Nevo and Zimmer on non-volume preserving actions.  See \S \ref{section:exoticactions}
below for more discussion. It is also worth noting that Dani has proven a topological
analogue of Margulis' result on quotients \cite{Dani-quotient}.  I.e. he has proven
that continuous quotients of the $\Gamma$ action on $G/P$ are all $\Gamma$ actions
on $G/Q$.

The usual use of Margulis' normal subgroup theorem in studying rigidity of group
actions is usually quite straightforward.  If one wants to prove that a group satisfying
the normal subgroups theorem acts finitely, it suffices to find one infinite order
element that acts trivially.

\section{Low dimensional actions: conjectures and results}
\label{section:lowdimensional}

We begin this section by discussing results in particular, very low dimensions,
namely dimensions $1,2$ and $3$.

Before we begin this discussion, we recall a result of Thurston that is often
used to show that low dimensional actions are trivial \cite{Thurston-Reeb}.

\begin{theorem}[Thurston Stability]
\label{theorem:thurstonstability}
Assume $\G$ is a finitely generated group with finite abelianization. Let $\G$
act on a manifold $M$ by $C^1$ diffeomorphisms, fixing a point $p$ and with trivial derivative at $p$.
Then the $\G$ action is trivial in a neighborhood of $p$.  In particular, if
$M$ is connected, the $\G$ action is trivial.
\end{theorem}

The main point of Theorem \ref{theorem:thurstonstability} is that to show an action
is trivial, it often suffices to find a fixed point.  This is because, for the groups
we consider, there are essentially no non-trivial low dimensional linear representations
and therefore the derivative at a fixed point is trivial.  More precisely, the groups
usually only have finite image low dimensional linear representations, which allows
one to see that the action is trivial on a subgroup of finite index.

\subsection{Dimension one}
\label{subsection:thecircle}

The most dramatic results obtained in the Zimmer program concern a question
first brought into focus by Dave Witte Morris in \cite{Witte-circle}: can higher rank
lattices  act on the circle?  In fact, the paper \cite{Witte-circle} is more directly
concerned with actions on the line $\Ra$.  A detailed survey of results in this
direction is contained in the paper by Witte Morris in this volume, so we do
not repeat that discussion here.  We merely state a conjecture and a question.

\begin{conjecture}
\label{conjecture:higherrankcircle}
Let $\Gamma$ be a higher rank lattice.  Then any continuous
$\Gamma$ action on $S^1$ is finite.
\end{conjecture}

This conjecture is well known and first appeared in print in \cite{Ghys-circleaction}.
By results in \cite{Witte-circle} and \cite{LifchitzMorris-bgline},
the case of non-cocompact lattices is almost known. The paper \cite{Witte-circle}
does the case of higher $\mathbb Q$ rank.  The latter work of Lifchitz and Witte Morris reduces the general case to the case of quasi-split lattices in $SL(3,\Ra)$ and $SL(3, \Ca)$. It follows from work of Ghys or Burger-Monod that one can assume the action
fixes a point, see e.g. \cite{BurgerMonod,Burger-thisvolume, Ghys-circleboundedcohomology,Ghys-circleaction,Ghys-circlesurvey}, so the question is equivalent to asking if the groups
act on the line.  An interesting approach might be to study the induced action of $\G$ on
the space of left orders on $\G$, for ideas about this approach, we refer the reader to work of Navas and Witte Morris \cite{Morris-amenableline, Navas-orderdynamics}.

Perhaps the following should also be a conjecture, but here we only ask it is a
question.

\begin{question}
\label{question:TonS^1}
Let $\Gamma$ be a discrete group with property $(T)$.  Does $\Gamma$ admit an infinite action by $C^1$ diffeomorphisms
on $S^1$?  Does $\Gamma$ admit an infinite action by homeomorphisms on $S^1$?
\end{question}

By a result of Navas, the answer to the above question is no if $C^1$ diffeomorphisms are replaced by
$C^k$ diffeomorphisms for any $k>\frac{3}{2}$ \cite{Navas-Circlediff}.  A result noticed by the author and
Margulis and contained in a paper of Bader, Furman, Gelander, and Monod allows one to adapt this proof for
 values of $k$ slightly less than $\frac{3}{2}$ \cite{BaderFurmanGelanderMonod, Fisher-Margulis1}. By Thurston's Theorem \ref{theorem:thurstonstability}, in the $C^1$ case it suffices to find a fixed point for the action.

 We add a remark here pointed out to the author by Navas, that perhaps justifies only calling Question \ref{question:TonS^1} a question.
 In \cite{Navas-circleII}, Navas extends his results from \cite{Navas-Circlediff} to groups with relative
 property $(T)$.  This means that no such group acts on the circle by $C^k$ diffeomorphisms where $k>\frac{3}{2}$.
 However, if we let $\Gamma$ be the semi-direct product of a finite index free subgroup of $SL(2,\Za)$ and $\Za^2$,
 this group is left orderable, and so acts on $S^1$, see e.g. \cite{Morris-amenableline}.  This is the prototypical
 example of a group with relative property $(T)$.  This leaves open the possibility that groups with property $(T)$
 would behave in a similar manner and admit continuous actions on the circle and the line.

 Other possible candidate for a group with property $(T)$ acting on the circle by homeomorphisms
 are the more general variants of Thompson's group $F$ constructed by Stein in \cite{Stein-transam}.
 The proofs that $F$ does not have $(T)$ do not apply to these groups \cite{Farley}.

A closely related question is the following, suggested to the author by Andr\'{e}s Navas along with
examples in the last paragraph.

\begin{question}
\label{question:navas}
Is there  a group $\Gamma$ which is bi-orderable and does not admit a proper isometric
action on a Hilbert space?  I.e. does not have the Haagerup property?
\end{question}

For much more information on the fascinating topic of groups of diffeomorphisms of the
circle see both the survey by Ghys \cite{Ghys-circlesurvey}, the more recent book
by Navas \cite{Navas-book} and the article by Witte Morris in this volume \cite{WitteMorris-thisvolume}.

\subsection{Dimension $2$}
Already in dimension $2$ much less is known.  There are some results in the volume preserving setting.
In particular, we have

\begin{theorem}
\label{theorem:frankshandelpolterovich}
Let $\Gamma$ be a non-uniform irreducible higher rank lattice.
Then any volume preserving $\G$ action on a closed orientable surface other than $S^2$ is finite. If $\G$ has $\Qa$-rank at least one, then the same holds for
actions on $S^2$.
\end{theorem}

This theorem was proven by Polterovich for all surfaces but the sphere and shortly afterwards proven
by Franks and Handel for all surfaces \cite{FranksHandel-normalformhamiltonian, FranksHandel-distortionvolume, Polterovich}.  It is worth noting that the proofs use entirely different
ideas.  Polterovich's proof belongs very clearly to symplectic geometry and also implies some results
for actions in higher dimensions.  Franks and Handel use a theory of normal forms for $C^1$ surface diffeomorphisms
that they develop in analogy with the Thurston theory of normal forms for surface homeomorphisms with finite
fixed sets.  The proof of Franks and Handel can be adapted to a setting where one assume much less regularity of
the invariant measure \cite{FranksHandel-distortionmeasure}.

There are some major reductions in the proofs that are similar, which we now describe.  The first of these
should be useful for studying actions of cocompact lattices as well. Namely, one can assume that the
homomorphism $\rho:\Gamma{\rightarrow}\Diff(S)$ defining the action takes values in the connected component.
This follows from the fact that any $\rho:\Gamma{\rightarrow}MCG(S)$ is finite, which can now be deduced from a
variety of results \cite{BestvinaFujiwara, BurgerMonod,FarbMasur, KaimanovichMasur}. This result holds not only for the groups considered in Theorem \ref{theorem:frankshandelpolterovich}, but also for cocompact lattices and even for lattices in $SP(1,n)$ by
results of Sai-Kee Yeung \cite{Yeung-MCG}. It may be possible to show something similar for all groups
with property $(T)$ using recent results of Andersen showing that the mapping class group does not have
property $(T)$ \cite{Andersen-notT}. It is unrealistic to expect a simple analogue of this result for homomorphisms
to $\Diff(M)/\Diff(M)^0$ for general manifolds, see section \ref{subsection:farbshalen} below for a discussion
of dimension $3$.

Also, in the setting of Theorem \ref{theorem:frankshandelpolterovich},  Margulis normal subgroup theorem implies that it suffices to show that a single infinite order element of $\G$ acts trivially.  The proofs of Franks-Handel and
Polterovich then use the existence of distortion elements in $\G$, a fact established by Lubotzky, Mozes and Raghunathan in \cite{LubotzkyMozesRaghunathan}.
The main result in both cases shows that $\Diff(S,\omega)$ does not contain exponentially distorted elements and
the proofs of this fact are completely different.
An interesting result of Calegari and Freedman shows that this is not true of $\Diff(S^2)$ and that $\Diff(S^2)$
contains subgroups with elements of arbitrarily large distortion \cite{CalegariFreedman-distortion}.  These examples
are discussed in more detail in section \ref{section:exoticactions}.

Polterovich's methods also allow him to see that there are no exponentially distorted elements in
$\Diff(M,\omega)^0$ for certain symplectic manifolds $(M,\omega)$.  Again, this yields some
partial results towards Zimmer's conjecture.  On the other hand, Franks and Handel are able
to work with a Borel measure $\mu$ with some properties and show that any map $\rho:\G{\rightarrow}\Diff(S,\mu)$
is finite.

We remark here that a variant on this is due to Zimmer, when there is an invariant measure
supported on a finite set.

\begin{theorem}
\label{theorem:Tfixingpointonsurface}
Let $\Gamma$ be a group with property $(T)$ acting on a compact surface $S$
by $C^1$ diffeomorphisms.
Assume $\Gamma$ has a periodic orbit on $S$, then the $\Gamma$ action is finite.
\end{theorem}

 The proof
is quite simple.  First, pass to a finite index subgroup which fixes a point $x$. Look at the derivative representation $d\rho_x$ at the fixed point.  Since $\Gamma$ has property $(T)$, and $SL(2,\mathbb R)$
has the Haagerup property, it is easy to prove that the image of $d\rho_x$
is bounded.  Bounded subgroups of $SL(2,\mathbb R)$ are all virtually abelian
and this implies that the image of $d\rho_x$ is finite.  Passing to a subgroup
of finite index, one has a fixed point where the derivative action is trivial
of a group which has no cohomology in the trivial representation.  One now
applies Thurston Theorem \ref{theorem:thurstonstability}.

The difficulty in combining Theorem \ref{theorem:Tfixingpointonsurface} with ideas
from the work of Franks and Handel is that while Franks and Handel can show
that individual surface diffeomorphisms have large sets of periodic orbits,
their techniques do not easily yield periodic orbits for the entire large
group action.

We remark here that there is another approach to showing that lattices have no volume
preserving actions on surfaces.  This approach is similar to the proof that lattices have no $C^1$
actions on $S^1$ via bounded cohomology \cite{BurgerMonod, Burger-thisvolume, Ghys-circleboundedcohomology}.
That $\Diff(S,\omega)^0$ admits many interesting quasi-morphisms, follows from work of Entov-Polterovich, Gambaudo-Ghys and Py \cite{EntovPolterovich,GambaudoGhys,Py-CRAS, Py-thesis, Py-CRAS}.  By results of Burger and Monod, for any higher rank lattice and any homomorphism $\rho:\G{\rightarrow}\Diff(S,\omega)^0$, the image
is in the kernel of all of these quasi-morphisms.  (To make this statement meaningful and the kernel well-defined, one needs to take the homogeneous versions of the quasi-morphisms.)  What remains to be done is to extract useful
dynamical information from this fact, see the article by Py in this volume for more discussion \cite{Py-thisvolume}.

In our context, the work of Entov-Polterovich mentioned above really only constructs a single
quasi-morphism on $\Diff(S^2,\omega)^0$.  However, this particular quasimorphism
is very nice in that it is Lipschitz in metric known as Hofer's metric on
$\Diff(S^2,\omega)^0$.  This construction does apply more generally in higher dimensions and indicates connections between quasimorphisms and the geometry
of $\Diff(M,\omega)^0$, for $\omega$ a symplectic form, that are beyond the scope of this survey.  For an introduction to this fascinating topic, we refer the reader
to \cite{Polterovich-book}.


Motivated by the above discussion, we recall the following conjecture of Zimmer.

\begin{conjecture}
\label{conjecture:surfaces}
Let $\Gamma$ be a group with property $(T)$, then any volume preserving smooth $\Gamma$ action on a
a surface is finite.
\end{conjecture}

Here the words ``volume preserving" or at least ``measure preserving" are quite necessary.  As $SL(3,\mathbb R)$ acts on a $S^2$, so does any lattice in $SL(3,\mathbb R)$ or any irreducible lattice in $G$ where $G$ has $SL(3,\Ra)$ as
a factor.  The following question seem reasonable, I believe I first learned it from Leonid
Polterovich.

\begin{question}
\label{question:sl3}
Let $\Gamma$ be a higher rank lattice (or even just a group with property $(T)$) and assume
$\Gamma$ acts by diffeomorphisms on a surface $S$.  Is it true that either $(1)$ the action is finite
or $(2)$ the surface is $S^2$ and the action is smoothly conjugate to an action defined
by some embedding $i:\G{\rightarrow}SL(3,\Ra)$ and the projective action of $SL(3,\Ra)$
on $S^2$?
\end{question}

This question seems quite far beyond existing technology.

\subsection{Dimension 3}We now discuss briefly some work of Farb and Shalen that constrains actions
by homeomorphisms in dimension $3$ \cite{FarbShalen-3manifold}.
This work makes strong use of the geometry of $3$ manifolds, but uses very little about
higher rank lattices and is quite soft.  A special case of their results is the following:

\begin{theorem}
\label{theorem:farbshalenprime3manifold}
Let $M$ be an irreducible $3$ manifold and $\G$ be a higher rank lattice. Assume $\G$ acts on $M$ by homeomorphisms so that the action
on homology is non-trivial.  Then $M$ is homeomorphic to $\Ta^3$, $\G<SL(3,\Za)$ with finite index
and the $\G$ action on $H^1(M)$ is the standard $\G$ action on $Z^3$.
\end{theorem}

Farb and Shalen actually prove a variant of Theorem \ref{theorem:farbshalenprime3manifold} for an arbitrary
$3$ manifold admitting a homologically infinite action of a higher rank lattice.  This can be considered
as a significant step towards understanding when, for $\G$ a higher rank lattice and $M^3$ a closed three manifold,  $\rho:\G{\rightarrow}\Diff(M^3)$ must have image in $\Diff(M^3)^0$.
Unlike the results discussed above for dimension two, the answer is not simply ``always".  This three dimensional
result uses a great deal of the known structure of $3$ manifolds, though it does not use the full geometrization
conjecture proven by Perelman, but only the Haken case due to Thurston.  The same sort of result in higher dimensions
seems quite out of reach.  A sample question is the following.  Here we let $\G$ be a higher rank action
and $M$ a compact manifold.

\begin{question}
\label{question:isotopy}
Under what conditions on the topology of $M$ do we know that a homomorphism $\rho:\G{\rightarrow}\Diff(M)$
has image in $\Diff(M)^0$?
\end{question}

\subsection{Analytic actions in low dimensions}
\label{subsection:farbshalen}

We first mention a direction pursued by Ghys that is in a similar spirit to the Zimmer program,
and that has interesting consequences for that program.  Recall that the Zassenhaus lemma shows that
any discrete linear group generated by small enough elements is nilpotent.  The main
point of Ghys's article \cite{Ghys-proches} is to attempt to generalize this result
for subgroups of $\Diff^{\omega}(M)$.  While the result is not actually true in that
context, Ghys does prove some intriguing variants which yield some corollaries
for analytic actions of large groups.  For instance, he proves that $SL(n,\Za)$ for
$n>3$ admits no analytic action on the two sphere.  We remark that the attempt
to prove the Zassenhaus lemma for diffeomorphism groups suggests that one can
attempt to generalize other facts about linear groups to the category of diffeomorphism
groups.  We discuss several questions in this direction, mainly due to Ghys, in Section
\ref{section:ghysprogram}.

For the rest of this subsection we discuss a different approach of Farb and Shalen for showing that real analytic actions of large groups are finite.  This method is pursued in \cite{FarbShalen-analytic, FarbShalen-analyticS^1, FarbShalen-analyticdim4}.

We begin by giving a cartoon of the main idea. Given any action of a group $\Gamma$, an element
$\gamma{\in}\G$ and the centralizer $Z(\g)$, it is immediate that $Z(\g)$ acts on the set of $\g$
fixed points.   If the action is analytic, then the fixed sets are analytic and so have good structure
and are ``reasonably close" to being submanifolds.  If one further assumes that all normal subgroups
of $\Gamma$ have finite index, then this essentially allows one to bootstrap results about $Z(\g)$ not acting on manifolds of dimension at most $n-1$ to facts about $\G$ not having
actions on manifolds of dimension $n$ provided one can show that the fixed set for $\g$ is not
empty.  This is not true, as analytic sets are not actually manifolds, but the idea can be implemented
using the actual structure of analytic sets in a way that yields many results.

For example, we have:

\begin{theorem}
\label{theorem:farbshalen}
Let $M$ be a real analytic four manifold with zero Euler characteristic, then any
real analytic, volume preserving action of any finite index subgroup in $SL(n,\Za)$ for $n\geq7$
is trivial.
\end{theorem}

This particular theorem also requires a result of Rebelo \cite{Rebelo-T2} which concerns fixed sets
for actions of nilpotent groups on $\Ta^2$ and uses the ideas of \cite{Ghys-proches}.

The techniques of Farb and Shalen can also be used to prove that certain cocompact higher rank lattices
have no real analytic actions on surfaces of genus at least one.  For this result one needs only that the
lattice contains an element $\g$ whose centralizer already contains a higher rank lattice.  In the paper
\cite{FarbShalen-analytic} a more technical condition is required, but this can be removed using the results
of Ghys and Burger-Monod on actions of lattices on the circle. (This simplification was first pointed out to
the author by Farb.) The point is that $\g$ has fixed points for topological reasons and the set of these fixed points contains either $(1)$ a $Z(\g)$ invariant circle or $(2)$ a $Z(\g)$ invariant point.  Case $(1)$ reduces to case $(2)$ after passing to a finite index subgroup via the results on circle actions.  Case $(2)$ is dealt with by the proof of Theorem \ref{theorem:Tfixingpointonsurface}.

\subsection{Zimmer's full conjecture, partial results}

We now state the full form of Zimmer's conjecture.  In fact we generalize it slightly
to include all lattices with property $(T)$.  Throughout this subsection $G$ will be
a semisimple Lie group with property $(T)$ and $\G$ will be a lattice in $\G$.
We define two numerical invariants of these groups. First, for any group $F$, let $d(F)$ be
 the lowest dimension in which $F$ admits an infinite image linear representation. We note that the superrigidity
theorems imply that $d(\G)=d(G)$ when $\G$ is a lattice in $G$ and either $G$ is a semisimple
group property $(T)$ or $\G$ is irreducible higher rank.
The second number, $n(G)$ is the lowest dimension of a homogeneous space $K/C$ for a
compact group $K$ on which a lattice $\G$ in $G$ can act via a homomorphism $\rho:\Gamma{\rightarrow}K$.
In Zimmer's work, $n(G)$ is defined differently, in a way that makes clear that there is a bound on $n(G)$ that
does not depend on the choice of $\G$.  Namely, using the superrigidity theorems, it can be shown that $n$ satisfies
$$\frac{n(n+1)}{2}\geq\min\{\dim_{\Ca}G' \text{ where } G' \text{ is a simple factor of } G\}.$$

The more fashionable variant of Zimmer's conjecture is the following, first made
explicitly by Farb and Shalen for higher rank lattices in \cite{FarbShalen-analytic}.

\begin{conjecture}
\label{conjecutre:lowdimgen}
Let $\Gamma$ be a lattice as above.  Let $b=\min\{n,d\}$ and let $M$ be a manifold
of dimension less than $b-1$, then any $\Gamma$ action on $M$ is trivial.
\end{conjecture}

It is particularly bold to state this conjecture including lattices in $Sp(1,n)$
and $F_4^{-20}$. This is usually avoided because such lattices have abundant non-volume preserving
actions on certain types of highly regular fractals. I digress briefly to explain why I believe the
 conjecture is plausible in that case.  Namely such a lattice $\Gamma$, being a hyperbolic group, has many infinite proper quotients which are also hyperbolic groups \cite{Gromov-hyp}. If $\Gamma'$ is a quotient of $\Gamma$ by an
 infinite index, infinite normal subgroup $N$ and $\Gamma'$ is hyperbolic, then the boundary $\partial \Gamma'$ is a $\Gamma$ space with interesting dynamics and a good (Ahlfors regular) quasi-invariant measure class.  However, $\partial \Gamma'$ is only a manifold when it is a sphere. It seems highly unlikely that this is ever the case for these groups and even more unlikely that this
is ever the case with a smooth boundary action.  The only known way to build a hyperbolic group
which acts smoothly on its boundary is to have the group be the fundamental group of a compact negatively
curved manifold $M$ with smoothly varying horospheres.  If smooth is taken to mean $C^{\infty}$, this then implies the manifold is locally symmetric \cite{BCG-GAFA} and then superrigidity results make it impossible for this to occur with $\pi_1(M)$ a quotient of a lattice by an infinite index infinite normal subgroup.  If smooth only means $C^1$, then even in this
context, no result known rules out $M$ having fundamental group a quotient of a lattice in $Sp(1,n)$ or $F_4^{-20}$.
All results one can prove using harmonic map techniques in this context only rule out $M$ with
non-positive complexified sectional curvature.  We make the following conjecture, which is stronger than what
is needed for Conjecture \ref{conjecutre:lowdimgen}.

\begin{conjecture}
\label{conjecture:boundariesofquotients}
Let $\G$ be a cocompact lattice in $Sp(1,n)$ or $F_4^{-20}$ and let $\G'$ be a quotient of $\G$ which is
Gromov hyperbolic.
If $\partial \G'$ is a sphere, then the kernel of the quotient map is finite and $\partial \G' = \partial \G$.
\end{conjecture}

\noindent For background on hyperbolic groups and their boundaries from a point of view relevant to this
conjecture, see \cite{Kleiner-ICM}.  As a cautionary note, we point the reader to subsection \ref{subsection:farrelllafont} where we recall a construction of Farrell and Lafont that shows
that any Gromov hyperbolic group has an action by homeomorphisms on a sphere.

The version of Zimmer's conjecture that Zimmer made in \cite{Zimmer-ICM} and \cite{Zimmer-Mostow}
was only for volume preserving actions.  Here we break it down somewhat explicitly to clarify the
role of $d$ and $c$.

\begin{conjecture}[Zimmer]
\label{conjecture:lowdimvolume}
Let $\Gamma$ be a lattice as above and assume $\G$ acts smoothly on a compact manifold $M$
preserving a volume form.  Then if $\dim(M) < d$, the $\G$ action is isometric.  If, in
addition, $\dim(M)<n$ or if $\G$ is non-uniform, then the $\Gamma$ action is finite.
\end{conjecture}

Some first remarks are in order.  The cocycle super-rigidity theorems (discussed below)
imply that, when the conditions of the conjecture hold, there is always a measurable
invariant Riemannian metric.  Also, the finiteness under the conditions in the second
half of the conjecture follow from cases of Margulis' superrigidity theorem as soon
as one knows that the action preserves a smooth Riemannian metric.  So from one
point of view, the conjecture is really about the regularity of the invariant metric.

We should also mention that the conjecture is proven under several additional
hypotheses by Zimmer around the time he made the conjecture.  The first example
is the following.

\begin{theorem}
\label{theorem:rgslowdim}
Conjecture \ref{conjecture:lowdimvolume} holds provided the action also
preserves a rigid geometric structure, e.g. a torsion free affine connection or a
pseudo-Riemannian metric.
\end{theorem}

This is proven in \cite{Zimmer-Finitetype} for structures of finite type in the
sense of Elie Cartan, see also \cite{Zimmer-Mostow}.  The fact that roughly the same
proof applies for rigid structures in the sense of Gromov
was remarked in \cite{FisherZimmer}.  The point is simply that the isometry group
of a rigid structure acts properly on some higher order frame bundle and
that the existence of the measurable metric implies that $\G$ has bounded
orbits on all frame bundles as soon as $\G$ has property $(T)$.  This immediately implies that $\G$ must
be contained in a compact subgroup of the isometry group of the structure.

Another easy version of the conjecture follows from the proof of Theorem
\ref{theorem:Tfixingpointonsurface}.  That is:

\begin{theorem}
\label{theorem:Zfixingpointonsurface}
Let $G$ be a semisimple Lie group with property $(T)$ and $\G<G$ a lattice.  Let $\G$ act on a compact manifold $M$
by $C^1$ diffeomorphisms  where $\dim(M)<d(G)$.
If $\Gamma$ has a periodic orbit on $S$, then the $\Gamma$ action is finite.
\end{theorem}

A more difficult theorem of Zimmer shows that Conjecture \ref{conjecture:lowdimvolume} holds
when the action is distal in a sense defined in \cite{Zimmer-distal}.

\subsection{Some approaches to the conjectures}
\label{subsection:approaches}

\subsubsection{Discrete spectrum of actions}
\label{subsubsection:zimmerdiscretespectrum}

A measure preserving action of a group $D$ on a finite measure
space $(X,\mu)$ is said to have {\em discrete spectrum} if
$L^2(X,\mu)$ splits as sum of finite dimensional $D$ invariant
subspaces.  This is a strong condition that is (quite formally)
the opposite of weak mixing, for a detailed discussion see \cite{Furstenberg-book}.  It is a theorem of Mackey (generalizing
earlier results of Halmos and Von Neumann for $D$ abelian) that
an ergodic discrete spectrum $D$ action is measurably isomorphic
to one described by a dense embedding of $D$ into a compact group $K$
and considering a $K$ action on a homogeneous $K$-space.  The following
remarkable result of Zimmer from \cite{Zimmer-discretespectrum} is perhaps the strongest evidence for
Conjecture \ref{conjecture:lowdimvolume}.  This result is little known and has only
recently been applied by other authors, see \cite{FisherSilberman, FurmanMonod}.

\begin{theorem}
\label{theorem:discretespectrum}
Let $\G$ be a group with property $(T)$ acting by smooth, volume
preserving diffeomorphisms on a compact manifold $M$.  Assume in
addition that $\G$ preserves a measurable invariant metric.  Then
the $\G$ action has discrete spectrum.
\end{theorem}

This immediately implies that no counterexample to Conjecture
\ref{conjecture:lowdimvolume} can be weak mixing or even admit a
weak mixing measurable factor.

The proof of the theorem involves constructing finite dimensional subspaces of $L^2(M,\omega)$
that are $\G$ invariant.  If enough of these subspaces could be shown to be spanned
by smooth functions, one would have a proof of Conjecture \ref{conjecture:lowdimvolume}.
Here by ``enough" we simply mean that it suffices to have a collection of finite dimensional
$D$ invariant subspaces that separate points in $M$.  These functions would then specify
a $D$ equivariant smooth embedding of $M$ into $\Ra^N$ for some large value of $N$.

To construct finite dimensional invariant subspaces of $L^2(M)$, Zimmer uses an approach
similar to the proof of the Peter-Weyl theorem.  Namely, he constructs $\Gamma$ invariant kernels on $L^2(M \times M)$
which are used to define self-adjoint, compact operators on $L^2(M)$.  The eigenspaces of these
operators are then finite dimensional, $\G$ invariant subspaces of $L^2(M)$.  The kernels should be thought of as functions of the distance to the diagonal in $M \times M$.
The main difficulty here is that for these to be invariant by ``distance" we need to mean something defined
in terms of the measurable metric instead of a smooth one.  The construction of the kernels in this setting
is quite technical and we refer readers to the original paper.

It would be interesting to try to combine the information garnered from this theorem
with other approaches to Zimmer's conjectures.

\subsubsection{Effective invariant metrics}

We discuss here briefly another approach to Zimmer's conjecture, due to the author, which
seems promising.

We begin by briefly recalling the construction of the space of ``$L^2$
metrics" on a manifold $M$.  Given a volume form $\omega$ on $M$, we
can consider the space of all (smooth) Riemannian metrics on $M$ whose
associated volume form is $\omega$. This is the space of smooth sections
of a bundle $P{\rightarrow}M$.  The fiber of $P$ is $X=SL(n,\Ra)/SO(n)$.
The bundle $P$ is an associated bundle to the $SL(n,\Ra)$ sub-bundle of the
frame bundle of $M$ defined by $\omega$. The space $X$ carries a natural
$SL(n,\Ra)$-invariant Riemannian metric of non-positive curvature; we denote its associated distance function by $d_X$.
This induces a natural notion of distance on the space of metrics, given by
$d(g_1,g_2)^2=\int_M d_X(g_1(m),g_2(m))^2 d\omega$.
The completion of the sections with respect to the metric $d$ will be denoted
$L^2(M,\omega,X)$;  it is commonly referred to as the {\em
space of $L^2$ metrics on $M$} and its elements will be called $L^2$ \emph{metrics} on $M$.
That this space is $\CAT(0)$ follows easily from the fact that $X$ is $\CAT(0)$.  For more
discussion of $X$ and its structure as a Hilbert
manifold, see e.g. \cite{FisherHitchman-csr}. It is easy to check that a volume preserving
$\Gamma$ action on $M$ defines an isometric $\Gamma$ action on
$L^2(M,\omega,X)$.  Given a generating set $S$ for $\Gamma$ and a metric $g$ in $L^2(M,\omega,X)$,
we write $\disp(g)=max_{\gamma{\in}S} d(\gamma g, g)$.

Given a group $\G$ acting smoothly on $M$ preserving $\omega$, this
gives an isometric $\G$ action on $L^2(M,\omega, X)$ which preserves
the subset of smooth metrics.  Let $S$ be a generating set for $\G$.
We define an operator $P:L^2(M,\omega, X){\rightarrow}L^2(M,\omega, X)$
by taking a metric $g$ to the barycenter of the measure $\sum_S \delta_{(\gamma g)}$.
The first observation is a consequence of the (standard, finite dimensional) implicit
function theorem.

\begin{lemma}
\label{lemma:smoothpreserved}
If $g$ is a smooth metric, then $Pg$ is also smooth.
\end{lemma}

Moreover, we have the following two results.

\begin{theorem}
\label{theorem:Tcontracts}
Let $M$ be a surface and $\G$ a group with property $(T)$ and finite generating set $S$. Then there exists $0<C<1$
such that the operator $P$ satisfies:
\begin{enumerate}
\item $\disp(Pg)< C \disp(g)$
\item for any $g$, the $\lim_n(P^ng)$ exists and is $\G$ invariant.
\end{enumerate}
\end{theorem}

A proof of this theorem can be given by using the standard construction
of a negative definite kernel on $\Ha^2$ to produce a negative definite
kernel on $L^2(S,\mu,\Ha^2)$.  The theorem is then proved by transferring
the first property from the resulting $\G$ action on a Hilbert space.
The second property is an obvious consequence of the first and completeness
of the space $L^2(M,\omega, X)$.

\begin{theorem}
\label{theorem:zimmercontracting}
Let $G$ be a semisimple Lie group all of whose simple factors have property $(T)$
and $\G<G$ a lattice.  Let $M$ be a compact manifold such that $\dim(M)<d(G)$.
Then the operator $P$ on $L^2(M,\omega, X)$ satisfies the conclusions of Theorem
\ref{theorem:Tcontracts}.
\end{theorem}

This theorem is proven from results in \cite{FisherHitchman-IMRN} using convexity of
the distance function on $L^2(M,\omega, X)$.  For cocompact lattices a proof can be
given using results in \cite{KorevaarSchoen3} instead.

The problem now reduces to estimating the behavior of the derivatives of $P^ng$
for some initial smooth $g$.  This expression clearly involves derivatives of
random products of elements of $\G$, i.e derivatives of elements of $\G$ weighted
by measures that are convolution powers of equidistributed measure on $S$. The main
cause for optimism is that the fact that $\disp(P^ng)$ is small immediately implies
that the first derivative of any $\g{\in}S$ must be small when measured at that point in $L^2(M,\omega,X)$.
One can then try to use estimates on compositions of diffeomorphisms and convexity
of derivatives to control derivatives of $P^ng$.  The key difficulty is that the
initial estimate on the first derivative of $\g$ applied to $P^ng$ is only
small in an $L^2$ sense.


\subsection{Some related questions and conjectures on countable subgroups of
$\Diff(M)$}
\label{subsection:relatedquestions}

In this subsection, we discuss related conjectures and results on countable subgroups of $\Diff(M)$.
All of these are motivated by the belief that countable subgroups of $\Diff(M)$ are quite special,
though considerably less special than say linear groups.  We defer positive constructions of non-linear
subgroups of $\Diff(M)$ to Section \ref{section:moreexamples} and a discussion of possible algebraic and geometric
properties shared by all finitely generated subgroups of $\Diff(M)$ to Section \ref{section:ghysprogram}. Here we concentrate on groups which do not
act on manifolds, either by theorems or conjecturally.

\subsubsection{Groups with property $(T)$ and generic groups}
We begin by focusing on actions of groups with property $(T)$.
For a finitely generated group $\Gamma$, we recall that $d(\G)$ be the smallest dimension in
which $\G$ admits an infinite image linear representation.  We then make the following
conjecture:

\begin{conjecture}
\label{conjecture:volume}
Let $\G$ be a group with property $(T)$ acting smoothly on a compact
manifold $M$, preserving volume. Then if $\dim(M)<d(\G)$, the action
preserves a smooth Riemannian metric.
\end{conjecture}

For many groups $\G$ with property $(T)$, one can produce a measurable
invariant metric, see \cite{FisherSilberman}.  In fact, in \cite{FisherSilberman}, Silberman
and the author prove that there are many groups with property $(T)$
with no volume preserving actions on compact manifolds.  Key steps include
finding the invariant measurable metric, applying Zimmer's theorem
on discrete spectrum from \S \ref{subsection:approaches}, and producing
groups with no finite quotients and so no linear representations at all.

A result of Furman announced in \cite{Furman-random} provides some further evidence
for the conjecture.  This result is analogous to Proposition \ref{proposition:furstenberg}
and shows that any action of a group with property $(T)$ either leaves
invariant a measurable metric or has positive {\em random entropy}.  We
refer the reader to \cite{Furman-random} for more discussion.  While the
proof of this result is not contained in \cite{Furman-random}, it is possible
to reconstruct it from results there and others in \cite{FurmanMonod}.

Conjecture \ref{conjecture:volume} and the work in \cite{FisherSilberman} are motivated in part by the
following conjecture of Gromov:

\begin{conjecture}
\label{conjecture:gromovrandom}
There exists a model for random groups in which a ``generic" random group
admits no smooth actions on compact manifolds.
\end{conjecture}

It seems quite likely that the conjecture could be true for random groups
in the density model with density more than $\frac{1}{3}$.  These groups
have property $(T)$, see Ollivier's book \cite{Ollivier-book} for discussion on random groups.
For the conjecture to be literally true would require that a random hyperbolic
group have no finite quotients and it is a well-known question to determine
if there are any hyperbolic groups which are not residually finite, let alone
generic ones. If one is satisfied by saying the generic random group has only
finite smooth actions on compact manifolds, one can avoid this well-known open
question.

\subsubsection{Universal lattices} Recently, Shalom has proven that the groups $SL(n,\Za[X])$ have property $(T)$
when $n>3$.  (As well as some more general results.) Shalom refers to these groups as {\em universal lattices}. An interesting and approachable question is:

\begin{question}
\label{question:yehudagroups}
Let $\G$ be a finite index subgroup in $SL(n,\Za[X])$ and let $\G$ acting by diffeomorphisms
on a compact manifold $M$ preserving volume.  If $\dim(M)<n$ is there a measurable $\G$
invariant metric?
\end{question}

If one gives a positive answer to this question, one is then clearly interested in
whether or not the metric can be chosen to be smooth.  It is possible that this
is easier for these ``larger" groups than for the lattices originally considered by Zimmer.
One can ask a number of variants of this question, including trying to prove a full
cocycle superrigidity theorem for these groups, see below.  The question just asked
is particularly appealing as it can be viewed as a fixed point problem for the $\G$
action on the space of metrics (see \S \ref{subsection:approaches} for a definition).
In fact, one knows one has the invariant metric for the action of any conjugate of
$SL(n,\Za)$ and only needs to show that there is a consistent choice of invariant
metrics over all conjugates.  One might try to mimic the approach from \cite{Shalom-ICM},
though a difficulty clearly arises at the point where Shalom applies a scaling
limit construction.  Scaling limits of $L^2(M,\omega, X)$ can be described using non-standard
analysis, but are quite complicated objects and not usually isomorphic to the original
space.

\subsubsection{Irreducible actions of products}
Another interesting variant on Zimmer's conjecture is introduced in \cite{FurmanMonod}.
In that paper they study obstructions to irreducible actions of product groups.  An measure
preserving action of a product $\G_1 \times \G_2$ on a measure space $(X,\mu)$ is said to be {\em irreducible} if both $\G_1$ and $\G_2$ act ergodically on $X$.  Furman and Monod produce many obstructions to
irreducible actions, e.g. one can prove from their results that:

\begin{theorem}
\label{theorem:furmanmonod}
Let $\G=\G_1 \times \G_2 \times \G_3$.  Assume that $\G_1$ has property $(T)$ and no unbounded linear representations.
Then there are no irreducible, volume preserving $\G$ actions on compact manifolds.  The same is true for $\G=\G_1 \times \G_2$ if $\G_1$ is as above and $\G_2$ is solvable.
\end{theorem}

This motivates the following conjecture.

\begin{conjecture}
\label{conjecture:furmanmonodpluslubotzky}Let $\G=\G_1 \times \G_2 $.  Assume that $\G_1$ has property $(T)$ and no unbounded linear representations and that $\G_2$ is amenable.
Then there are no irreducible, volume preserving $\G$ actions on compact manifolds.
\end{conjecture}

One might be tempted to prove this by showing that no compact group contains both a dense finitely generated
amenable group and a dense Kazhdan group.  This is however not true.  The question was raised by Lubotzky
in \cite{Lubotzky-book} but has recently been resolved in the negative by Kassabov \cite{Kassabov-personal}.

In the context of irreducible actions, asking that groups have property $(T)$ is perhaps too strong.
In fact, there are already relatively few known irreducible actions of $\Za^2=\Za \times \Za$!  If one
element of $\Za^2$ acts as an Anosov diffeomorphism, then it is conjectured by Katok and Spatzier
that all actions are algebraic \cite{KatokSpatzier}.  Even in more general settings where one
assumes non-uniform hyperbolicity, there are now hints that a classification might be possible, though
clearly more complicated than in the Anosov case, see \cite{KalininKatokRodriguezHertz} and references there.  In a sense that paper indicates that the ``exotic examples" that might arise in this context may be no worse than those that arise for actions of higher rank lattices.

\subsubsection{Torsion groups and $\Homeo(M)$}

A completely different and well studied aspect of the theory of transformation groups
of compact manifolds is the study of finite subgroups of $\Homeo(M)$.  It has recently
been noted by many authors that this study has applications to the study of ``large subgroups"
of $\Homeo(M)$.  Namely, one can produce many finitely generated and even finitely presented
groups that have no non-trivial homomorphisms to $\Homeo(M)$ for any compact $M$.  This is
discussed in e.g. \cite{BridsonVogtmann, FisherSilberman, Weinberger-thisvolume}. As far
as I know, this observation was first made by Ghys as a remark in the introduction to \cite{Ghys-proches}.
In all cases, the main trick is to construct infinite, finitely generated groups which contains
infinitely many conjugacy classes of finite subgroups, usually just copies of $(\Za/p\Za)^k$
for all $k>1$.  A new method of constructing such groups was recently introduced by Chatterji and
Kassabov \cite{ChatterjiKassabov}.

As far as the author knows, even if one fixes $M$ in advance, this is the only existing method
for producing groups $\Gamma$ with no non-trivial (or even no infinite image) homomorphisms to $\Homeo(M)$ unless $M=S^1$.  For $M=S^1$ we refer back to subsection \ref{subsection:thecircle} and to \cite{WitteMorris-thisvolume}.

\section{Cocycle superrigidity and immediate consequences}
\label{section:csr}

\subsection{Zimmer's cocycle super-rigidity theorem and generalizations}
\label{subsection:Zcsr}

A main impetus for studying rigidity of group actions on manifolds
came from Zimmer's theorem on superrigidity for cocycles.  This
theorem and its proof were strongly motivated by Margulis' work. In
fact, Margulis' theorem reduces to the special case of Zimmer's theorem for a certain cocycle
$\alpha:G{\times}G/{\Gamma}{\rightarrow}\G$. In
order to avoid technicalities, we describe only a special case of
this result, essentially avoiding boundedness and integrability assumptions on cocycles
that are automatic fulfilled in any context arising from a continuous action on a
compact manifold. Let $M$ be a compact manifold, $H$ a matrix group and
$P$ an $H$ bundle over $M$. For readers not familiar with bundle theory,
the results are interesting even in the case where $P=M \times H$.  Now let a
group $\G$ act on $P$
continuously by bundle automorphisms, i.e. such that there is a $\G$ action on $M$ for which
 the projection from $P$ to $M$ is equivariant.
Further assume that the action on $M$ is measure preserving and
ergodic.  The cocycle superrigidity theorem says that if $\G$ is a lattice in a simply connected, semi-simple Lie
group $G$ all of whose simple factors are noncompact and have property $(T)$
then there is a measurable map $s:M{\rightarrow}H$, a representation
$\pi:G{\rightarrow}H$, a compact subgroup $K<H$ which commutes with
$\pi(G)$ and a measurable map $\G{\times}M{\rightarrow}K$ such that
\begin{equation}
\label{equation:csr} \g{\cdot}s(m)=k(m,\g)\pi(\g)s(\g{\cdot}m).
\end{equation}
\noindent It is easy to check from this equation that the map $K$
satisfies the equation that makes it into a cocycle over the
action of $\G$.  One should view $s$ as providing coordinates on $P$
in which the $\G$ action is {\em almost a product}. For more
discussion of this theorem, particularly in the case where all simple
factors of $G$ have higher rank, the reader should see any of
\cite{Feres-book,Feres-csr,FisherMargulis-lrc,Furstenberg-csr,Zimmer-book}. (The version stated here is only proven in
\cite{FisherMargulis-lrc}, previous proofs all yielded somewhat more complicated
statements that require passing to finite ergodic extensions of the action.)
For the case of $G$ with simple factors of the form $Sp(1,n)$ and $F_4^{-20}$,
the results follows from work of the author and Hitchman \cite{FisherHitchman-IMRN},
building on earlier results of Korevaar-Schoen and Corlette-Zimmer \cite{CorletteZimmer, KorevaarSchoen3,  KorevaarSchoen2, KorevaarSchoen1}.

As a sample application, let $M=\Ta^n$ and let $P$ be
the frame bundle of $M$, i.e. the space of frames in the tangent
bundle of $M$. Since $\Ta^n$ is parallelizable, we have
$P=\Ta^n{\times}GL(n,\Ra)$. The cocycle super-rigidity theorem
then says that ``up to compact noise" the derivative of any measure
preserving $\G$ action on $\Ta^n$ looks measurably like a constant
linear map.  In fact, the cocycle superrigidity theorems apply more
generally to continuous actions on any principal bundle $P$ over $M$
with fiber $H$, an algebraic group, and in this context produces a
measurable section $s:M{\rightarrow}P$ satisfying equation
$(\ref{equation:csr})$. So in fact, cocycle superrigidity implies
that for any action preserving a finite measure on any manifold the
derivative cocycle looks measurably like a constant cocycle, up to
compact noise. That cocycle superrigidity provides information about
actions of groups on manifolds through the derivative cocycle  was
first observed by Furstenberg in \cite{Furstenberg-csr}.  Zimmer originally proved cocycle
superrigidity in order to study orbit equivalence of group actions.
For recent surveys of subsequent developments concerning orbit
equivalence rigidity and other forms of superrigidity for cocycles,
see \cite{Furman-thisvolume, Popa-ICM, Shalom-ECM}.

\subsection{First applications to group actions and the problem of regularity}
\label{subsection:csrfirstapps}

The following result, first observed by Furstenberg, is an immediate consequence
of cocycle superrigidity.

\begin{proposition}
\label{proposition:furstenberg}
Let $G$ be semisimple Lie group with no compact factors and with
property $(T)$ of Kazhdan and let $\G<G$ be a lattice.  Assume
$G$ or $\Gamma$ acts by volume preserving diffeomorphisms on a
compact manifold $M$.  Then there is a linear representation $\pi:G{\rightarrow}GL(\dim(M),\Ra)$
such that the Lyapunov exponents of $g{\in}G$ or $\G$ are exactly the
absolute values of the eigenvalues of $\pi(g)$.
\end{proposition}

This Proposition is most striking to those familiar with classical dynamics,
where the problem of estimating, let alone computing, Lyapunov exponents,
is quite difficult.

Another immediate consequence of cocycle super-rigidity is the following:

\begin{proposition}
\label{proposition:measmetric}
Let $G$ or $\G$ be as above, acting by smooth diffeomorphisms
on a compact manifold $M$ preserving volume.  If every element of $G$ or $\G$ acts
with zero entropy, then there is  a measurable invariant
metric on $M$.  In particular, if $G$ admits no non-trivial
representations of $\dim(M)$, then there is a measurable
invariant metric on $M$.
\end{proposition}

As mentioned above in subsection \ref{section:lowdimensional}, this
reduces Conjecture \ref{conjecture:lowdimvolume} to the following
technical conjecture.

\begin{conjecture}
\label{conjecture:metricissmooth}
Let $G$ or $\G$ acting on $M$ be as above.  If $G$ or $\G$ preserves
a measurable Riemannian metric, then they preserve a smooth invariant Riemannian metric.
\end{conjecture}

We recall that  first evidence towards this conjecture is Theorem \ref{theorem:rgslowdim}.
We remark that in the proof of that theorem, it is not the case that the measurable
metric from Proposition \ref{proposition:measmetric} is shown to be smooth, but instead
it is shown that the image of $\G$ in $\Diff(M)$ lies in a compact subgroup. In other
work, Zimmer explicitly improves regularity of the invariant metric, see particularly \cite{Zimmer-distal}.

In general the problem that arises immediately from cocycle superrigidity in any context
is understanding the regularity of the straightening section $\sigma$.  This question
has been studied from many points of view, but still relatively little is known in general.
For certain examples of actions of higher rank lattices on compact manifolds, discussed below
in section \ref{section:exoticactions}, the $\sigma$ that straightens the derivative cocycle
cannot be made smooth on all of $M$.  It is possible that for volume preserving actions on manifolds and the
derivative cocycle, $\sigma$ can always be chosen to be smooth on a dense open set of full measure.

We will return to theme of regularity of the straightening section in Section \ref{section:hyperbolic}.
First we turn to more geometric contexts in which the output of cocycle superrigidity is also often used but
usually more
indirectly.

\section{Geometric actions and rigid geometric structures}
\label{section:rgs}

In this section, we discuss the role of rigid geometric structures in the study
of actions of large groups.  The notion of rigid geometric structure was introduced
by Gromov, partially in reaction to Zimmer's work on large group actions.

The first subsection of this section recalls the definition of rigid geometric structure,
gives some examples and explains the relation of Gromov's rigid geometric structures to
other notions introduced by Cartan.  Subsection \ref{subsection:rigidreps} recalls Gromov's
initial results relating actions of simple groups preserving rigid geometric structures on $M$ to representations
of the fundamental group of $M$, and extensions of these results to lattice actions due to Zimmer
and the author.  The third section concerns a different topic, namely the rigidity of connection
preserving actions of a lattice $\G$, particularly on manifolds of dimension not much larger than
$d(G)$ as defined in section \ref{section:lowdimensional}.  The fourth subsection recalls some
obstructions to actions preserving geometric structures, particularly a result known as Zimmer's
geometric Borel density theorem and some recent related results of Bader, Frances and Melnick.
We discuss actions preserving a complex structure in subsection \ref{subsection:complex} and
end with a subsection on questions concerning geometric actions.

\subsection{Rigid structures and structures of finite type}
\label{subsection:rgsdefs}

In this subsection we recall the formal definition of a rigid
geometric structure.  Since the definition is somewhat technical,
some readers may prefer to skip it, and read on keeping in mind
examples rather than the general notion.  The basic examples of
rigid geometric structures are Riemannian metrics, pseudo-Riemannian
metrics, conformal structures defined by either type of metric and
affine or projective connections.  Basic examples of geometric structures
which are not rigid are a volume form or symplectic structure.  An
intermediate type of structure which exhibits some rigidity but which is not
literally rigid in the sense discussed here is a complex structure.

If $N$ is a manifold, we denote the $k$-th order frame bundle of
$N$ by $F^k(N)$, and by $J^{s,k}(N)$ the bundle of $k$-jets at $0$
of maps from ${\mathbb R}^s$ to $N$. If $N$ and $N^{\prime}$ are
two manifolds, and $f:N \to N^\prime$ is a map between them, then
the $k$-jet $j^k (f)$ induces a map $J^{s,k} N \to J^{s,k} N
^\prime$ for all $s$.  We let $D^k (N)$ be the bundle whose fiber
$D^k _p$ at a point $p$ consists of the set of $k$-jets at $p$ of
germs of diffeomorphisms of $N$ fixing $p$. We abbreviate $D^k _0
({\mathbb R}^n)$ by $D_n^k$ or simply $D^k$; this is a real
algebraic group. For concreteness, one can represent
each element uniquely, in terms of standard coordinates
$(\xi_1,...,\xi_n)$ on ${\mathbb R}^n$, in the form
$$
(P_1(\xi_1,...,\xi_n),...,P_n (\xi_1,..,\xi_n))
$$
where $P_1$, $P_2$,...,$P_n$ are polynomials of degree $\leq k$. We
denote the vector space of such polynomial maps of degree $\leq k$ by ${\mathcal
P}_{n,k}$.


The group $D_n^k$ has a natural action on $F^k (N)$, where $n$ is
the dimension of $N$. Suppose we are given an algebraic action of
$D_n^k$ on a smooth algebraic variety $Z$. Then following Gromov
(\cite {Gromov-RigidStructures}), we make the following definition:
\begin{defn}
\begin{enumerate}
\item An {\it A -structure} on $N$ (of order $k$, of type $Z$) is
a smooth map $\phi : F^k (N) \to Z$ equivariant for the $D^k_n$
actions. \item With notation as above, the {\it $r$-th
prolongation} of $\phi$, denoted $\phi^r$, is the map $\phi^r :
F^{k+r}(N) \to J ^{n,r}(Z)$ defined by $\phi^r = j^r(\phi) \circ
\iota _k ^{r+k}$ where $\iota _k^{r+k} : F^{k+r} (N) \to J^{n,r}
(F^k (N))$ is the natural inclusion and $j^k(h) : J^{n,r} (F^k
(N)) \to J^{n,r}(Z)$ is as before; this is an A-structure of type
$J^{n,r}(Z)$ and order $k+r$.
\end{enumerate}
\end{defn}
Equivalently, an A-structure of type $Z$ and order $k$ is a smooth
section of the associated bundle $F^k(N) \times _{D^k}Z$ over N.
Note that an A-structure on $N$ defines by restriction an
A-structure $\phi|_U$ on any open set $U \subset N$.
\begin{remark}
A-structures were introduced in \cite {Gromov-RigidStructures}; a good introduction to
the subject, with many examples, can be found in \cite {Benoist}.  A
comprehensive and accessible discussion the results of \cite{Gromov-RigidStructures} concerning actions
of simple Lie groups can be found in \cite{Feres-Gromov}.
\end{remark}

\noindent Note that if $N$ and $N^\prime$ are $n$-manifolds, and
$h: N \to N^{\prime}$ is a diffeomorphism, then $h$ induces a
bundle map $j^k(h): F^k(N) \to F^k (N^\prime)$.

\begin{defn}
\begin{enumerate}
\item If $\phi: F^k (N) \to Z$, $\phi^\prime : F^k (N^\prime) \to
Z$ are A-structures, a diffeomorphism $h: N \to N^\prime$ is an
{\it isometry} from $\phi$ to $\phi^\prime$ if $\phi ^\prime \circ
j^k(h) = \phi$. \item A {\it local isometry} of $\phi$ is a
diffeomorphism $h: U_1 \to U_2$, for open sets $U_1, U_2 \subset
N$, which is an isometry from $\phi|_{U_1}$ to $\phi|_{U_2}$.
\end{enumerate}
\end{defn}

\noindent For $p \in M$ denote by $Is^{loc} _p (\phi)$ the
pseudogroup of local isometries of $\phi$ fixing $p$, and, for $l
\geq k$, we denote by $Is ^l _p (\phi)$ the set of elements $j^l_p
(h) \in D^l _p$ such that $j_p ^l (\phi \circ j^k _p(h)) =
\phi^{l-k}$, where both sides are considered as maps $F^{k+l}(N)
\to J^{l-k}(Z)$. Note that $Is _p ^l (\phi)$ is a group, and there is a
natural homomorphism $r_p ^{l;m}: Is _p ^l (\phi) \to Is_p^m
(\phi)$ for $m < l$; in general, it is neither injective nor
surjective.

\begin{defn}
The structure $\phi$ is called $k$-rigid if for every point $p$, the
map $r_p ^{k+1;k}$ is injective.
\end{defn}

 A first remark worth making is that an affine connection is a rigid geometric structure and that any generalized
quasi-affine action on a manifold of the form $K \backslash H/ \Lambda {\times} M$
with $\Lambda$ discrete preserves an affine connection and therefore also a torsion
free affine connection.
In order to provide some examples, we recall the following lemma of Gromov.

\begin{lemma}
\label{lemma:algebraicrigid}
Let $V$ be an algebraic variety and $G$ a group acting algebraically on $V$.
For every $k$, there is a tautological $G$ invariant geometric structure
of order $k$ on $V$, given by  $\omega:P^k(V){\rightarrow}P^k(V)/G$.
This structure is rigid if and only if the action of
$G$ on $P^k(V)$ is free and proper.
\end{lemma}

\noindent The conclusion in the first sentence is obvious.  The
second sentence is proven in section $0.4$, pages 69-70, of
\cite{Gromov-RigidStructures}.

\noindent {\bf Examples}

\noindent
\begin{enumerate}

\item The action of $G=SL_n(\mathbb R)$ on ${\mathbb R}^n$ is algebraic.
So is the action of $G$ on the manifold $N_1$ obtained by blowing up the
origin.  The reader can easily verify that the action of $G$ on $P^2(N_1)$
is free and proper.

\item We can compactify $N_1$ by $N_2$ by viewing the complement of the blow up as
a subset of the projective space $P^n$.  Another description of the same
action, which may make the rigid structure more visible to the naked eye, is
as follows.  $SL_{n+1}(\mathbb R)$ acts on $P^n$. Let $G$ be
$SL_n(\mathbb R)<SL_{n+1}(\mathbb R)$ as block diagonal matrices with
blocks of size $n$ and $1$ and $1{\times}1$ block equal to $1$.  Then $G$
acts on $P^n$ fixing a point $p$.  We can obtain $N_2$ by blowing up the fixed point $p$.
The $G$ actions on both $P^n$ and $N_2$ are algebraic and again the reader
can verify that the action is free and proper on $P^2(N_2)$.

\item In the construction from $2$ above, there is an action of a
group $H$ where $H=SL_n(\mathbb R){\ltimes}{\mathbb R}^n$ and
$G=SL_n(\mathbb R)<SL_n(\mathbb R){\ltimes}{\mathbb
R}^n<SL_{n+1}(\mathbb R)$. The $H$ action fixes the point $p$ and
so also acts on $N_2$ algebraically.  However, over the
exceptional divisor, the action is never free on any frame bundle,
since the subgroup ${\mathbb R}^n$ acts trivially to all orders at
the exceptional divisor.

\end{enumerate}

The behavior in example $3$ above illustrates the fact that existence
of invariant rigid structures is more complicated for algebraic groups which are
not semisimple, see the discussion in section $0.4.C.$ of \cite{Gromov-RigidStructures}.

\begin{defn}
\label{definition:finitetype}
A rigid geometric structure is called a {\em finite type} geometric structure
if $V$ is a homogeneous $D_n^k$ space.
\end{defn}

This is by no means the original definition which is due to Cartan
and predates Gromov's notion of a rigid geometric structure by
several decades. It is equivalent to Cartan's definition of a finite
type structure by work of Candel and Quiroga
\cite{CandelQuiroga-one, CandelQuiroga-two, CandelQuiroga-three}. Candel
and Quiroga also give a development of rigid geometric structures
that more closely parallels the older notion of a structure of finite type as
presented in e.g. \cite{Kobayashi-book}.

All standard examples of rigid geometric structures are structures of finite type.
However, example $(2)$ above is not a structure of finite type.  The notion
of structure of finite type was first given in terms of prolongations of Lie
algebras and so yields a criterion that is, in principle, computable.  The work
of Candel and Quiroga extends this computable nature to general rigid geometric
structures.

\subsection{Rigid structures and representations of fundamental groups}
\label{subsection:rigidreps}

In this subsection, we restrict our attention to actions of simple Lie groups $G$
and lattices $\G<G$.  Many of the results of this section extend to semisimple
$G$, though the formulations become more complicated.

A major impetus for Gromov's introduction of rigid geometric structures is the
following theorem from \cite{Gromov-RigidStructures}.

\begin{theorem}
\label{theorem:rigid}
Let $G$ be a simple noncompact Lie group and $M$ a compact real analytic
manifold.  Assume $G$ acts on $M$ preserving an analytic rigid geometric structure
$\omega$ and a volume form $\nu$. Further assume the action is ergodic. Then there
is a linear representation $\rho:\pi_1(M){\rightarrow}GL(n,\Ra)$ such that the Zariski
closure of $\rho(\pi_1(M))$ contains a group locally isomorphic to $G$.
\end{theorem}

We remark that, by the Tits' alternative, the theorem immediately implies that actions of the type described cannot
occur on manifolds with amenable fundamental group or even on manifolds whose fundamental groups
do not contain a free group on two generators \cite{Tits}.  Expository accounts of the proof can
be found in \cite{Feres-Gromov,Zimmer-Gromov, Zimmer-CBMS}.  The representation $\rho$
is actually defined on Killing fields of the lift $\tilde \omega$ of $\omega$ to $\tilde M$.

This is a worthwhile moment to indicate the weakness of Lemma
\ref{lemma:induction}. If we start with a cocompact lattice $\G<G$
and an action of $\G$ on a compact manifold $M$ satisfying the
assumptions of Theorem \ref{theorem:rigid}, we can induce the action
to a $G$ action on the compact manifold $N=(G \times M)/\G$.
However, in this setting, there is an obvious representation of
$\pi_1(N)$ satisfying the conclusion of Theorem \ref{theorem:rigid}.
This is because there is a surjection $\pi_1(N){\rightarrow}\G$.  In
fact for most cocompact lattices $\G<G$ there are homomorphisms
$\sigma:\G{\rightarrow}K$ where $K$ is compact and simply connected
so that the we can have $\G$ act on $K$ by left translation and
obtain examples where $\pi_1(N)=\G$. For $G$ of higher rank, the
following theorem of Zimmer and the author shows that this is the
only obstruction to a variant of Theorem \ref{theorem:rigid} for
lattices.

\begin{theorem}
\label{theorem:pireplattice}
Let $\G<G$ be a lattice, where $G$ is a simple group and $\Ra-\text{rank}(G)\ge 2$. Suppose $\G$ acts analytically and ergodically on a compact manifold $M$ preserving a unimodular rigid geometric structure. Then either

\begin{enumerate}

\item the action is isometric and $M=K/C$ where $K$ is a compact Lie group and the action is by right translation via $\rho:\G\rightarrow K$, a dense image homomorphism, or

\item there exists an infinite image linear representation $\sigma:\pi_1(M)\to GL_n\Ra$, such that the algebraic automorphism group of the Zariski closure of $\sigma\bigl(\pi_1(M)\bigr)$ contains a group locally isomorphic to $G$.

\end{enumerate}
\end{theorem}

The proof of this theorem makes fundamental use of a notion that does not occur in it's statement: the entropy
of the action.  A key observation is the following formula for the entropy of the restriction of the induced
action to $\G$:

$$h_{(G{\times}M)/\G}(\g)=h_{G/\Gamma}(\g)+h_M(\g)$$

We then use the fact that the last term on the right hand side is
zero if and only if the action preserves a measurable Riemannian
metric as already explained in Proposition
\ref{proposition:furstenberg}. In this setting, it then follows from
Theorem \ref{theorem:rgslowdim} that the action is
isometric and the other conclusions in $(1)$ are simple consequences
of ergodicity of the action.

When $h_M(\g)>0$ for some $\G$, we use relations between the entropy of the action
and the Gromov representation discovered by Zimmer \cite{Zimmer-entropyquotient}.  We discuss some of these
ideas below in Section \ref{section:arithmeticquotients}.

\subsection{Affine actions close to the critical dimension}
\label{subsection:affine}

In this section we briefly describe a direction pursued by Feres, Goetze, Zeghib and Zimmer that
classify connection preserving actions of lattices in dimensions close to,
but not less than, the critical dimension $d$.   A representative result is the
following:

\begin{theorem}
\label{theorem:affine}
Let $n>2$, $G=SL(n,\Ra)$ and $\G<G$ a lattice.  Let
$M$ be a compact manifold with $\dim(M)\leq n+1 =d(G)+1$.
Assume that $\G$ acts on $M$ preserving a volume form and an
affine connection.  Then either
\begin{enumerate}
\item $\dim(M)<n$ and the action is isometric
\item $\dim(M)=n$ and either the action is isometric or, upon passing to finite
covers and finite index subgroups, smoothly conjugate to the
standard $SL(n,\Za)$ action on $\Ta^n$
\item $\dim(M)=n+1$ and either the action is isometric or , upon passing to finite
covers and finite index subgroups, smoothly conjugate to the
action of $SL(n,\Za)$ on $\Ta^{n+1}=\Ta^n {\times} \Ta$ where
the action on the second factor is trivial.
\end{enumerate}
\end{theorem}

Variants of this theorem under more restrictive hypotheses were obtained
by Feres, Goetze and Zimmer \cite{Feres-IJMconnection,Goetze-affine, Zimmer-affine}.
The theorem as stated is due to Zeghib \cite{Zeghib-affine}.  The cocycle
superrigidity theorem is used in all of the proofs, mainly to force vanishing
of curvature and torsion tensors of the associated connection.

One would expect similar results for a more general class of acting groups
and also for a wider range of dimensions.  Namely if we let $d_2(G)$ be the
dimension of the second non-trivial representation of $G$, we would expect
a similar result to $(3)$ for any affine, volume preserving action on a
manifold $M$ with $d(G)<d_2(G)$.  With the further assumption that the
connection is Riemannian (and still only for lattices in $SL(n,\Ra)$) this
is proven by Zeghib in \cite{Zeghib-CAG}.

Further related results are contained in other papers of Feres \cite{Feres-affine, Feres-Zimmerconjecture}.
No results of this kind are known for $\dim(M)\geq d_2(G)$. A major difficulty arises as soon as one has $\dim(M) \geq d_2(G)$, namely that more complicated examples
can arise, including affine actions on nilmanifolds and left translation actions on spaces
of the form $G/\Lambda$.  All the results mentioned in this subsection depend on the particular
fact that flat Riemannian manifolds have finite covers which are tori.

\subsection{Zimmer's Borel density theorem and generalizations}
\label{subsection:zbdt}

An obvious first question concerning $G$ actions is the structure of the
stabilizer subgroups.  In this direction we have:

\begin{theorem}
\label{theorem:zimmerbdt}
Let $G$ be a simple Lie group acting essentially faithfully on a compact manifold
$M$ preserving a volume form, then the stabilizer of almost every
point is discrete.
\end{theorem}

As an application of this result, one can prove the following:

\begin{theorem}
\label{theorem:embeddingstructuregroup}
Let $G$ be a simple Lie group acting essentially faithfully on a compact
manifold, preserving a volume form and a homogeneous geometric structure
with structure group $H$.  Then there is an inclusion $\mathfrak g {\rightarrow} \mathfrak h$.
\end{theorem}

In particular, the theorem provides an obstruction to a higher rank simple Lie group
having a volume preserving action on a compact Lorentz manifold.  For more discussion
of these theorems, see \cite{Zimmer-CBMS}.  This observation lead to a major series
of works studying automorphism groups of Lorentz manifolds, see e.g. \cite{Adams-book, Kowalsky-thesis,
Zeghib-lorentz1, Zeghib-lorentz2}. We remark here that Theorem \ref{theorem:embeddingstructuregroup}
does not require that the homogeneous structure be rigid.

More recently some closely related phenomena have been discovered in a joint work
of Bader, Frances and Melnick \cite{BaderFrancesMelnick}.  Their work uses yet another notion of a geometric
structure, that of a {\em Cartan geometry}.  We will not define this notion
rigorously here, it suffices to note that any rigid homogeneous geometric
structure defines a Cartan geometry, as discussed in \cite[Introduction]{BaderFrancesMelnick}.
The converse is not completely clear in general, but most classical examples of rigid geometric
structures can also be realized as Cartan geometries.  A Cartan geometry is essentially a way
of saying that a manifold is {\em infinitesimally} modeled on some homogeneous space $G/P$.
To recapture the notion of a Riemannian connection, $G=O(n){\ltimes}\Ra^n$ and $P=O(n)$, to recapture the
 notion of an affine connection $G=Gl(n,\Ra){\ltimes}\Ra^n$ and $P=Gl(n,\Ra)$.  Cartan connections
 come with naturally defined curvatures which vanish if and only if the manifold is {\em locally} modeled
 on $G/P$.  We refer the reader to the book \cite{Sharpe-book} for a general discussion of Cartan geometries
 and their use in differential geometry.

 Given a connected linear group $L$, we define its real rank, $rk(L)$ as before to be the dimension of a maximal $\Ra$-diagonalizable subgroup of $L$, and let $n(L)$ denote the maximal nilpotence degree of a connected nilpotent subgroup. One of the main results of \cite{BaderFrancesMelnick} is:

 \begin{theorem}
 \label{theorem:baderfrancesmelnickembedding}
 If a group $L$ acts by automorphisms of a compact Cartan geometry modeled on $G/P$, then
\begin{enumerate}
\item  $rk(Ad L^0) \leq rk(Ad_{\mathfrak g} P)$
\item  $n(Ad L^0)\leq n(Ad_{\mathfrak g} P)$
 \end{enumerate}
 \end{theorem}

The main point is that $L$ is not assumed either simple or connected.  This theorem is
deduced from an embedding theorem similar in flavor to Theorem \ref{theorem:zimmerbdt}.

In addition, when $G$ is a simple group and $P$ is a maximal parabolic subgroup, Bader, Frances
and Melnick prove rigidity results classifying all possible actions when the rank bound in Theorem
\ref{theorem:baderfrancesmelnickembedding} is achieved.  This classification essentially says that
all examples are algebraic, see \cite{BaderFrancesMelnick} for detailed discussion.

An earlier paper by Feres and Lampe also explored applications of Cartan geometries to rigidity
and dynamical conditions for flatness of Cartan geometries  \cite{FeresLampe}.

\subsection{Actions preserving a complex structure}
\label{subsection:complex}

We mention here one recent result by Cantat and a few active related directions
of research.  In \cite{Zimmer-Mostow}, Zimmer asked whether one had restrictions
on low dimensional holomorphic actions of lattices.  The answer was obtained in
\cite{Cantat-zimmer} and is:

\begin{theorem}
\label{theorem:cantat}
Let $G$ be a connected simple Lie group of real rank at least $2$ and suppose $\G<G$ is a lattice.
If $\G$ admits an action by automorphisms on a compact K\"ahler manifold
$M$, then the rank of $G$ is at least the complex dimension of $M$.
\end{theorem}

The proof of the theorem depends primarily on results concerning $\Aut(M)$, particularly
recent results on holomorphic actions of abelian groups by Dinh and Sibony \cite{DinhSibony}.
It is worth noting that while the theorem does depend on Margulis' superrigidity theorem, it
does not depend on Zimmer's cocycle superrigidity theorem.

More recently Cantat, Deserti and others have begun a program of studying large subgroups
of automorphism groups of complex manifolds, see e.g. \cite{Cantat-birat, Deserti}. In
particular, Cantat has proven an analogue of conjecture \ref{conjecture:surfaces} in the
context of birational actions on complex surfaces.

\begin{theorem}
\label{theorem:cantatsurface} Let S be a compact K\"{a}hler surface and G an infinite, countable group
of birational transformations of S. If G has property $(T)$, then there is a birational map $j : S \rightarrow \Pa^2(\Ca)$ which conjugates
$G$ to a subgroup of $\Aut (\Pa^2(\Ca))$.
\end{theorem}

\subsection{Conjectures and questions}
\label{subsection:conjectures}

A major open problem, generalizing Conjecture \ref{conjecture:rgsslactions} is the following:

\begin{conjecture}[Gromov-Zimmer]
\label{conjecture:gromovzimmer}
Let $D$ be a semisimple Lie group with all factors having property $(T)$ or a lattice in such a Lie group.
Then any $D$ action on a compact manifold preserving a rigid geometric structure and a volume form is generalized quasi-affine.
\end{conjecture}

A word is required on the attribution. Both Gromov and Zimmer made various less precise conjectures
concerning the classification of actions as in the conjecture \cite{Gromov-RigidStructures,Zimmer-ICM}.  The exact statement
of the correct conjecture was muddy for several years while it was not known if the Katok-Lewis and Benveniste
type examples admitted invariant rigid geometric structures. See Section \ref{section:exoticactions} for more discussion of these examples.  In the context of the results of \cite{BenvenisteFisher}, where Benveniste and the author prove that those actions do not preserve rigdi geometric structures, this version of the classification seems quite plausible. It is perhaps more plausible if one assumes
the action is ergodic or that the geometric structure is homogeneous.

A possibly easier question that is relevant is the following:

\begin{question}
\label{question:homogeneity}
Let $M$ be a compact manifold equipped with a homogeneous rigid geometric structure $\omega$.  Assume $\Aut(M,\omega)$
is ergodic or has a dense orbit.  Is $M$ locally homogeneous?
\end{question}

Gromov's theorem on the open-dense implies that a dense open set in $M$ is locally homogeneous
even if $\omega$ is not homogeneous.  The question is whether this homogeneous structure extends
to all of $M$ if $\omega$ is homogeneous.  If $\omega$ is not homogeneous, the examples following
Lemma \ref{lemma:algebraicrigid} show that $M$ need not be homogeneous.

It seems possible to approach the case of Conjecture \ref{conjecture:gromovzimmer} where $\omega$ is homogeneous
and $D$ acts ergodically by answering Question \ref{question:homogeneity} positively.  At this point,
one is left with the problem of classifying actions of higher rank groups and lattices on locally homogeneous
manifolds. Induction easily reduces one to considering $G$ actions. The techniques Gromov uses to produce the locally homogeneous structure gives slightly more precise information in this setting: one is left trying to classify homogeneous manifolds modeled on $H/L$ where $G$ acts via an inclusion in $H$ centralizing $L$.  In fact
this reduces one
to problems about locally homogeneous manifolds studied by Zimmer and collaborators, see particularly
\cite{Labourie-homogeneoussurvey, LabourieMozesZimmer, LabourieZimmer, Zimmer-homogeneous}.  For a more general survey of locally homogeneous
spaces, see also \cite{Kobayashi-survey} and \cite{KobayashiYoshino}.




The following question concerns a possible connection between preserving a rigid
geometric structure and having uniformly partially hyperbolic dynamics.

\begin{question}
\label{question:dynamical}
Let $D$ be any group acting on a compact manifold $M$ preserving a volume form and
a rigid geometric structure.  Is it true that if some element $d$ of $D$ has positive
Lyapunov exponents, then $d$ is uniformly partially hyperbolic?
\end{question}

This question is of particular interest for us when $D$ is a semisimple Lie group or a lattice,
but would be interesting to resolve in general.  The main reason to believe that the answer might
be yes is that the action of $D$ on the space of frames of $M$ is proper.  Having a proper
action on tangent bundle minus the zero section implies that an action is Anosov,
see Ma\~{n}\'{e}'s article for a proof \cite{Mane}.

\section{Topological super-rigidity: regularity results from hyperbolic dynamics and
applications} \label{section:hyperbolic}

A major area of research in the Zimmer program has been the
application of hyperbolic dynamics.  This area might be described
by the maxim:  in the presence of hyperbolic dynamics, the straightening
section from cocycle superrigidity is often more regular.  Some major successes in this direction
are the work of Goetze-Spatzier and Feres-Labourie. In the context of hyperbolic dynamical approaches, there are two main settings.
In the first one considers actions of higher rank lattices on a particular
class of compact manifolds and a second in which one makes no assumption on
the topology of the manifold acted upon.

In this section we
discuss a few such results, after first recalling some facts about the stability of
hyperbolic dynamical systems in subsection \ref{subsection:hyperbolic}.  We then discuss
best known rigidity results for actions on tori in subsection \ref{subsection:hyperbolicontori}
and after this discuss work of Goetze-Spatzier and Feres-Labourie in the more general
context in subsection \ref{subsection:topsuperrigid}.

The term ``topological superrigidity" for this area of research was coined by
Zimmer, whose early unpublished notes on the topic dramatically influenced research
in the area \cite{Zimmer-tcsr}.

This aspect of the Zimmer program has also given rise to a study of rigidity properties
for other uniformly hyperbolic actions of large groups, most particularly higher
rank abelian groups.  See e.g \cite{KalininSpatzier, KatokSpatzier, RodriguezHertz}.

The following subsection recalls basic notions from hyperbolic dynamics that are needed
in this section.

\subsection{Stability in hyperbolic dynamics.} \label{subsection:hyperbolic}

A diffeomorphism $f$ of a manifold $X$ is said to be {\em Anosov}
if there exists a continuous $f$ invariant splitting of the
tangent bundle $TX=E_{f}^u{\oplus}E_{f}^s$ and constants $a>1$ and
$C,C'>0$ such that for every $x{\in}X$,
\begin{enumerate}
\item $\|Df^n(v^u)\|{\geq}Ca^n\|v^u\|$ for all
$v^u{\in}E_{f}^u(x)$ and,

\item $\|Df^n(v^s)\|{\leq}C'a^{-n}\|v^s\|$ for all
$v^s{\in}E_{f}^s(x)$.
\end{enumerate}

\noindent We note that the constants $C$ and $C'$ depend on the
choice of metric, and that a metric can always be chosen so that
$C=C'=1$.  There is an analogous notion for a flow $f_t$, where
$TX=T\mathcal{O}{\oplus}E_{f_t}^u{\oplus}E_{f_t}^s$ where
$T{\mathcal{O}}$ is the tangent space to the flow direction and
vectors in $E_{f_t}^u$ (resp. $E_{f_t}^s$) are uniformly expanded
(resp. uniformly contracted) by the flow.  This notion was
introduced by Anosov and named after Anosov by Smale, who
popularized the notion in the United States \cite{Anosov,Smale}. One of
the earliest results in the subject is Anosov's proof that Anosov
diffeomorphisms are {\em structurally stable}, i.e that any $C^1$ perturbation
of an Anosov diffeomorphism is conjugate back to the original diffeomorphism by
a homeomorphism.  There is an analogous result for
flows, though this requires that one introduce a notion of time
change that we will not consider here.   Since Anosov also showed
that $C^2$ volume preserving Anosov flows and diffeomorphisms are ergodic,
structural stability implies that the existence of an open set of
``chaotic" dynamical systems.

The notion of an Anosov diffeomorphism has had many interesting
generalizations, for example: Axiom A diffeomorphisms,
non-uniformly hyperbolic diffeomorphisms, and diffeomorphisms
admitting a dominated splitting.  A notion that has been particularly
useful in the study of rigidity of group actions is the notion of a partially
hyperbolic diffeomorphism as introduced by Hirsch, Pugh and Shub.
Under strong enough hypotheses, these diffeomorphisms have a
weaker stability property similar to structural stability. More or
less, the diffeomorphisms are hyperbolic relative to some
foliation, and any nearby action is hyperbolic to some nearby
foliation. To describe more precisely the class of diffeomorphisms
we consider and the stability property they enjoy, we require some
definitions.

The use of the word {\em foliation} varies with context. Here a
{\em foliation by $C^k$ leaves} will be a continuous foliation
whose leaves are $C^k$ injectively immersed submanifolds that vary
continuously in the $C^k$ topology in the transverse direction. To
specify transverse regularity we will say that a foliation is
transversely $C^r$.  A foliation by $C^k$ leaves which is
tranversely $C^k$ is called simply a $C^k$ foliation. (Note our
language does not agree with that in the reference \cite{HirschPughShub}.)

Given an automorphism $f$ of a vector bundle $E{\rightarrow}X$ and
constants $a>b{\geq}1$, we say $f$ is {\em $(a,b)$-partially
hyperbolic} or simply {\em partially hyperbolic} if there is a
metric on $E$, a constant and $C{\geq}1$ and a continuous $f$
invariant, non-trivial splitting
$E=E_{f}^u{\oplus}E_{f}^c{\oplus}E_{f}^s$
 such that for every $x$ in $X$:
\begin{enumerate}

\item $\|f^n(v^u)\|{\geq}Ca^n\|v^u\|$ for all
$v^u{\in}E_{f}^u(x)$,

\item $\|f^n(v^s)\|{\leq}C{\inv}a^{-n}\|v^s\|$ for all
$v^s{\in}E_{f}^s(x)$ and

\item  $C{\inv}b^{-n}\|v^0\|<\|f^n(v^0)\|{\leq}C{b^n}\|v^0\|$ for
all $v^0{\in}E_{f}^c(x)$ and all integers $n$.

\end{enumerate}

\noindent A $C^1$ diffeomorphism $f$ of a manifold $X$ is {\em
$(a,b)$-partially hyperbolic} if the derivative action $Df$ is
$(a,b)$-partially hyperbolic on $TX$. We remark that for any
partially hyperbolic diffeomorphism, there always exists an {\it
adapted metric} for which $C=1$. Note that $E_{f}^c$ is called the
{\em central distribution} of $f$, $E_{f}^u$ is called the {\em
unstable distribution} of $f$ and $E_{f}^s$ the {\em stable
distribution} of $f$.

Integrability of various distributions for partially hyperbolic
dynamical systems is the subject of much research.  The stable and
unstable distributions are always tangent to invariant foliations
which we call the stable and unstable foliations and denote by
$\sw_f^s$ and $\sw_f^u$.  If the central distribution is tangent
to an $f$ invariant foliation, we call that foliation a {\em
central foliation} and denote it by $\sw^c_f$. If there is a
unique foliation tangent to the central distribution we call the
central distribution {\em uniquely integrable}. For smooth
distributions unique integrability is a consequence of
integrability, but the central distribution is usually not smooth.
If the central distribution of an $(a,b)$-partially hyperbolic
diffeomorphism $f$ is tangent to an invariant foliation $\sw^c_f$,
then we say $f$ is {\em $r$-normally hyperbolic to $\sw^c_f$} for
any $r$ such that $a>b^r$.  This is a special case of the
definition of $r$-normally hyperbolic given in \cite{HirschPughShub}.


Before stating a version of one of the main results of \cite{HirschPughShub},
we need one more definition.  Given a group $G$, a manifold $X$,
two foliations $\ff$ and $\ff'$ of $X$, and two actions $\rho$ and
$\rho'$ of $G$ on $X$, such that $\rho$ preserves $\ff$ and
$\rho'$ preserves $\ff'$, following \cite{HirschPughShub} we call $\rho$ and
$\rho'$ {\em leaf conjugate} if there is a homeomorphism $h$ of
$X$ such that:
\begin{enumerate}
\item $h(\ff)=\ff'$ and \item for every leaf $\fL$ of $\ff$ and
every $g{\in}G$, we have $h(\rho(g)\fL)=\rho'(g)h(\fL)$.
\end{enumerate}

\indent\indent
The map $h$ is then referred to as a {\em leaf conjugacy} between
$(X,\ff,\rho)$ and $(X,\ff',\rho')$.  This essentially means that
the actions are conjugate modulo the central foliations.

We state a special case of some the results of Hirsch-Pugh-Shub on
perturbations of partially hyperbolic actions of $\mathbb Z$, see
\cite{HirschPughShub}.  There are also analogous definitions and results for
flows.  As these are less important in the study of
rigidity, we do not discuss them here.

\begin{theorem}
\label{theorem:hps} Let $f$ be an  $(a,b)$-partially hyperbolic
$C^k$ diffeomorphism of a compact manifold $M$ which is
$k$-normally hyperbolic to a $C^k$ central foliation $\sw^c_f$.
Then for any $\delta>0$, if $f'$ is a $C^k$ diffeomorphism of $M$
which is sufficiently $C^1$ close to $f$ we have the following:
\begin{enumerate}
\item $f'$ is $(a',b')$-partially hyperbolic, where
$|a-a'|<\delta$ and $|b-b'|<\delta$,  and  the splitting
$TM=E_{f'}^u{\oplus}E_{f'}^c{\oplus}E_{f'}^s$ for $f'$ is $C^0$
close to the splitting for $f$;

\item there exist $f'$ invariant foliations by $C^k$ leaves
$\sw^c_{f'}$ tangent to $E^c_{f'}$, which is close in the natural
topology on foliations by $C^k$ leaves to $\sw^c_f$,

\item there exists a (non-unique) homeomorphism $h$ of $M$ with
$h(\sw^c_f)=\sw^c_{f'}$, and $h$ is $C^k$ along leaves of
$\sw^c_f$, furthermore $h$ can be chosen to be $C^0$ small and
$C^k$ small along leaves of $\sw^c_{f}$

\item the homeomorphism $h$ is a leaf conjugacy between the
actions $(M,\sw^c_f,f)$ and $(M,\sw^c_{f'},f')$.
\end{enumerate}
\end{theorem}


\noindent Conclusion $(1)$ is easy and probably older than
\cite{HirschPughShub}. One motivation for Theorem \ref{theorem:hps} is to
study stability of dynamical properties of partially hyperbolic
diffeomorphisms. See the survey, \cite{BurnsPughShubWilkinson}, by Burns, Pugh, Shub and Wilkinson for more discussion of that and related issues.

\subsection{Uniformly hyperbolic actions on tori}
\label{subsection:hyperbolicontori}

Many works have been written considering local rigidity of actions
with some affine, quasi-affine and generalized quasi-affine actions with
hyperbolic behavior.  For a discussion of this, we refer to
\cite{Fisher-survey}.  Here we only discuss results which prove
some sort of global rigidity of groups acting on manifolds.  The first
such results were contained in papers of Katok-Lewis and Katok-Lewis-Zimmer
\cite{KatokLewis-global,KatokLewisZimmer}. As these are now special cases
of later more general results, we do not discuss them in detail here.

In this section we discuss results which only provide continuous
conjugacies to standard actions.  This is primarily because these
results are less technical and easier to state.  In this context, one
can improve regularity of the conjugacy given certain technical dynamical
hypotheses on certain dynamical foliations.

\begin{defn}
\label{definition:weaklyhyperbolic}
An action of a group $\G$ on a manifold $M$ is weakly hyperbolic if there exist
elements $\g_1, \ldots, \g_k$ each of which is partially hyperbolic such that the
sum of the stable sub-bundles of the $\g_i$ spans the tangent bundle to $M$ at every
point, i.e. $\sum_i E_{\g_i} = TM$.
\end{defn}

To discuss the relevant results we need a related topological notion that captures hyperbolicity
at the level of fundamental group.  This was introduced in \cite{FisherWhyte-quotient}
by Whyte and the author.
If a group $\G$ acts on a manifold with torsion free nilpotent fundamental group and the action lifts
to the universal cover, then the action of $\G$ on $\pi_1(M)$ gives rise to an action of $\G$ on the
Malcev  completion $N$ of $\pi_1(M)$ which is a nilpotent Lie group.  This yields a representation of $\G$ on
the Lie algebra $\mathfrak n$.

\begin{defn}
\label{definition:pi1hyperbolic}
We say an action of $\G$ on a manifold $M$ with nilpotent fundamental group is $\pi_1$-hyperbolic
if for the resulting $\G$ representation on $\mathfrak n$, we have finitely many elements
$\g_1, \ldots, \g_k$ such that the sum of their eigenspaces with eigenvalue of modulus less
than one is all of $\mathfrak n$.
\end{defn}

One can make this definition more general by considering $M$ where $\pi_1(M)$ has a $\G$ equivariant nilpotent
quotient, see \cite{FisherWhyte-quotient}. We now discuss results that follow
by combining work of Margulis-Qian with later work of Schmidt and Fisher-Hitchman
\cite{FisherHitchman-csr, MargulisQian, Schmidt-thesis}.

\begin{theorem}
\label{theorem:pi1hyperbolicnil}
Let $M=N/\Lambda$ be a compact nilmanifold and let $\Gamma$ be a lattice
in a semisimple Lie group with property $(T)$.  Assume $\Gamma$ acts on $M$ such that
the action lifts to the universal cover and is $\pi_1(M)$ hyperbolic, then the
action is continuously semi-conjugate to an affine action.
\end{theorem}

This theorem was proven by Margulis and Qian, who noted that if the $\Gamma$ action
contained an Anosov element, then the conjugacy could be taken to be a homeomorphism.
In \cite{FisherWhyte-quotient}, the author and Whyte point out that this theorem extends
easily to the case of any compact manifold with torsion free nilpotent fundamental group.
In \cite{FisherWhyte-quotient}, we also discuss extensions to manifolds with fundamental
group with a quotient which is nilpotent.

Margulis and Qian asked whether the assumption of $\pi_1$ hyperbolicity could be replaced
by the assumption that the action on $M$ was weakly hyperbolic.  In the case of actions
of Kazhdan groups on tori, Schmidt proved that weak hyperbolicity implies $\pi_1$ hyperbolicity,
yielding:

\begin{theorem}
\label{theorem:weaklyhyperbolictorus}
Let $M=\Ta^n$ be a compact torus and let $\Gamma$ be a lattice
in a semisimple Lie group with property $(T)$.  Any weakly hyperbolic
$\Gamma$ action $M$ and that lifts to the universal cover
is continuously semi-conjugate to an affine action.
\end{theorem}

\noindent{\bf Remarks:}

\begin{enumerate}

\item  The contribution of Fisher-Hitchman in both theorems is just in extending cocycle superrigidity
to a wider class of groups, as discussed above.

\item The assumption that the action lifts to the universal cover of $M$ is often
vacuous because of results concerning cohomology of higher rank lattices.  In
particular, it is vacuous for cocompact lattices in simple Lie groups of real
rank at least $3$.

\end{enumerate}

It remains an interesting, open question to take this result and prove that
the semiconjugacy is always a conjugacy and is also always a smooth diffeomorphism.

\subsection{Rigidity results for uniformly hyperbolic actions}
\label{subsection:topsuperrigid}

We begin by discussing some work of Goetze and Spatzier.  To avoid technicalities
we only discuss some of their results.  We begin with the following definition.

\begin{defn}
\label{definition:Cartanaction}
Let $\rho:\Za^k{\times}M{\rightarrow}M$ be an action and $\gamma_1, \ldots, \gamma_l$
be a collection of elements which generate for $\Za^k$.  We call $\rho$ a Cartan action if

\begin{enumerate}
\item each $\rho(\gamma_i)$ is an Anosov diffeomorphism,

\item each $\rho(\gamma_i)$ has one dimensional strongest stable foliation,

\item the strongest stable foliations of the $\rho(\gamma_i)$ are pairwise transverse
and span the tangent space to the manifold.

\end{enumerate}
\end{defn}

It is worth noting that Cartan actions are very special in three ways.  First we assume
that a large number of elements in the acting group are Anosov diffeomorphisms, second
we assume that each of these has one dimensional strongest stable foliation and lastly
we assume these one dimensional directions span the tangent space.  All aspects of these
assumptions are used in the following theorem.  Reproving it even assuming two dimensional
strongest stable foliations would require new ideas.

\begin{theorem}
\label{theorem:goetzespatzier}
Let $G$ be a semisimple Lie group with all simple factors of real rank at least two and
$\Gamma$ in $G$ a lattice.  Then any volume preserving Cartan action of $\Gamma$ is smoothly
conjugate to an affine action on an infranilmanifold.
\end{theorem}

This is slightly different than the statement in Goetze-Spatzier, where they pass to a finite
cover and a finite index subgroup.  It is not too hard to prove this statement from theirs.
The proof spans the two papers \cite{GoetzeSpatzier-Duke} and \cite{GoetzeSpatzier-Annals}.
The first paper \cite{GoetzeSpatzier-Duke} proves, in a somewhat more general context, that
the $\pi$-simple section arising in the cocycle superrigidity theorem is in fact H\"older continuous.
The second paper makes use of the resulting H\"older Riemannian metric in conjunction with ideas
arising in other work of Katok and Spatzier to produce a smooth homogeneous structure on the
manifold.

The work of Feres-Labourie differs from other work on rigidity of actions with hyperbolic properties
in that does not make any assumptions concerning existence of invariant measures.  Here we state
only some consequences of their results, without giving the exact form of cocycle superrigidity that
is their main result.

\begin{theorem}
\label{theorem:fereslabourieactions}
Let $\G$ be a lattice in $SL(n,\Ra)$ for $n\geq 3$ and assume $\G$ acts
smoothly on a compact manifold $M$ of dimension $n$.  Further assume
that for the induced action $N=(G\times M)/\G$ we have
\begin{enumerate}
\item every $\Ra$-semisimple $1$-parameter subgroup of $G$ acts
transitively on $N$ and
\item some element $g$ in $G$ is uniformly partially hyperbolic with
$E_s{\oplus}E_w$ containing the tangent space to $M$ at any point,
\end{enumerate}
then $M$ is a torus and the action on $M$ is a standard affine action.
\end{theorem}

These hypotheses are somewhat technical and essentially ensure that one
can apply the topological version of cocycle superrigidity proven in
\cite{FeresLabourie}.  The proof also uses a deep result of Benoist and Labourie
classifying Anosov diffeomorphisms with smooth stable and unstable foliations
\cite{BenoistLabourie}.

The nature of the hypotheses of Theorem \ref{theorem:fereslabourieactions} make
an earlier remark clear.  Ideally one would only have hypotheses on the $\G$
action, but here we require hypotheses on the induced action instead.  It
is not clear how to reformulate these hypotheses on the induced action as
hypotheses on the original $\G$ action.

Another consequence of the work of Feres and Labourie is a criterion for
promoting invariance of rigid geometric structures.  More precisely they
give a criterion for a $G$ action to preserve a rigid geometric structure
on a dense open subset of a manifold $M$ provided a certain type of subgroup preserves a rigid
geometric structure on $M$.

\subsection{Conjectures and questions on uniformly hyperbolic actions}
\label{subsection:hyperbolicquestions}

We begin with a very general variant of Conjecture \ref{conjecture:hyperbolicclassificationspec}.

\begin{conjecture}
\label{conjecture:hyperbolicclassification}
Let $G$ be a semisimple Lie group all of whose simple factors have property $(T)$, let $\Gamma<G$ be a lattice.  Assume
$G$ or $\Gamma$ acts smoothly on a compact manifold $M$ preserving volume such that some element
$g$ in the acting group is non-trivially uniformly partially hyperbolic.  Then
the action is generalized quasi-affine.
\end{conjecture}

There are several weaker variants on this conjecture, where e.g. one assumes
the action is volume weakly hyperbolic. Even the following much
weaker variant seems difficult:

\begin{conjecture}
\label{conjecture:torusweakhyperbolic}
Let $M=\Ta^n$ and $\G$ as in Conjecture \ref{conjecture:hyperbolicclassification}.
Assume $\G$ acts on $\Ta^n$ weakly hyperbolicly and preserving a smooth measure.
Then the action is affine.
\end{conjecture}

This conjecture amounts to conjecturing that the semiconjugacy in Theorem \ref{theorem:weaklyhyperbolictorus}
is a diffeomorphism.  In the special case where some element of $\G$ is Anosov, the semiconjugacy
is at least a homeomorphism.  If this is true and enough dynamical foliations are one and two dimensional and $\G$
has higher rank, one can then deduce smoothness of the conjugacy from work of Rodriguez-Hertz on
rigidity of actions of abelian groups \cite{RodriguezHertz}.  Work in progress by the author, Kalinin
and Spatzier seems likely to provide a similar result when $\G$ contains many commuting Anosov diffeomorphisms
without any assumptions on dimensions of foliations.

Finally, we recall an intriguing question from \cite{FisherWhyte-quotient} which arises in this context.

\begin{question}
\label{question:fisherwhyte} Let $M$ be a compact manifold with $\pi_1(M)=\Za^n$ and assume $\G<SL(n,\Za)$ has finite
index.  Let $\Gamma$ act on $M$ fixing a point so that the resulting $\G$ action on $\pi_1(M)$ is
given by the standard representation of $SL(n,\Za)$ on $\Za^n$.  Is it true that $\dim(M) \geq n$?
\end{question}

The question is open even if the action on $M$ is assumed to be smooth.  The results of \cite{FisherWhyte-quotient}
imply that there is a continuous map from $M$ to $\Ta^n$ that is equivariant for the standard $\G$ action
on $M$.  Since the image of $M$ is closed and invariant, it is easy to check that it is all of $\Ta^n$.
So the question amounts to one about the existence of equivariant ``space filling curves", where curve
is taken in the generalized sense of continuous map from a lower dimensional manifold.  In another
contexts there are equivariant space filling curves, but they seem quite special.  They arise as
surface group equivariant maps from the circle to $S^2$ and come from three manifolds which fiber over
the circle, see \cite{CannonThurston}.

\section{Representations of fundamental groups and arithmetic quotients}
\label{section:arithmeticquotients}

In this section we discuss some results and questions related to topological approaches
to classifying actions particularly some related to Conjecture
\ref{conjecture:fundamentalgroup}.  In the second subsection, we also discuss related
results and questions concerning maximal generalized affine quotients of actions.

\subsection{Linear images of fundamental groups}
\label{subsection:linearimages}

This section is fundamentally concerned with the question:

\begin{question}
\label{question:fundamentalgroup}
Let $G$ be a semisimple Lie group all of whose factors are not compact
and have property $(T)$.  Assume $G$ acts by homeomorphisms on a manifold
$M$ preserving a measure.  Can we classify linear representations of
$\pi_1(M)$?  Similarly for actions of $\G<G$ a lattice on a manifold $M$.
\end{question}

We remark that in this context, it is possible that $\pi_1(M)$ has
no infinite image linear representations, see discussion in \cite{FisherWhyte-quotient}
and Section \ref{section:exoticactions} below. In all known examples where
 this occurs there is an ``obvious" infinite image linear representation on some finite index subgroup.  Also as first observed
by Zimmer \cite{Zimmer-fundamentalgroup1} for actions of Lie groups, under mild conditions on the
action, representations of the fundamental group become severely restricted.
Further work in this direction was done by Zimmer in conjunction with Spatzier and later
Lubotzky \cite{LubotzkyZimmer-canonical, LubotzkyZimmer-nontrivial,SpatzierZimmer}.
Analogous results for lattices are surprisingly difficult in this context and
constitute the authors dissertation \cite{Fisher-IJM, Fisher-ETDS}.

We recall a definition from \cite{Zimmer-fundamentalgroup1}:

\begin{defn}
\label{definition:engaging}
Let $D$ be a Lie group and assume that $D$ acts on a compact
manifold $M$ preserving a finite measure $\mu$ and ergodically.  We call the action {\em engaging}
if the action of $\tilde D$ on any finite cover of $M$ is ergodic.
\end{defn}

There is a slightly more technical definition of engaging for non-ergodic actions which
says there is no loss of ergodicity on passing to finite covers.  I.e. that the ergodic
decomposition of $\mu$ and its lifts to finite covers are canonically identified by
the covering map.  There are also two variants of this notion {\em totally engaging}
and {\em topologically engaging}, see e.g. \cite{Zimmer-CBMS} for more discussion.

It is worth noting that for ergodic $D$ actions on $M$, the action on any finite
cover has at most finitely many ergodic components, in fact at most the degree
of the cover many ergodic components.  We remark here that any generalized affine
action of a Lie group is engaging if it is ergodic.  The actions constructed below
in Section \ref{section:exoticactions} are not in general engaging.

\begin{theorem}
\label{theorem:lubotzkyzimmer}
Let $G$ be a simple Lie group of real rank at least $2$.  Assume $G$ acts by homeomorphisms
on a compact manifold $M$, preserving a finite measure $\mu$ and engaging.  Assume
$\sigma:\pi_1(M){\rightarrow}GL(n,\Ra)$ is an infinite image linear representation.  Then $\sigma(\pi_1(M))$
contains an arithmetic group $H_{\Za}$ where $H_{\Ra}$ contains a group locally isomorphic
to $G$.
\end{theorem}

For an expository account of the proof of this theorem and a more detailed discussion of
engaging conditions for Lie groups, we refer the reader to \cite{Zimmer-CBMS}.

The extension of this theorem to lattice actions is non-trivial and is in fact the author's
dissertation.  Even the definition of engaging requires modification, since it is not at all
clear that a discrete group action lifts to the universal cover.

\begin{defn}
\label{definition:engagingdiscrete}
Let $D$ be a discrete group and assume that $D$ acts on a compact
manifold $M$ preserving a finite measure $\mu$ and ergodically.  We call the action {\em engaging}
if for every
\begin{enumerate}
\item finite index subgroup $D'$ in $D$ ,

\item finite cover $M'$ of $M$, and

\item lift of the $D'$ action to $M'$,

\end{enumerate}
the action of $D'$ of $M'$ is ergodic.
\end{defn}

The definition does immediately imply that every finite index subgroup $D'$ of $D$ acts
ergodically on $M$.  In \cite{Fisher-ETDS}, a definition is given which does not require
the $D$ action to be ergodic, but even in that context the ergodic decomposition for $D'$
is assumed to be the same as that for $D$.  Definition \ref{definition:engagingdiscrete} is rigged
to guarantee the following lemma:

\begin{lemma}
\label{lemma:inductionengaging}
Let $G$ be a Lie group and $\Gamma<G$ a lattice.  Assume $\Gamma$ acts on a manifold $M$
preserving a finite measure and engaging, then the induced $G$ action on $(G \times M)/\Gamma$
is engaging.
\end{lemma}

We remark that with our definitions here, the lemma only makes sense for $\Gamma$ cocompact,
but this is not an essential difficulty. We can now state a first result for lattice actions.  We let $\Lambda=\pi_1((G \times M)/\Gamma)$.

\begin{theorem}
\label{theorem:fisherthesis}
Let $G$ be a simple Lie group of real rank at least $2$ and $\Gamma<G$ a lattice.  Assume $\Gamma$ acts by homeomorphisms
on a compact manifold $M$, preserving a finite measure $\mu$ and engaging.  Assume
$\sigma:\Lambda{\rightarrow}GL(n,\Ra)$ is a linear representation whose restriction to $\pi_1(M)$ has infinite image.  Then $\sigma(\pi_1(M))$
contains an arithmetic group $H_{\Za}$ where $\Aut(H_{\Ra})$ contains a group locally isomorphic
to $G$.
\end{theorem}

This theorem is proven by inducing actions, applying Theorem \ref{theorem:lubotzkyzimmer}, and
analyzing the resulting output carefully.  One would like to assume $\sigma$ a priori only
defined on $\pi_1(M)$, but there seems no obvious way to extend such a representation to the linear
representation of $\Lambda$ required by Theorem \ref{theorem:lubotzkyzimmer} without a priori information
on $\Lambda$.  This is yet another example of the difficulties in using induction to study lattice
actions.  As a consequence of Theorem \ref{theorem:fisherthesis}, we discover that at least $\sigma(\Lambda)$
splits as a semidirect product of $\Gamma$ and $\sigma(\pi_1(M))$.  But this is not clear a priori and not
clear a posteriori for $\Lambda$.

\subsection{Arithmetic quotients}
\label{subsection:quotient}

In the context of Theorems \ref{theorem:lubotzkyzimmer} and Theorem \ref{theorem:fisherthesis}, one can obtain
much greater dynamical information concerning the relation of $H_{\Za}$ and the dynamics of the action on $M$.
In particular, there is a compact subgroup $C<H_{\mathbb R}$ and a measurable equivariant map
$\phi:M{\rightarrow}C \backslash H_{\Ra} /H_{\Za}$, which we refer to as a {\em measurable arithmetic
quotient}.   The papers  \cite{Fisher-IJM} and \cite{LubotzkyZimmer-canonical} prove that there is always
a canonical maximal quotient of this kind for any action of $G$ or $\Gamma$ on any compact manifold, essentially by using Ratner's theorem to prove that every pair of arithmetic quotients are dominated by a common, larger arithmetic quotient. Earlier results of Zimmer also obtained arithmetic quotients, but only under the assumption that there
 was an infinite image linear representation with discrete image \cite{Zimmer-ratner}. The results we mention from \cite{Fisher-ETDS} and \cite{LubotzkyZimmer-nontrivial} then show that
there are ``lower bounds" on the size of this arithmetic quotient, provided that the action is engaging, in terms
of the linear representations of the fundamental group of the manifold.  In particular, one obtains arithmetic
quotients where $H_{\Ra}$ and $H_\Za$ are essentially determined by $\sigma(\pi_1(M))$.  In fact $\sigma(\pi_1(M))$
contains $H_{\Za}$ and is contained in $H_{\Qa}$.

The book \cite{Zimmer-CBMS} provides a good description of how to produce arithmetic quotients for $G$ actions and the article \cite{Fisher-Cambridge} provides an exposition of the relevant constructions for $\Gamma$ actions.

Under even stronger hypotheses, the papers \cite{FisherZimmer} and \cite{Zimmer-entropyquotient} imply that
the arithmetic quotient related to the ``Gromov representation" discussed above in subsection \ref{subsection:rigidreps} has the same entropy as the original action.  This means that, in a sense, the arithmetic quotient captures
most of the dynamics of the original action.

\subsection{Open questions}

As promised in the introduction, we have the following analogue of Conjecture \ref{conjecture:fundamentalgroup}
for lattice actions.

\begin{conjecture}
\label{conjecture:fundamentalgrouplattice}
Let $\G$ be a lattice in a semisimple group $G$ with all simple factors of real rank at least $2$.  Assume $\Gamma$ acts faithfully, preserving
volume on a compact manifold $M$. Further assume the action is not isometric. Then $\pi_1(M)$ has a finite index subgroup $\Lambda$ such
that $\Lambda$ surjects onto an arithmetic lattice in a Lie group $H$ where $\Aut(H)$ locally contains $G$.
\end{conjecture}

The following questions about arithmetic quotients are natural.

\begin{question}
Let $G$ be a simple Lie group of real rank at least $2$ and $\Gamma<G$ a lattice.  Assume
that $G$ or $\Gamma$ act faithfully on a compact manifold $M$, preserving a smooth volume.
\begin{enumerate}
\item Is there a non-trivial measurable arithmetic quotient?
\item Can we take the quotient map $\phi$ smooth on an open dense set?
\end{enumerate}
\end{question}

Due to a construction in \cite{FisherWhyte-quotient}, one cannot expect that every $M$ admitting a volume
preserving $G$
or $\Gamma$ action has an arithmetic quotient that is even globally continuous or has $\pi_1(M)$ admitting
an infinite image linear representation.  However, the difficulties created in those examples all vanish on passage to a finite cover.

\begin{question}
Let $G$ be a simple Lie group of real rank at least $2$ and $\Gamma<G$ a lattice.  Assume
that $G$ or $\Gamma$ act smoothly on a compact manifold $M$, preserving a smooth volume.
\begin{enumerate}
\item Is there a finite cover of $M'$ of $M$ such that $\pi_1(M')$ admits an infinite image
linear representation?
\item Can we find a finite cover $M'$ of $M$, a lift of the action to $M'$ (on a subgroup of finite index) and an arithmetic quotient where  the quotient map $\phi$ is continuous and smooth on an open dense set?
\end{enumerate}
\end{question}

The examples discussed in the next subsection imply that $\phi$ is at best H\"{o}lder continuous globally.  It is not clear whether the questions and conjectures we have
just discussed are any less reasonable for $SP(1,n)$, $F_4^{-20}$ and their lattices.

\section{Exotic volume preserving actions}
\label{section:exoticactions}

In this subsection, I discuss what is known about what are typically
called ``exotic actions".  These are the only known smooth volume preserving
actions of higher rank lattices and Lie groups which are not generalized
affine algebraic.  These actions make it clear that a clean classification
of volume preserving actions is out of reach. In particular, these actions
have continuous moduli and provide counter-examples to any naive conjectures of the
form ``the moduli space of actions of some lattice $\G$ on any
compact manifold $M$ are countable." In particular, I explain
examples of actions of either $\G$ or $G$ which have large continuous
moduli of deformations as well as manifolds where these moduli have
multiple connected components.

Essentially all of the examples given here derive from the simple
construction of ``blowing up" a point or a closed orbit, which was
introduced to this subject in \cite{KatokLewis-global}.  The further developments
after that result are all further elaborations on one basic
construction.  The idea is to use the ``blow up" construction to
introduce distinguished closed invariant sets which can be varied in
some manner to produce deformations of the action.  The ``blow up"
construction is a classical tool from algebraic geometry which takes
a manifold $N$ and a point $p$ and constructs from it a new manifold
$N'$ by replacing $p$ by the space of directions at $p$. Let
${\Ra}P^l$ be the $l$ dimensional projective space.  To blow up a
point, we take the product of $N{\times}{\Ra}P^{\dim(N)}$ and then
find a submanifold where the projection to $N$ is a diffeomorphism
off of $p$ and the fiber of the projection over $p$ is
${\Ra}P^{\dim(N)}$.   For detailed discussion of this construction
we refer the reader to any reasonable book on algebraic geometry.

The easiest example to consider is to take the action of
$SL(n,\Za)$, or any subgroup  $\G<SL(n,\Za)$ on the torus $\Ta^n$
and blow up the fixed point, in this case the equivalence class of
the origin in $\Ra^n$. Call the resulting manifold $M$. Provided
$\G$ is large enough, e.g. Zariski dense in $SL(n,\Ra)$,  this
action of $\G$ does not preserve the measure defined by any volume
form on $M$. A clever construction introduced in \cite{KatokLewis-global} shows
that one can alter the standard blowing up procedure in order to
produce a one parameter family of $SL(n,\Za)$ actions on $M$, only
one of which preserves a volume form. This immediately shows that
this action on $M$ admits perturbations, since it cannot be
conjugate to the nearby, non-volume preserving actions. Essentially,
one constructs different differentiable structures on $M$ which are
diffeomorphic but not equivariantly diffeomorphic.

After noticing this construction, one can proceed to build more
complicated examples by passing to a subgroup of finite index, and
then blowing up several fixed points.  One can also glue together
the ``blown up" fixed points to obtain an action on a manifold with
more complicated topology.  In particular, one can achieve a fundamental
group which is an essentially arbitrary free product with amalgamation or
HNN extension of the fundamental group of the original manifold over
the fundamental group of (blown-up) orbits. In the context described in more
 detail below, of blowing up along closed orbits instead of points, it is not 
 hard to do the blowing up and gluing in way that guarantees that there are no linear representations of the
fundamental group of the ``exotic example".  To prove non-existence of
linear representations, one chooses examples where all groups involved are
higher rank lattices and have very constrained linear representation theory.
See \cite{FisherWhyte-quotient, KatokLewis-global} for discussion of the
topological complications one can introduce.

While these actions do not preserve a rigid geometric structure, they do preserve
a slightly more general object, an {\em almost} rigid structure introduced by Benveniste and the author in \cite{BenvenisteFisher} and described below.

In \cite{Benveniste-thesis} it is observed that a similar construction can be used
for the action of a simple group $G$ by left translations on a
homogeneous space $H/\Lambda$ where $H$ is a Lie group containing
$G$ and $\Lambda<H$ is a cocompact lattice.  Here we use a slightly
more involved construction from algebraic geometry, and ``blow up"
the directions normal to a closed submanifold. I.e. we replace some
closed submanifold $N$ in $H/{\Lambda}$ by the projective normal
bundle to $N$.  In all cases we consider here, this normal bundle is
trivial and so is just $N{\times}{\Ra}P^{l}$ where
$l=\dim(H)-\dim(N)$.

Benveniste used his construction to produce more interesting
perturbations of actions of higher rank simple Lie group $G$ or a
lattice $\G$ in $G$.  In particular, he produced volume preserving
actions which admit volume preserving perturbations.  He does this
by choosing $G<H$ such that not only are there closed $G$ orbits but
so that the centralizer $Z=Z_H(G)$ of $G$ in $H$  has no-trivial
connected component.  If we take a closed $G$ orbit $N$, then any
translate $zN$ for $z$ in $Z$ is also closed and so we have a
continuum of closed $G$ orbits.  Benveniste shows that if we choose
two closed orbits $N$ and $zN$ to blow up and glue, and then vary
$z$ in a small open set, the resulting actions can only be conjugate
for a countable set of choices of $z$.

This construction is further elaborated in
\cite{Fisher-deformations}. Benveniste's construction is not optimal
in several senses, nor is his proof of rigidity.  In
\cite{Fisher-deformations}, I give a modification of the
construction that produces non-conjugate actions for every choice of
$z$ in a small enough neighborhood.  By blowing up and gluing more
pairs of closed orbits, this allows me to produce actions where the
space of deformations contains a submanifold of arbitrarily high,
finite dimension. Further, Benveniste's proof that the deformation
are non-trivial is quite involved and applies only to higher rank
groups. In \cite{Fisher-deformations}, I give a different proof of
non-triviality of the deformations, using consequences of Ratner's
theorem due to Shah and Witte Morris
\cite{Ratner,Shah,Witte-measquotients}. This shows that the
construction produces non-trivial perturbations for any semisimple
$G$ and any lattice $\G$ in $G$.

As mentioned above, in \cite{BenvenisteFisher} we show that none of these actions preserve any rigid
geometric structure in the sense of Gromov but that they do preserve a slightly more complicated
object which we call an {\em almost rigid structure}.  Both rigid and almost rigid structures are
easiest to define in dimension $1$.  In this context a rigid structure is a non-vanishing vector
field and an almost rigid structure is a vector field vanishing at isolated points and to finite
degree.

 We continue to use the notation of section
\ref{subsection:rgsdefs} in order to give the precise definition of almost rigid geometric structure.

\begin{defn}
An A-structure $\phi$ is called $(j,k)$-almost rigid (or just almost
rigid) if for every point $p$, $r^{k,k-1}_p$ is injective on the
subgroup $r^{k+j, k}(Is^{k+j}) \subset Is^{k}$.
\end{defn}

\noindent
Thus $k$-rigid structures are the $(0,k)$-almost rigid structures.

\noindent
{\bf Basic Example:} Let $V$ be an $n$-dimensional
manifold. Let $X_1, \ldots X_n$ be a collection of vector fields on
$M$. This defines an A-structure $\psi$ of type ${\mathbb R}^{n^2}$ on
$M$. If $X_1,{\ldots},X_n$ span the tangent space of $V$ at every point, then the
structure is rigid in the sense of Gromov. Suppose instead that there exists a
point $p$ in $V$ and
$X_1 \wedge \ldots \wedge X_n$ vanishes to order
$\leq j$ at $p$ in $V$.
Then $\psi$ is  a $(j,1)$-almost rigid structure. Indeed,
let $p \in M$, and let $(x_1, \dots, x_n)$ be coordinates around
$p$. Suppose that in terms of these coordinates, $X_l =a_l ^m
\frac{\partial}{\partial x_j}$. Suppose that $f \in Is^{j+1}_p$.
We must show that $r_p ^{j+1,1}(f)$
is trivial. Let $(f^1, \ldots , f^n)$ be the coordinate functions of
$f$. Then $f \in Is^{j+1} _p$ implies that
\begin{equation}
\label{1}
a_k ^l - a_k ^m \frac{\partial
f^l}{\partial x^m }
\end{equation}
 vanishes to order $j+1$ at $p$ for all $k$ and
$l$. Let $(b_k ^l)$ be the matrix so that $b_k ^m a_m ^l = \det(a_r
^s) \delta_k ^l$. Multiplying expression (\ref{1}) by $(b_k ^l)$, we see that
$\det(a_r ^s) (\delta_k ^l - \frac{\partial f^l}{\partial x^k})$
vanishes to order $j+1$. But since by assumption $\det(a_r ^s)$
vanishes to order $\leq j$, this implies that $(\partial f^l /
\partial x^k)(p) = \delta_k ^l$, so $r_p ^{j+1,1}(f)$ is the
identity, as required.

If confused by the notation, the interested reader may find it enlightening to work out
the basic example in the trivial case $n=1$.  Similar arguments can be given to show that
frames that degenerate to subframes are also almost rigid, provided the order of vanishing
of the form defining the frame is always finite.

\begin{question}
\label{question:almostrigid} Does any smooth (or analytic) action of
a higher rank lattice $\Gamma$ admit a smooth (analytic) almost
rigid structure in the sense of \cite{BenvenisteFisher}?  More generally
does such an action admit a smooth (analytic) rigid geometric structure
on an open dense set of full measure?
\end{question}

This question is, in a sense, related to the discussion above about
regularity of the straightening section in cocycle superrigidity. In
essence, cocycle superrigidity provides one with a measurable
invariant connection and what one wants to know is whether one can
improve the measurable connection to a smooth geometric structure
with some degeneracy on a small set.  The examples described in this
section show, among other things, that one cannot expect the
straightening section to be smooth in general, though one might hope
it is smooth in the complement of a closed submanifold of positive
codimension or at least on an open dense set of full measure.

We remark that there are other possible notions of almost rigid structures.
See the article in this volume by Dumitrescu for a detailed discussion
of a different useful notion in the context of complex analytic manifolds \cite{Dumitrescu-thisvolume}.
Dumitrescu's notion is strictly weaker than the one presented here.

\section{Non-volume preserving actions}
\label{section:novolume}

This section describes what is known for non-volume preserving actions.
The first two subsections describe examples which show that a classification in this
setting is not in any sense possible.  The last subsection describes some recent
work of Nevo and Zimmer that proves surprisingly strong rigidity results in special
settings.

\subsection{Stuck's examples}

 The following observations are from Stuck's paper \cite{Stuck}.
Let $G$ be any semisimple group.  Let $P$ be a minimal parabolic
subgroup.  Then there is a homomorphism $P{\rightarrow}\Ra$.  As in
the proof of Theorem \ref{theorem:rankoneactions}, one can take
any $\Ra$ action, view it is a $P$ action and induce to a $G$ action.
If we take an $\Ra$ action on a manifold $M$, then the induced
action takes place on $(G \times M)/P$.  We remark that the
$G$ action here is not volume preserving, simply because the $G$
action on $G/P$ is proximal.  The same is true of the restriction
to any $\G$ action when $\G$ is a lattice in $G$.  This
implies that classifying $G$ actions on all compact manifolds implicitly
involves classifying all vector fields on compact manifolds.  It is relatively
easy to convince oneself that there is no reasonable sense in which the moduli
space of vector fields can be understood up to smooth conjugacy.

\subsection{Weinberger's examples}

This is a variant on the Katok-Lewis examples obtained by blowing
up a point and is similar to a blowing up construction common
in foliation theory.  The idea is that one takes an action of a subgroup
$\G$ of $SL(n,\Za)$ on $M=\Ta^n$,
removes a fixed or periodic point $p$, retracts onto a manifold with
boundary $\bar M$ and then glues in a copy of the $\Ra^n$ compactified
at infinity by the projective space of rays.  It is relatively easy to
check that the resulting space admits a continuous $\G$ action and even that
there are many invariant measures for the action, but no invariant volume.
One can also modify this construction by doing the same construction at multiple fixed
or periodic points simultaneously and by doing more complicated gluings on the
resulting $\bar M$.

This construction is discussed in \cite{FarbShalen-3manifold} and a variant for abelian
group actions is discussed
in \cite{KalininKatokRodriguezHertz}.  In the abelian case it is possible to smooth
the action, but this does not seem to be the case for actions of higher rank lattices.

As far as I can tell, there is no obstruction to repeating this construction for closed
orbits as in the case of the algebro-geometric blow-up, but this does not seem
to be written formally anywhere in the literature.

Also, a recent construction of Hurder shows that one can iterate this construction
infinitely many times, taking retracts in smaller and smaller neighborhoods
of periodic points of higher and higher orders.  The resulting object is
a kind of fractal admitting an $SL(n,\Za)$ action \cite{Hurder-personalcommunication}.

\subsection{Work of Nevo-Zimmer}

In this subsection, we describe some work of Nevo and Zimmer from the
sequence of papers \cite{NevoZimmer-firstquotients, NevoZimmer-entropy, NevoZimmer-annals,
NevoZimmer-okalready}.  For more detailed discussion see the survey by Nevo and Zimmer
\cite{NevoZimmer-survey} as well as the following two articles for related results \cite{NevoZimmer-intermediate,NevoZimmer-deRham}.

Given a group $G$ acting on a space $X$ and a measure $\mu$ on $G$,
we call a measure $\nu$ on $X$ {\em stationary} if $\mu * \nu = \nu$.
We will only consider the case where the group generated by the support
$\mu$ is $G$, such measures are often called {\em admissible}.
This is a natural generalization of the notion of an invariant measure.
If $G$ is an amenable group, any action of $G$ on a compact metric space
admits an invariant measure.  If $G$ is not amenable, invariant measures
need not exist, but stationary measures always do.  We begin with the following
cautionary example:

\begin{example}
\label{example:projective}
Let $G=SL(n,\Ra)$ acting on $\Ra^{n+1}$ by the standard linear action on the first
$n$ coordinates and the trivial action on the last.  Then the corresponding $G$
action on  $\Pa(\Ra^{n+1})$ has the property that any stationary measure is supported
either on the subspace $\Pa(\Ra^n)$ given by the first $n$ coordinates or on the subspace
$\Pa(\Ra)$ given by zeroing the first $n$ coordinates.
\end{example}

The proof of this assertion is an easy exercise. The set $\Pa(\Ra)$ is a collection of
  fixed points, so clearly admits invariant measures.  The orbit of any point in
  $\Pa(\Ra^{n+1})\backslash(\Pa(\Ra^n)\bigsqcup\Pa(\Ra))$ is $SL(n, \Ra)/(SL(n,\Ra)\ltimes \Ra^n)$.
  It is straightforward to check that no stationary measures can be supported on unions
  of sets of this kind.
This fact should not be a surprise as it generalizes the fact that invariant measures
for amenable group actions are often supported on minimal sets.

To state the results of Nevo and Zimmer, we need a slightly stronger notion of admissibility. We say a measure $\mu$ on a locally compact group $G$ is {\em strongly admissible} if the support
of $\mu$ generates $G$ and $\mu^{*k}$ is absolutely continuous with respect to
Haar measure on $G$ for some positive $k$. In the papers of Nevo and Zimmer, this stronger notion is called
admissible.

\begin{theorem}
\label{theorem:nevozimmer}
Let $X$ be a compact $G$ space where $G$ is a semisimple Lie group with all
factors of real rank at least $2$.  Then for any admissible measure $\mu$ on $G$
and any $\mu$-stationary measure $\nu$ on $X$, we have either
\begin{enumerate}
\item $\nu$ is $G$ invariant or
\item there is a non-trivial measurable quotient of the $G$ space $(X,\nu)$
which is of the form $(G/Q, Lebesgue)$ where $Q \subset G$ is a parabolic subgroup.
\end{enumerate}
\end{theorem}

The quotient space $G/Q$ is called a {\em projective quotient} in the work of
Nevo and Zimmer.  Theorem \ref{theorem:nevozimmer} is most interesting for us
for minimal actions, where the measure $\nu$ is necessarily supported on all of $X$ and
the quotient therefore reflects the $\Gamma$ action on $X$ and not some smaller set.
Example \ref{example:projective} indicates the reason to be concerned, since there
the action on the larger projective space is not detected by any stationary measure.

We remark here that Feres and Ronshausen have introduced some interesting ideas for studying group
actions on sets not contained in the support of any stationary measure \cite{Feres-thisvolume}.  Similar ideas are developed in a somewhat different
context by Deroin, Kleptsyn and Navas \cite{DeroinKleptsynNavas}. 

In \cite{NevoZimmer-okalready}, Nevo and Zimmer show that for smooth non-measure preserving
actions on compact manifolds which also preserve a rigid geometric structure, one can sometimes prove the existence of a projective factor where the factor map is smooth on an open dense set.

We also want to mention the main result of Stuck's paper \cite{Stuck}, which we
have already cited repeatedly for the fact that non-volume preserving actions of
higher rank groups on compact manifolds cannot be classified.  

\begin{theorem}
\label{theorem:stuck}
Let $G$ be a semisimple Lie group with finite center.  Assume $G$ acts minimally by
homeomorphisms on a compact manifold $M$. Then either the action is locally free
or the action is induced from a minimal action by homeomorphisms of a proper parabolic subgroup of $G$ on a manifold $N$.
\end{theorem}

Stuck's theorem is proven by studying the Gauss map from $M$ to the Grassman variety 
of subspaces of $\mathfrak g$ defined by taking a point to the Lie algebra of it's
stabilizer.  This technique also plays an important role in the work of Nevo and Zimmer.

\section{Groups which do act on manifolds}
\label{section:moreexamples}

Much of the work discussed so far begs an obvious question: ``are there
many interesting subgroups of $\Diff(M)$ for a general $M$?"  So far
we have only seen ``large" subgroups of $\Diff(M)$ that arise in a
geometric fashion, from the presence of a connected Lie group in either
$\Diff(M)$ or $\Diff(\tilde M)$.  In this section, we describe two
classes of examples which make it clear that other phenomena exist.
The following problem, however, seems open:

\begin{problem}
\label{problem:nonlinear}
For a compact manifold $M$ with a volume form $\omega$ construct
a subgroup $\Gamma<\Diff(M,\omega)$ such that $\Gamma$ has no
linear representations.
\end{problem}

An example that is not often considered in this context is that
$\Aut(F_n)$ acts on the space $\Hom(F_n,K)$ for $K$ any compact
Lie group.  Since  $\Hom(F_n,K){\simeq}K^n$ this defines an action
of $\Aut(F_n)$ on a manifold. This action clearly preserves the
Haar measure on $K^n$, see\cite{Goldman-Out}.  This action is not
very well studied, we only know of \cite{Fisher-Out,Gelander-Out}.
Similar constructions are possible, and better known, with mapping
class groups, although in that case on obtains a representation variety
which is not usually a manifold.  Since $\Aut(F_n)$ has no faithful
linear representations, this yields an example of truly ``non-linear"
action of a large group.  This action is still very special and one
expects many other examples of non-linear actions.  We now
describe two constructions which yield many examples if we drop
the assumption that the action preserves volume.

\subsection{Thompson's groups}
\label{subsection:thompson}

Richard Thompson introduced a remarkable family of groups, now referred
to as Thompson's groups.  These come in various flavors and have been
studied from several points of view, see e.g. \cite{CannonFloydParry, Higman-Thompson}
For our purposes, the most important of these groups are the one's
typically denoted $T$.  One description of this group is the collection
of piecewise linear diffeomorphisms of a the circle where the break points
and slopes are all dyadic rationals.  (One can replace the implicit $2$ here
with other primes and obtain similar, but different, groups.)  We record
here two important facts about this group  $T$ of piecewise linear homeomorphisms of $S^1$.

\begin{theorem}[Thompson, see \cite{CannonFloydParry}]
\label{theorem:thompsonsimple}
The group $T$ is simple.
\end{theorem}

\begin{theorem}[Ghys-Sergiescu \cite{GhysSergiescu}]
\label{theorem:ghyssergesciu}
The defining piecewise linear action of the group $T$ on $S^1$ is
conjugate by a homeomorphism to a smooth action.
\end{theorem}

These two facts together provide us with a rather remarkable class of examples
of groups which act on manifolds.  As finitely generated linear groups are
always residually finite, the group $T$ has no linear representations whatsoever.
A simpler variant of Problem \ref{problem:nonlinear} is:

\begin{problem}
\label{problem:thompsonvolume}
Does $T$ admit a volume preserving action on a compact manifold?
\end{problem}

It is easy to see that compactness is essential in this question.  We can
construct a smooth action of $T$ on $S^1{\times}\Ra$ simply by taking the Ghys-Sergiescu
action on $S^1$ and acting on $\Ra$ by the inverse of the derivative cocycle.  Replacing
the derivative cocycle with the Jacobian cocycle, this procedure
quite generally converts non-volume preserving actions on compact manifolds to volume
preserving one's on non-compact manifolds, but we know of no real application of this
in the present context. Another variant of Problem \ref{problem:thompsonvolume} is

\begin{problem}
\label{problem:simplediff}
Given a compact manifold $M$ and a volume form $\omega$ does $\Diff(M,\omega)$ contains a finitely generated, infinite
discrete simple group?
\end{problem}

This question is reasonable for any degree of regularity on the
diffeomorphisms.

In Ghys survey on groups acting on the circle, he points out that Thompson's group
can be realized as piecewise $SL(2,\Za)$ homeomorphisms of the circle \cite{Ghys-circlesurvey}.
In \cite{FisherWhyte-oe}, Whyte and the author point out that the group of piecewise
$SL(n, \Za)$ maps on either the torus or the real projective space is quite large.
The following are natural questions, see \cite{FisherWhyte-oe} for more discussion.

\begin{question}
\label{question:piecewise}
Are there interesting finitely generated or finitely presented subgroups of piecewise $SL(n, \Za)$ maps on
$\Ta^n$ or $\Pa^{n-1}$?  Can any such group which is not a subgroup of $SL(n,\Za)$ be made to act smoothly?
\end{question}

\subsection{Highly distorted subgroups of $\Diff(S^2)$}
\label{subsection:calegarifriedman}

In \cite{CalegariFreedman-distortion}, Calegari and Freedman construct a very interesting
class of subgroups of $\Diff^{\infty}(S^2)$. Very roughly, they prove:

\begin{theorem}
\label{theorem:calegarifriedman}
There is a finitely generated subgroup $G$ of $\Diff^{\infty}(S^2)$ which contains a
rotation $r$ as an arbitrarily distorted element.
\end{theorem}

Here by {\em arbitrarily distorted}, we mean that we can choose the group $G$ so that
the function $f(n)=\|r^n\|_G$ grows more slowly than any function we choose.  It is
well-known that for linear groups, the function $f(n)$ is at worst a logarithm, so this
theorem immediately implies that we can find $G$ with no faithful linear representations.
This also answered a question raised by Franks and Handel in \cite{FranksHandel-distortionmeasure}.

More recently, Avila has constructed similar examples in $\Diff^{\infty}(S^1)$ \cite{Avila}.
This answers a question raised in \cite{CalegariFreedman-distortion} where such subgroups
were constructed in $\Diff^1(S^1)$.

We are naturally led to the following questions.  We say a diffeomorphism has {\em full support}
if the complement of the fixed set is dense.

\begin{question}
\label{question:distortionvolume}
For which compact manifolds $M$ does $\Diff^{\infty}(M)$ contain arbitrarily distorted elements of full support?
The same question for $\Diff^{\omega}(M)$?  The same question for $\Diff^{\infty}(M,\nu)$
where $\nu$ is a volume form on $M$? For the second two questions, we can drop the hypothesis of full support.
\end{question}

The second and third questions here seem quite difficult and the answer could conceivably be ``none".
The only examples where anything is known in the volume preserving setting are compact surfaces of genus at least one, where no element
is more than quadratically distorted by a result of Polterovich \cite{Polterovich}.  However this result depends heavily on the fact that in dimension two, preserving a volume form is the same as preserving a symplectic structure.

\subsection{Topological construction of actions on high dimensional spheres}
\label{subsection:farrelllafont}

In this subsection, we recall a construction due to Farrell and Lafont which allows one to construct
actions of a large class of groups on closed disks, and so by doubling, to construct actions on spheres
\cite{FarrellLafont}.
The construction yields actions on very high dimensional disks and spheres and is only known to produce
actions by homeomorphisms.

The class of groups involved is the set of groups which admit an $EZ$-structure.  This notion is a modification
of an earlier notion due to Bestvina of a $Z$-structure on a group \cite{Bestvina}.  We do not recall the precise definition here, but refer the reader to the introduction to \cite{FarrellLafont} and remark here that
both torsion free Gromov hyperbolic groups and $\CAT(0)$-groups admit $EZ$-structures.

The result that concerns us is:

\begin{theorem}
\label{theorem:farrelllafont}
Given a group $\Gamma$ with an $EZ$-structure, there is an action of $\Gamma$ by homeomorphisms
on a closed disk.
\end{theorem}

In fact, Farrell and Lafont give a fair amount of information concerning their actions which
are quite different from the actions we are usually concerned with here.  In particular,
the action is properly discontinuous off a closed subset $\Delta$ of the boundary of the disk.
So from a dynamical viewpoint $\Delta$ carries the ``interesting part" of the action, e.g.
is the support of any $\Gamma$ stationary measure.  Farrell and Lafont point out an analogy
between their construction and the action of a Kleinian group on the boundary of hyperbolic space.
An interesting general question is:

\begin{question}
\label{question:farrelllafont}
When can the Farrell-Lafont action on a disk or sphere be chosen smooth?
\end{question}

\section{Rigidity for other classes of acting groups}
\label{section:othergroups}

In this section, we collect some results and questions concerning actions of other
classes of groups.  In almost all cases, little is known.

\subsection{Lattices in semi-direct products}
\label{subsection:semidirect}
While it is not reasonable to expect classification results for arbitrary actions of
{\em all} lattices in {\em all} Lie groups, there are natural broader classes to consider.
To pick a reasonable class, a first guess is to try to exclude Lie groups whose lattices
have homomorphisms onto $\Za$ or larger free groups.  There are many such groups.  For
example, the groups $Sl(n,\Za)\ltimes \Za^n$ which are lattices in $Sl(n,\Ra)\ltimes \Ra^n$
have property $(T)$ as soon as $n>2$.  As it turns out, a reasonable setting is to consider
perfect Lie groups with no compact factors or factors locally isomorphic to $SO(1, n)$ or
$SU(1,n)$.  Any such Lie group will have property $(T)$ and therefore so will its lattices.
Many examples of such lattices are described in \cite{Valette-PairPaper}.  Some first
rigidity results for these groups are contained in Zimmer's paper \cite{Zimmer-algebraichull}.
The relevant full generalization of the cocycle superrigidity theorem is contained in
\cite{Witte-csr}.  In this context it seems that there are probably many rigidity theorems
concerning actions of these groups already implicit in the literature, following from
the results in \cite{Witte-csr} and existing arguments.

\subsection{Universal lattices}
\label{subsection:universallattices}

As mentioned above in subsection \ref{subsection:relatedquestions}, the groups $SL(n,\Za[X])$
for $n>2$ have property $(T)$ by a result of Shalom \cite{Shalom-ICM}.  His proof also works
with larger collections of variables and some other arithmetic groups.  The following is an
interesting problem.

\begin{problem}
\label{problem:universal}
Prove cocycle superrigidity for universal lattices.
\end{problem}

\noindent For clarity, we indicate that we expect that all linear representations and cocycles (up to ``compact noise")
will be described in terms of representations of the ambient groups, e.g. $SL(n,\Ra)$, and
will be determined by specifying a numerical value of $X$.

Some partial results towards superrigidity are known by Farb  \cite{Farb-helly} and Shenfield \cite{Shenfield}.  The
problem of completely classifying linear representations in this context does seem to be open.

\subsection{$SO(1,n)$ and $SU(1,n)$ and their lattices}
\label{subsection:rank1}

It is conjectured that all lattices in $SO(1,n)$ and $SU(1,n)$ admit finite index subgroups with surjections to $\Za$, see \cite{Borel-Wallach}.  The conjecture is usually attributed to Thurston and sometimes to Borel for
the case of $SU(1,n)$.
If this is true, it immediately implies that actions of those lattices can never be classified.

There are still some interesting results concerning actions of these lattices.  In \cite{Shalom-Annals},
Shalom places restrictions on the possible actions of $SO(1,n)$ and $SU(1,n)$ in terms of the fundamental group of the manifold
acted upon. This work is similar in spirit to work of Lubotzky and Zimmer described in Section \ref{section:arithmeticquotients}, but requires more restrictive hypotheses.

In \cite{Fisher-deformations}, the author exhibits large moduli spaces
of ergodic affine algebraic actions are constructed for certain lattices in $SO(1,n)$.  These
moduli spaces are, however, all finite dimensional. In \cite{Fisher-bending}, I construct
an infinite dimensional moduli of deformations of an isometric action of $SO(1,n)$.  Both
of these constructions rely on a notion of bending introduce by Johnson and Millson in the
finite dimensional setting \cite{JohnsonMillson}.

I do not formulate any precise questions or conjectures in this direction
as I am not sure what phenomenon to expect.

\subsection{Lattices in other locally compact groups}
\label{subsection:otherlattices}

Much recent work has focused on developing a theory of lattices in locally compact
group other than Lie groups.  This theory is fully developed for algebraic groups
over other local fields.  Though we did not mention it here, some of Zimmer's own
conjectures were made in the context of $S$-arithmetic groups, i.e. lattices in
products of real and p-adic Lie groups.  The following conjecture is natural in
this context and does not seem to be stated in the literature.

\begin{conjecture}
\label{conjecture:positivecharacteristic}
Let $G$ be semisimple algebraic group defined over a field $k$ of positive characteristic
and $\Gamma<G$ a lattice.  Further assume that all simple factors of $G$ have $k$-rank at least $2$.
Then any $\Gamma$ action on a compact manifold factors through a finite quotient.
\end{conjecture}

The existence of a measurable invariant metric in this context should be something
one can deduce from the cocycle superrigidity theorems, though it is not clear that
the correct form of these theorems is known or in the literature.

There is also a growing interest in lattices in locally compact groups that are
not algebraic.  We remark here that Kac-Moody lattices typically admit no
non-trivial homomorphisms even to $\Homeo(M)$, see \cite{FisherSilberman} for
a discussion.

The only other interesting class of lattices known to the author is the
lattices in the isometry group of a product of two trees constructed by
Burger and Mozes \cite{BurgerMozes-twotrees1,BurgerMozes-twotrees2}.  These groups
are infinite simple finitely presented groups. 

\begin{problem}
\label{problem:burgermozes}  Do the Burger-Mozes lattices admit any non-trivial
homomorphisms to $\Diff(M)$ when $M$ is a compact manifold.
\end{problem}

That these groups do act on high dimensional spheres by the construction discussed
in subsection \ref{subsection:farrelllafont} was pointed out to the author by Lafont.

\subsection{Automorphism groups of free and surface groups}
We briefly mention a last set of natural questions.  There is a longstanding
analogy between higher rank lattices and two other classes of groups.  These
are mapping class groups of surfaces and the outer automorphism group of the
free group.  See \cite{BridsonVogtmann-survey} for a detailed discussion of
this analogy.  In the context of this article, this raises the following
question.  Here we denote the mapping class group of a surface by $MCG(\Sigma)$
and the outer automorphism group of $F_n$ by $\Out(F_n)$.

\begin{question}
\label{question:mcgout}
Assume $\Sigma$ has genus at least two or that $n>2$.  Does $MCG(\Sigma)$ or $\Out(F_n)$
admit a faithful action on a compact manifold? A faithful action by smooth, volume
preserving diffeomorphisms?
\end{question}

By not assuming that $M$ is connected, we are implicitly asking the same question about
all finite index subgroups of $MCG(\Sigma)$ and $Out(F_n)$.  We recall from section \ref{section:moreexamples}
that $\Aut(F_n)$ does admit a volume preserving action on a compact manifold.  This makes
the question above particularly intriguing.

\section{Properties of subgroups of $\Diff(M)$}
\label{section:ghysprogram}

As remarked in subsection \ref{subsection:farbshalen}, in the paper \cite{Ghys-proches}, Ghys
attempts to reprove the classical Zassenhaus lemma for linear groups for groups of analytic diffeomorphisms.
While the full strength of the Zassenhaus lemma does not hold in this setting, many interesting
results do follow.   This immediately leaves one wondering to what extent other properties
of linear groups might hold for diffeomorphism groups or at least what analogues of many theorems
might be true.  This direction of research was initiated by Ghys and most of the questions below are due to him.

It is worth noting that some properties of finitely generated linear groups, like residual finiteness,
do not appear to have reasonable analogues in the setting of diffeomorphisms groups.  For residual
finiteness, the obvious example is the Thompson group discussed above, which is simple.

\subsection{Jordan's theorem}
\label{subsection:jordan}

For linear groups, there is a classical (and not too difficult) result
known as Jordan's theorem.  This says that for any finite subgroup of $GL(n, \Ca)$,
there is a subgroup of index at most $c(n)$ that is abelian.  One cannot expect
better than this, as $S^1$ has finite subgroups of arbitrarily large order and
is a subgroup of $GL(n,\Ca)$.  For proofs of Jordan's
theorem as well as the theorems on linear groups mentioned in the next subsection,
we refer the reader to e.g. \cite{Robinson-book}.

\begin{question}
\label{question:jordan}
Given a compact manifold $M$ and a finite subgroup $F$ of $\Diff(M)$, is
there a constant $c(M)$ such that $F$ has an abelian subgroup of index
$c(M)$.
\end{question}

As above, one cannot expect more than this, simply because one can
construct actions of $S^1$ on $M$ and therefore finite abelian
subgroups of $\Diff(M)$ of arbitrarily large order.

For this question, it may be more natural to ask about finite groups of homeomorphisms
and not assume any differentiability of the maps. Using the
results in e.g. \cite{MannSu}, one can show that at most finitely many simple
finite groups act on a given compact manifold.  To be clear, one can show
this using the classification of finite simple groups.  It would be most
interesting to resolve Question \ref{question:jordan} without reference
to the classification.

A recent preprint of Mundet i Rieri provides some evidence for a positive 
answer to the question \cite{Rieri}.

\subsection{Burnside problem}
\label{subsection:burnside}

A group is called {\em periodic} if all of its elements have finite order.
We say a periodic group $G$ has {\em bounded exponent} if every element has order at most $m$.
In 1905 Burnside proved that finitely generated linear groups of bounded exponent are finite and in 1911 Schur proved that finitely generated periodic linear groups are finite.  For a general finitely generated group, this is not
true, counterexamples to the {\em Burnside conjecture} were constructed
in a sequence of works by many authors, including Novikov, Golod, Shafarevich and
Ol'shanskii.  We refer the reader to the website: $http://www-groups.dcs.st-and.ac.uk/~history/HistTopics/Burnside\_problem.html$
for a detailed discussion of the history.  This page also discusses the {\em
restricted Burnside conjecture} resolved by Zelmanov.

In our context, the following questions seem natural:

\begin{question}
\label{question:burnside}
Are there infinite, finitely generated periodic groups of diffeomorphisms
of a compact manifold $M$?  Are there infinite, finitely generated bounded exponent
groups of diffeomorphisms of a compact manifold $M$?
\end{question}

Some first results in this direction are contained in the paper of Rebelo and Silva
\cite{RebeloSilva-burnside}.

\subsection{Tits' alternative}
\label{subsection:tits}

In this section, we ask a sequence of questions related to a famous theorem
of Tits' and more recent variants on it \cite{Tits}.

\begin{theorem}
\label{theorem:tits}
Let $\Gamma$ be a finitely generated linear group.  Then either $\Gamma$
contains a free subgroup on two generators or $\Gamma$ is virtually solvable.
\end{theorem}

The following conjecture of Ghys is a reasonable alternative to this for
groups of diffeomorphisms.

\begin{conjecture}
\label{conjecture:ghys}
Let $M$ be a compact manifold and $\Gamma$ a finitely generated group of
smooth diffeomorphisms of $M$.  Then either $\Gamma$ contains a free group
on two generators or $\Gamma$ preserves some measure on $M$.
\end{conjecture}

The best evidence for this conjecture to date is a
theorem of Margulis which proves the conjecture for $M=S^1$ \cite{Margulis-tits}.
Ghys has also asked if the more exact analogue of Tits' theorem might be
true for analytic diffeomorphisms.

A recent related line of research for linear groups concerns uniform
exponential growth.  A finitely generated group $\G$ has uniform exponential growth
if the number of elements in a ball of radius $r$ in a Cayley graph for $\G$
grows at least as $\lambda^r$ for some $\lambda>1$ that does not depend on the
choice of generators.  For linear groups, exponential growth implies
exponential growth by a theorem of Eskin, Mozes and Oh \cite{EskinMozesOh}.
There are examples of groups having exponential but not uniform exponential
growth due to Wilson \cite{Wilson}. This raises the following question.

\begin{question}
\label{question:uniformexponential}
Let $M$ be a compact manifold and $\Gamma$ a finitely generated subgroup
of $\Diff(M)$.  If $\Gamma$ has exponential growth does $\Gamma$ have uniform
exponential growth?
\end{question}

The question seems most likely to have a positive answer if one further assumes
that $\Gamma$ is non-amenable.

\bibliographystyle{amsplain}
\bibliography{zimmersurvey}

\def\cprime{$'$}
\providecommand{\bysame}{\leavevmode\hbox to3em{\hrulefill}\thinspace}
\providecommand{\MR}{\relax\ifhmode\unskip\space\fi MR }
\providecommand{\MRhref}[2]{%
  \href{http://www.ams.org/mathscinet-getitem?mr=#1}{#2}
}
\providecommand{\href}[2]{#2}
\begin{thebibliography}{100}

\bibitem{Adams-book}
Scot Adams, \emph{Dynamics on {L}orentz manifolds}, World Scientific Publishing
  Co. Inc., River Edge, NJ, 2001. \MR{MR1875067 (2003c:53101)}

\bibitem{Andersen-notT}
Jurgen Andersen, \emph{Mapping class groups do not have kazhdan's property
  $({T})$}, preprint 2007.

\bibitem{Anosov}
D.~V. Anosov, \emph{Geodesic flows on closed {R}iemann manifolds with negative
  curvature.}, Proceedings of the Steklov Institute of Mathematics, No. 90
  (1967). Translated from the Russian by S. Feder, American Mathematical
  Society, Providence, R.I., 1969. \MR{MR0242194 (39 \#3527)}

\bibitem{Avila}
Artur Avila, \emph{Distortion elements in {D}iff$^{\infty}$( {R}/{Z})},
  preprint.

\bibitem{BaderFurmanGelanderMonod}
Uri Bader, Alex Furman, Tsachik Gelander, and Nicolas Monod, \emph{Property
  ({T}) and rigidity for actions on {B}anach spaces}, Acta Math. \textbf{198}
  (2007), no.~1, 57--105. \MR{MR2316269 (2008g:22007)}

\bibitem{BaderShalom}
Uri Bader and Yehuda Shalom, \emph{Factor and normal subgroup theorems for
  lattices in products of groups}, Invent. Math. \textbf{163} (2006), no.~2,
  415--454. \MR{MR2207022 (2006m:22017)}

\bibitem{Banyaga-structure}
Augustin Banyaga, \emph{Sur la structure du groupe des diff\'eomorphismes qui
  pr\'eservent une forme symplectique}, Comment. Math. Helv. \textbf{53}
  (1978), no.~2, 174--227. \MR{MR490874 (80c:58005)}

\bibitem{Banyaga-book}
\bysame, \emph{The structure of classical diffeomorphism groups}, Mathematics
  and its Applications, vol. 400, Kluwer Academic Publishers Group, Dordrecht,
  1997. \MR{MR1445290 (98h:22024)}

\bibitem{Benoist}
Yves Benoist, \emph{Orbites des structures rigides (d'apr\`es {M}. {G}romov)},
  Integrable systems and foliations/Feuilletages et syst\`emes int\'egrables
  (Montpellier, 1995), Progr. Math., vol. 145, Birkh\"auser Boston, Boston, MA,
  1997, pp.~1--17. \MR{MR1432904 (98c:58126)}

\bibitem{BenoistLabourie}
Yves Benoist and Fran{\c{c}}ois Labourie, \emph{Sur les diff\'eomorphismes
  d'{A}nosov affines \`a\ feuilletages stable et instable diff\'erentiables},
  Invent. Math. \textbf{111} (1993), no.~2, 285--308. \MR{MR1198811
  (94d:58114)}

\bibitem{BenvenisteFisher}
E.~Jerome Benveniste and David Fisher, \emph{Nonexistence of invariant rigid
  structures and invariant almost rigid structures}, Comm. Anal. Geom.
  \textbf{13} (2005), no.~1, 89--111. \MR{MR2154667 (2006f:53056)}

\bibitem{Benveniste-thesis}
E.J. Benveniste, \emph{University of {C}hicago, {P}h.{D}. thesis.},
  unpublished, 1996.

\bibitem{BCG-GAFA}
G.~Besson, G.~Courtois, and S.~Gallot, \emph{Entropies et rigidit\'es des
  espaces localement sym\'etriques de courbure strictement n\'egative}, Geom.
  Funct. Anal. \textbf{5} (1995), no.~5, 731--799. \MR{MR1354289 (96i:58136)}

\bibitem{BCG-ETDS}
G{\'e}rard Besson, Gilles Courtois, and Sylvestre Gallot, \emph{Minimal entropy
  and {M}ostow's rigidity theorems}, Ergodic Theory Dynam. Systems \textbf{16}
  (1996), no.~4, 623--649. \MR{MR1406425 (97e:58177)}

\bibitem{Bestvina}
Mladen Bestvina, \emph{Local homology properties of boundaries of groups},
  Michigan Math. J. \textbf{43} (1996), no.~1, 123--139. \MR{MR1381603
  (97a:57022)}

\bibitem{BestvinaFujiwara}
Mladen Bestvina and Koji Fujiwara, \emph{Bounded cohomology of subgroups of
  mapping class groups}, Geom. Topol. \textbf{6} (2002), 69--89 (electronic).
  \MR{MR1914565 (2003f:57003)}

\bibitem{Borel-Wallach}
A.~Borel and N.~Wallach, \emph{Continuous cohomology, discrete subgroups, and
  representations of reductive groups}, second ed., Mathematical Surveys and
  Monographs, vol.~67, American Mathematical Society, Providence, RI, 2000.
  \MR{MR1721403 (2000j:22015)}

\bibitem{Borel}
Armand Borel, \emph{Compact {C}lifford-{K}lein forms of symmetric spaces},
  Topology \textbf{2} (1963), 111--122. \MR{MR0146301 (26 \#3823)}

\bibitem{BridsonVogtmann}
Martin Bridson and Karen Vogtmann, \emph{get title}, preprint.

\bibitem{BridsonVogtmann-survey}
Martin~R. Bridson and Karen Vogtmann, \emph{Automorphism groups of free groups,
  surface groups and free abelian groups}, Problems on mapping class groups and
  related topics, Proc. Sympos. Pure Math., vol.~74, Amer. Math. Soc.,
  Providence, RI, 2006, pp.~301--316. \MR{MR2264548 (2008g:20091)}

\bibitem{BurgerMonod}
M.~Burger and N.~Monod, \emph{Continuous bounded cohomology and applications to
  rigidity theory}, Geom. Funct. Anal. \textbf{12} (2002), no.~2, 219--280.
  \MR{MR1911660 (2003d:53065a)}

\bibitem{Burger-thisvolume}
Marc Burger, \emph{get title}, in this volume.

\bibitem{Burger-ICM}
\bysame, \emph{Rigidity properties of group actions on {${\rm
  CAT}(0)$}-spaces}, Proceedings of the {I}nternational {C}ongress of
  {M}athematicians, {V}ol.\ 1, 2 ({Z}\"urich, 1994) (Basel), Birkh\"auser,
  1995, pp.~761--769. \MR{MR1403976 (97j:20033)}

\bibitem{BurgerMozes-twotrees2}
Marc Burger and Shahar Mozes, \emph{Groups acting on trees: from local to
  global structure}, Inst. Hautes \'Etudes Sci. Publ. Math. (2000), no.~92,
  113--150 (2001). \MR{MR1839488 (2002i:20041)}

\bibitem{BurgerMozes-twotrees1}
\bysame, \emph{Lattices in product of trees}, Inst. Hautes \'Etudes Sci. Publ.
  Math. (2000), no.~92, 151--194 (2001). \MR{MR1839489 (2002i:20042)}

\bibitem{BurnsPughShubWilkinson}
Keith Burns, Charles Pugh, Michael Shub, and Amie Wilkinson, \emph{Recent
  results about stable ergodicity}, Smooth ergodic theory and its applications
  (Seattle, WA, 1999), Proc. Sympos. Pure Math., vol.~69, Amer. Math. Soc.,
  Providence, RI, 2001, pp.~327--366. \MR{MR1858538 (2002m:37042)}

\bibitem{Calabi}
Eugenio Calabi, \emph{On compact, {R}iemannian manifolds with constant
  curvature. {I}}, Proc. Sympos. Pure Math., Vol. III, American Mathematical
  Society, Providence, R.I., 1961, pp.~155--180. \MR{MR0133787 (24 \#A3612)}

\bibitem{CalabiVesentini}
Eugenio Calabi and Edoardo Vesentini, \emph{On compact, locally symmetric
  {K}\"ahler manifolds}, Ann. of Math. (2) \textbf{71} (1960), 472--507.
  \MR{MR0111058 (22 \#1922b)}

\bibitem{CalegariFreedman-distortion}
Danny Calegari and Michael~H. Freedman, \emph{Distortion in transformation
  groups}, Geom. Topol. \textbf{10} (2006), 267--293 (electronic), With an
  appendix by Yves de Cornulier. \MR{MR2207794 (2007b:37048)}

\bibitem{CandelQuiroga-one}
A.~Candel and R.~Quiroga-Barranco, \emph{Gromov's centralizer theorem}, Geom.
  Dedicata \textbf{100} (2003), 123--155. \MR{MR2011119 (2004m:53075)}

\bibitem{CandelQuiroga-two}
Alberto Candel and Raul Quiroga-Barranco, \emph{Parallelisms, prolongations of
  {L}ie algebras and rigid geometric structures}, Manuscripta Math.
  \textbf{114} (2004), no.~3, 335--350. \MR{MR2076451 (2005f:53056)}

\bibitem{CannonFloydParry}
J.~W. Cannon, W.~J. Floyd, and W.~R. Parry, \emph{Introductory notes on
  {R}ichard {T}hompson's groups}, Enseign. Math. (2) \textbf{42} (1996),
  no.~3-4, 215--256. \MR{MR1426438 (98g:20058)}

\bibitem{CannonThurston}
James~W. Cannon and William~P. Thurston, \emph{Group invariant {P}eano curves},
  Geom. Topol. \textbf{11} (2007), 1315--1355. \MR{MR2326947 (2008i:57016)}

\bibitem{Cantat-birat}
Serge Cantat, \emph{Groupes de transformations birationnelles du plan},
  preprint.

\bibitem{Cantat-zimmer}
\bysame, \emph{Version k\"ahl\'erienne d'une conjecture de {R}obert {J}.
  {Z}immer}, Ann. Sci. \'Ecole Norm. Sup. (4) \textbf{37} (2004), no.~5,
  759--768. \MR{MR2103473 (2006b:22010)}

\bibitem{CapraceRemy}
Pierre-Emmanuel Caprace and Bertrand R{\'e}my, \emph{Simplicit\'e abstraite des
  groupes de {K}ac-{M}oody non affines}, C. R. Math. Acad. Sci. Paris
  \textbf{342} (2006), no.~8, 539--544. \MR{MR2217912 (2006k:20102)}

\bibitem{ChatterjiKassabov}
Indira Chatterji and Martin Kassobov, \emph{get title}, in preparation.

\bibitem{ConnellFarb}
Christopher Connell and Benson Farb, \emph{Some recent applications of the
  barycenter method in geometry}, Topology and geometry of manifolds (Athens,
  GA, 2001), Proc. Sympos. Pure Math., vol.~71, Amer. Math. Soc., Providence,
  RI, 2003, pp.~19--50. \MR{MR2024628 (2005e:53058)}

\bibitem{Corlette-Annals}
Kevin Corlette, \emph{Archimedean superrigidity and hyperbolic geometry}, Ann.
  of Math. (2) \textbf{135} (1992), no.~1, 165--182. \MR{MR1147961 (92m:57048)}

\bibitem{CorletteZimmer}
Kevin Corlette and Robert~J. Zimmer, \emph{Superrigidity for cocycles and
  hyperbolic geometry}, Internat. J. Math. \textbf{5} (1994), no.~3, 273--290.
  \MR{MR1274120 (95g:58055)}

\bibitem{Dani-quotient}
S.~G. Dani, \emph{Continuous equivariant images of lattice-actions on
  boundaries}, Ann. of Math. (2) \textbf{119} (1984), no.~1, 111--119.
  \MR{MR736562 (85i:22009)}

\bibitem{delaHarpeValette-book}
Pierre de~la Harpe and Alain Valette, \emph{La propri\'et\'e {$(T)$} de
  {K}azhdan pour les groupes localement compacts (avec un appendice de {M}arc
  {B}urger)}, Ast\'erisque (1989), no.~175, 158, With an appendix by M. Burger.
  \MR{MR1023471 (90m:22001)}

\bibitem{Delorme}
Patrick Delorme, \emph{{$1$}-cohomologie des repr\'esentations unitaires des
  groupes de {L}ie semi-simples et r\'esolubles. {P}roduits tensoriels continus
  de repr\'esentations}, Bull. Soc. Math. France \textbf{105} (1977), no.~3,
  281--336. \MR{MR0578893 (58 \#28272)}

\bibitem{DeroinKleptsynNavas}
Bertrand Deroin, Victor Kleptsyn, and Andr{\'e}s Navas, \emph{Sur la dynamique
  unidimensionnelle en r\'egularit\'e interm\'ediaire}, Acta Math. \textbf{199}
  (2007), no.~2, 199--262. \MR{MR2358052}

\bibitem{Deserti}
Julie D{\'e}serti, \emph{Groupe de {C}remona et dynamique complexe: une
  approche de la conjecture de {Z}immer}, Int. Math. Res. Not. (2006), Art. ID
  71701, 27. \MR{MR2233717 (2007d:22013)}

\bibitem{DinhSibony}
Tien-Cuong Dinh and Nessim Sibony, \emph{Groupes commutatifs d'automorphismes
  d'une vari\'et\'e k\"ahl\'erienne compacte}, Duke Math. J. \textbf{123}
  (2004), no.~2, 311--328. \MR{MR2066940 (2005g:32020)}

\bibitem{Dumitrescu-thisvolume}
Sorin Dumitrescu, \emph{Meromorphic almost rigid geometric structures}, in this
  volume.

\bibitem{EntovPolterovich}
Michael Entov and Leonid Polterovich, \emph{Calabi quasimorphism and quantum
  homology}, Int. Math. Res. Not. (2003), no.~30, 1635--1676. \MR{MR1979584
  (2004e:53131)}

\bibitem{Epstein-simple}
D.~B.~A. Epstein, \emph{The simplicity of certain groups of homeomorphisms},
  Compositio Math. \textbf{22} (1970), 165--173. \MR{MR0267589 (42 \#2491)}

\bibitem{EskinMozesOh}
Alex Eskin, Shahar Mozes, and Hee Oh, \emph{On uniform exponential growth for
  linear groups}, Invent. Math. \textbf{160} (2005), no.~1, 1--30.
  \MR{MR2129706 (2006a:20081)}

\bibitem{Farb-helly}
Benson Farb, \emph{Group actions and helly's theorem}, preprint, 2008.

\bibitem{FarbMasur}
Benson Farb and Howard Masur, \emph{Superrigidity and mapping class groups},
  Topology \textbf{37} (1998), no.~6, 1169--1176. \MR{MR1632912 (99f:57017)}

\bibitem{FarbShalen-analytic}
Benson Farb and Peter Shalen, \emph{Real-analytic actions of lattices}, Invent.
  Math. \textbf{135} (1999), no.~2, 273--296. \MR{MR1666834 (2000c:22017)}

\bibitem{FarbShalen-3manifold}
\bysame, \emph{Lattice actions, 3-manifolds and homology}, Topology \textbf{39}
  (2000), no.~3, 573--587. \MR{MR1746910 (2001b:57041)}

\bibitem{FarbShalen-analyticS^1}
\bysame, \emph{Groups of real-analytic diffeomorphisms of the circle}, Ergodic
  Theory Dynam. Systems \textbf{22} (2002), no.~3, 835--844. \MR{MR1908556
  (2003e:37030)}

\bibitem{FarbShalen-analyticdim4}
Benson Farb and Peter~B. Shalen, \emph{Real-analytic, volume-preserving actions
  of lattices on 4-manifolds}, C. R. Math. Acad. Sci. Paris \textbf{334}
  (2002), no.~11, 1011--1014. \MR{MR1913726 (2003e:57055)}

\bibitem{Farley}
Daniel~S. Farley, \emph{Proper isometric actions of {T}hompson's groups on
  {H}ilbert space}, Int. Math. Res. Not. (2003), no.~45, 2409--2414.
  \MR{MR2006480 (2004k:22005)}

\bibitem{FarrellLafont}
F.~T. Farrell and J.-F. Lafont, \emph{E{Z}-structures and topological
  applications}, Comment. Math. Helv. \textbf{80} (2005), no.~1, 103--121.
  \MR{MR2130569 (2006b:57022)}

\bibitem{Feres-affine}
R.~Feres, \emph{Affine actions of higher rank lattices}, Geom. Funct. Anal.
  \textbf{3} (1993), no.~4, 370--394. \MR{MR1223436 (96d:22013)}

\bibitem{FeresLabourie}
R.~Feres and F.~Labourie, \emph{Topological superrigidity and {A}nosov actions
  of lattices}, Ann. Sci. \'Ecole Norm. Sup. (4) \textbf{31} (1998), no.~5,
  599--629. \MR{MR1643954 (99k:58112)}

\bibitem{FeresLampe}
R.~Feres and P.~Lampe, \emph{Cartan geometries and dynamics}, Geom. Dedicata
  \textbf{80} (2000), no.~1-3, 29--41. \MR{MR1762497 (2001i:53059)}

\bibitem{Feres-IJMconnection}
Renato Feres, \emph{Connection preserving actions of lattices in {${\rm SL}\sb
  n{\bf R}$}}, Israel J. Math. \textbf{79} (1992), no.~1, 1--21. \MR{MR1195250
  (94a:58039)}

\bibitem{Feres-Zimmerconjecture}
\bysame, \emph{Actions of discrete linear groups and {Z}immer's conjecture}, J.
  Differential Geom. \textbf{42} (1995), no.~3, 554--576. \MR{MR1367402
  (97a:22016)}

\bibitem{Feres-book}
\bysame, \emph{Dynamical systems and semisimple groups: an introduction},
  Cambridge Tracts in Mathematics, vol. 126, Cambridge University Press,
  Cambridge, 1998. \MR{MR1670703 (2001d:22001)}

\bibitem{Feres-csr}
\bysame, \emph{An introduction to cocycle super-rigidity}, Rigidity in dynamics
  and geometry (Cambridge, 2000), Springer, Berlin, 2002, pp.~99--134.
  \MR{MR1919397 (2003g:37041)}

\bibitem{Feres-Gromov}
\bysame, \emph{Rigid geometric structures and actions of semisimple {L}ie
  groups}, Rigidit\'e, groupe fondamental et dynamique, Panor. Synth\`eses,
  vol.~13, Soc. Math. France, Paris, 2002, pp.~121--167. \MR{MR1993149
  (2004m:53076)}

\bibitem{FeresKatok-survey}
Renato Feres and Anatole Katok, \emph{Ergodic theory and dynamics of
  {$G$}-spaces (with special emphasis on rigidity phenomena)}, Handbook of
  dynamical systems, Vol.\ 1A, North-Holland, Amsterdam, 2002, pp.~665--763.
  \MR{MR1928526 (2003j:37005)}

\bibitem{Feres-thisvolume}
Renato Feres and Emily Ronshausen, \emph{Harmonic functions over group
  actions}, in this volume.

\bibitem{FisherWhyte-oe}
D.~Fisher and K.~Whyte, \emph{When is a group action determined by its orbit
  structure?}, Geom. Funct. Anal. \textbf{13} (2003), no.~6, 1189--1200.
  \MR{MR2033836 (2004k:37045)}

\bibitem{Fisher-bending}
David Fisher, \emph{Cohomology of arithmetic groups and bending group actions},
  in preparation.

\bibitem{Fisher-H1}
\bysame, \emph{First cohomology and local rigidity of group actions}, in
  preparation.

\bibitem{Fisher-IJM}
David Fisher, \emph{A canonical arithmetic quotient for actions of lattices in
  simple groups}, Israel J. Math. \textbf{124} (2001), 143--155. \MR{MR1856509
  (2002f:22016)}

\bibitem{Fisher-ETDS}
\bysame, \emph{On the arithmetic structure of lattice actions on compact
  spaces}, Ergodic Theory Dynam. Systems \textbf{22} (2002), no.~4, 1141--1168.
  \MR{MR1926279 (2004j:37004)}

\bibitem{Fisher-Cambridge}
\bysame, \emph{Rigid geometric structures and representations of fundamental
  groups}, Rigidity in dynamics and geometry ({C}ambridge, 2000), Springer,
  Berlin, 2002, pp.~135--147. \MR{MR1919398 (2003g:22007)}

\bibitem{Fisher-Out}
\bysame, \emph{{${\rm Out}(F\sb n)$} and the spectral gap conjecture}, Int.
  Math. Res. Not. (2006), Art. ID 26028, 9. \MR{MR2250018 (2007f:22006)}

\bibitem{Fisher-survey}
\bysame, \emph{Local rigidity of group actions: past, present, future},
  Dynamics, ergodic theory, and geometry, Math. Sci. Res. Inst. Publ., vol.~54,
  Cambridge Univ. Press, Cambridge, 2007, pp.~45--97. \MR{MR2369442}

\bibitem{Fisher-deformations}
\bysame, \emph{Deformations of group actions}, Trans. Amer. Math. Soc.
  \textbf{360} (2008), no.~1, 491--505 (electronic). \MR{MR2342012}

\bibitem{FisherHitchman-csr}
David Fisher and Theron Hitchman, \emph{Harmonic maps with infinite dimensional
  targets and cocycle superrigidity}, in preparation.

\bibitem{FisherHitchman-IMRN}
\bysame, \emph{Cocycle superrigidity and harmonic maps with
  infinite-dimensional targets}, Int. Math. Res. Not. (2006), 72405, 1--19.
  \MR{MR2211160}

\bibitem{FisherMargulis-lrc}
David Fisher and G.~A. Margulis, \emph{Local rigidity for cocycles}, Surveys in
  differential geometry, Vol.\ VIII (Boston, MA, 2002), Surv. Differ. Geom.,
  VIII, Int. Press, Somerville, MA, 2003, pp.~191--234. \MR{MR2039990
  (2004m:22032)}

\bibitem{Fisher-Margulis1}
David Fisher and Gregory Margulis, \emph{Almost isometric actions, property
  {$(T)$}, and local rigidity}, Invent. Math. \textbf{162} (2005), no.~1,
  19--80. \MR{MR2198325}

\bibitem{FisherSilberman}
David Fisher and Lior Silberman, \emph{Groups not acting on manifolds}, Int.
  Math. Res. Not. (2008).

\bibitem{FisherWhyte-quotient}
David Fisher and Kevin Whyte, \emph{Continuous quotients for lattice actions on
  compact spaces}, Geom. Dedicata \textbf{87} (2001), no.~1-3, 181--189.
  \MR{MR1866848 (2002j:57070)}

\bibitem{FisherZimmer}
David Fisher and Robert~J. Zimmer, \emph{Geometric lattice actions, entropy and
  fundamental groups}, Comment. Math. Helv. \textbf{77} (2002), no.~2,
  326--338. \MR{MR1915044 (2003e:57056)}

\bibitem{FranksHandel-distortionvolume}
John Franks and Michael Handel, \emph{Area preserving group actions on
  surfaces}, Geom. Topol. \textbf{7} (2003), 757--771 (electronic).
  \MR{MR2026546 (2004j:37042)}

\bibitem{FranksHandel-normalformhamiltonian}
\bysame, \emph{Periodic points of {H}amiltonian surface diffeomorphisms}, Geom.
  Topol. \textbf{7} (2003), 713--756 (electronic). \MR{MR2026545 (2004j:37101)}

\bibitem{FranksHandel-distortionmeasure}
\bysame, \emph{Distortion elements in group actions on surfaces}, Duke Math. J.
  \textbf{131} (2006), no.~3, 441--468. \MR{MR2219247 (2007c:37018)}

\bibitem{Furman-thisvolume}
Alex Furman, \emph{get title}, in this volume.

\bibitem{Furman-random}
\bysame, \emph{Random walks on groups and random transformations}, Handbook of
  dynamical systems, {V}ol.\ 1{A}, North-Holland, Amsterdam, 2002,
  pp.~931--1014. \MR{MR1928529 (2003j:60065)}

\bibitem{FurmanMonod}
Alex Furman and Nicolas Monod, \emph{Product groups acting on manifolds},
  preprint, 2007.

\bibitem{Furstenberg-book}
H.~Furstenberg, \emph{Recurrence in ergodic theory and combinatorial number
  theory}, Princeton University Press, Princeton, N.J., 1981, M. B. Porter
  Lectures. \MR{MR603625 (82j:28010)}

\bibitem{Furstenberg-poisson}
Harry Furstenberg, \emph{A {P}oisson formula for semi-simple {L}ie groups},
  Ann. of Math. (2) \textbf{77} (1963), 335--386. \MR{MR0146298 (26 \#3820)}

\bibitem{Furstenberg-ICM}
\bysame, \emph{Boundaries of {L}ie groups and discrete subgroups}, Actes du
  Congr\`es International des Math\'ematiciens (Nice, 1970), Tome 2,
  Gauthier-Villars, Paris, 1971, pp.~301--306. \MR{MR0430160 (55 \#3167)}

\bibitem{Furstenberg-csr}
\bysame, \emph{Rigidity and cocycles for ergodic actions of semisimple {L}ie
  groups (after {G}. {A}. {M}argulis and {R}. {Z}immer)}, Bourbaki Seminar,
  Vol. 1979/80, Lecture Notes in Math., vol. 842, Springer, Berlin, 1981,
  pp.~273--292. \MR{MR636529 (83j:22003)}

\bibitem{GambaudoGhys}
Jean-Marc Gambaudo and {\'E}tienne Ghys, \emph{Commutators and diffeomorphisms
  of surfaces}, Ergodic Theory Dynam. Systems \textbf{24} (2004), no.~5,
  1591--1617. \MR{MR2104597 (2006d:37071)}

\bibitem{Gelander-Out}
Tsachik Gelander, \emph{On deformtions of $f_n$ in compact lie groups}, to
  appear Israel Journal Math.

\bibitem{GelanderKarlssonMargulis}
Tsachik Gelander, Anders Karlsson, and Gregory~A. Margulis,
  \emph{Superrigidity, generalized harmonic maps and uniformly convex spaces},
  Geom. Funct. Anal. \textbf{17} (2008), no.~5, 1524--1550. \MR{MR2377496}

\bibitem{Ghys-circleboundedcohomology}
{\'E}tienne Ghys, \emph{Groupes d'hom\'eomorphismes du cercle et cohomologie
  born\'ee}, The Lefschetz centennial conference, Part III (Mexico City, 1984),
  Contemp. Math., vol.~58, Amer. Math. Soc., Providence, RI, 1987, pp.~81--106.
  \MR{MR893858 (88m:58024)}

\bibitem{Ghys-proches}
\bysame, \emph{Sur les groupes engendr\'es par des diff\'eomorphismes proches
  de l'identit\'e}, Bol. Soc. Brasil. Mat. (N.S.) \textbf{24} (1993), no.~2,
  137--178. \MR{MR1254981 (95f:58017)}

\bibitem{Ghys-circleaction}
\bysame, \emph{Actions de r\'eseaux sur le cercle}, Invent. Math. \textbf{137}
  (1999), no.~1, 199--231. \MR{MR1703323 (2000j:22014)}

\bibitem{Ghys-circlesurvey}
\bysame, \emph{Groups acting on the circle}, Enseign. Math. (2) \textbf{47}
  (2001), no.~3-4, 329--407. \MR{MR1876932 (2003a:37032)}

\bibitem{GhysSergiescu}
{\'E}tienne Ghys and Vlad Sergiescu, \emph{Sur un groupe remarquable de
  diff\'eomorphismes du cercle}, Comment. Math. Helv. \textbf{62} (1987),
  no.~2, 185--239. \MR{MR896095 (90c:57035)}

\bibitem{Goetze-affine}
Edward~R. Goetze, \emph{Connection preserving actions of connected and discrete
  {L}ie groups}, J. Differential Geom. \textbf{40} (1994), no.~3, 595--620.
  \MR{MR1305982 (95m:57052)}

\bibitem{GoetzeSpatzier-Duke}
Edward~R. Goetze and Ralf~J. Spatzier, \emph{On {L}iv\v sic's theorem,
  superrigidity, and {A}nosov actions of semisimple {L}ie groups}, Duke Math.
  J. \textbf{88} (1997), no.~1, 1--27. \MR{MR1448015 (98d:58134)}

\bibitem{GoetzeSpatzier-Annals}
\bysame, \emph{Smooth classification of {C}artan actions of higher rank
  semisimple {L}ie groups and their lattices}, Ann. of Math. (2) \textbf{150}
  (1999), no.~3, 743--773. \MR{MR1740993 (2001c:37029)}

\bibitem{Goldman-Out}
William~M. Goldman, \emph{An ergodic action of the outer automorphism group of
  a free group}, Geom. Funct. Anal. \textbf{17} (2007), no.~3, 793--805.
  \MR{MR2346275 (2008g:57001)}

\bibitem{Gorodnik-problemlist}
Alexander Gorodnik, \emph{Open problems in dynamics and related fields}, J.
  Mod. Dyn. \textbf{1} (2007), no.~1, 1--35. \MR{MR2261070 (2007f:37001)}

\bibitem{Gromov-hyp}
M.~Gromov, \emph{Hyperbolic groups}, Essays in group theory, Math. Sci. Res.
  Inst. Publ., vol.~8, Springer, New York, 1987, pp.~75--263. \MR{MR919829
  (89e:20070)}

\bibitem{Gromov-RigidStructures}
Michael Gromov, \emph{Rigid transformations groups}, G\'eom\'etrie
  diff\'erentielle (Paris, 1986), Travaux en Cours, vol.~33, Hermann, Paris,
  1988, pp.~65--139. \MR{MR955852 (90d:58173)}

\bibitem{GromovSchoen}
Mikhail Gromov and Richard Schoen, \emph{Harmonic maps into singular spaces and
  {$p$}-adic superrigidity for lattices in groups of rank one}, Inst. Hautes
  \'Etudes Sci. Publ. Math. (1992), no.~76, 165--246. \MR{MR1215595
  (94e:58032)}

\bibitem{Guichardet}
Alain Guichardet, \emph{Sur la cohomologie des groupes topologiques. {II}},
  Bull. Sci. Math. (2) \textbf{96} (1972), 305--332. \MR{MR0340464 (49 \#5219)}

\bibitem{Helgason-book}
Sigurdur Helgason, \emph{Differential geometry, {L}ie groups, and symmetric
  spaces}, Graduate Studies in Mathematics, vol.~34, American Mathematical
  Society, Providence, RI, 2001, Corrected reprint of the 1978 original.
  \MR{MR1834454 (2002b:53081)}

\bibitem{Herman-torus}
Michael-Robert Herman, \emph{Simplicit\'e du groupe des diff\'eomorphismes de
  classe {$C\sp{\infty }$}, isotopes \`a l'identit\'e, du tore de dimension
  {$n$}}, C. R. Acad. Sci. Paris S\'er. A-B \textbf{273} (1971), A232--A234.
  \MR{MR0287585 (44 \#4788)}

\bibitem{Higman-Thompson}
Graham Higman, \emph{Finitely presented infinite simple groups}, Department of
  Pure Mathematics, Department of Mathematics, I.A.S. Australian National
  University, Canberra, 1974, Notes on Pure Mathematics, No. 8 (1974).
  \MR{MR0376874 (51 \#13049)}

\bibitem{HirschPughShub}
M.~W. Hirsch, C.~C. Pugh, and M.~Shub, \emph{Invariant manifolds}, Lecture
  Notes in Mathematics, Vol. 583, Springer-Verlag, Berlin, 1977. \MR{MR0501173
  (58 \#18595)}

\bibitem{Hurder-personalcommunication}
Steve Hurder, personal communication.

\bibitem{Hurder-survey}
Steven Hurder, \emph{A survey of rigidity theory for {A}nosov actions},
  Differential topology, foliations, and group actions (Rio de Janeiro, 1992),
  Contemp. Math., vol. 161, Amer. Math. Soc., Providence, RI, 1994,
  pp.~143--173. \MR{MR1271833 (95b:58112)}

\bibitem{JohnsonMillson}
Dennis Johnson and John~J. Millson, \emph{Deformation spaces associated to
  compact hyperbolic manifolds}, Discrete groups in geometry and analysis
  ({N}ew {H}aven, {C}onn., 1984), Progr. Math., vol.~67, Birkh\"auser Boston,
  Boston, MA, 1987, pp.~48--106. \MR{MR900823 (88j:22010)}

\bibitem{Jost-Yau}
J{\"u}rgen Jost and Shing-Tung Yau, \emph{Harmonic maps and superrigidity},
  Differential geometry: partial differential equations on manifolds (Los
  Angeles, CA, 1990), Proc. Sympos. Pure Math., vol.~54, Amer. Math. Soc.,
  Providence, RI, 1993, pp.~245--280. \MR{MR1216587 (94m:58060)}

\bibitem{JostZuo}
J{\"u}rgen Jost and Kang Zuo, \emph{Harmonic maps of infinite energy and
  rigidity results for representations of fundamental groups of quasiprojective
  varieties}, J. Differential Geom. \textbf{47} (1997), no.~3, 469--503.
  \MR{MR1617644 (99a:58046)}

\bibitem{KaimanovichMasur}
Vadim~A. Kaimanovich and Howard Masur, \emph{The {P}oisson boundary of the
  mapping class group}, Invent. Math. \textbf{125} (1996), no.~2, 221--264.
  \MR{MR1395719 (97m:32033)}

\bibitem{KalininKatokRodriguezHertz}
Anatole Rodriguez-Hertz~Federico Kalinin, Boris;~Katok, \emph{Non-uniform
  measure rigidity}, preprint.

\bibitem{KalininSpatzier}
Boris Kalinin and Ralf Spatzier, \emph{On the classification of {C}artan
  actions}, Geom. Funct. Anal. \textbf{17} (2007), no.~2, 468--490.
  \MR{MR2322492 (2008i:37054)}

\bibitem{Kassabov-personal}
Martin Kassabov, \emph{personal communication,}, 2008.

\bibitem{KatokLewis-global}
A.~Katok and J.~Lewis, \emph{Global rigidity results for lattice actions on
  tori and new examples of volume-preserving actions}, Israel J. Math.
  \textbf{93} (1996), 253--280. \MR{MR1380646 (96k:22021)}

\bibitem{KatokLewisZimmer}
A.~Katok, J.~Lewis, and R.~Zimmer, \emph{Cocycle superrigidity and rigidity for
  lattice actions on tori}, Topology \textbf{35} (1996), no.~1, 27--38.
  \MR{MR1367273 (97e:22009)}

\bibitem{KatokSpatzier}
A.~Katok and R.~J. Spatzier, \emph{Differential rigidity of {A}nosov actions of
  higher rank abelian groups and algebraic lattice actions}, Tr. Mat. Inst.
  Steklova \textbf{216} (1997), no.~Din. Sist. i Smezhnye Vopr., 292--319.
  \MR{MR1632177 (99i:58118)}

\bibitem{Kazhdan-T}
D.~A. Ka{\v{z}}dan, \emph{On the connection of the dual space of a group with
  the structure of its closed subgroups}, Funkcional. Anal. i Prilo\v zen.
  \textbf{1} (1967), 71--74. \MR{MR0209390 (35 \#288)}

\bibitem{Kazhdan-su}
David Kazhdan, \emph{Some applications of the {W}eil representation}, J.
  Analyse Mat. \textbf{32} (1977), 235--248. \MR{MR0492089 (58 \#11243)}

\bibitem{Kleiner-ICM}
Bruce Kleiner, \emph{The asymptotic geometry of negatively curved spaces:
  uniformization, geometrization and rigidity}, International {C}ongress of
  {M}athematicians. {V}ol. {II}, Eur. Math. Soc., Z\"urich, 2006, pp.~743--768.
  \MR{MR2275621 (2007k:53054)}

\bibitem{Knapp-book}
Anthony~W. Knapp, \emph{Lie groups beyond an introduction}, Progress in
  Mathematics, vol. 140, Birkh\"auser Boston Inc., Boston, MA, 1996.
  \MR{MR1399083 (98b:22002)}

\bibitem{Kobayashi-book}
Shoshichi Kobayashi, \emph{Transformation groups in differential geometry},
  Classics in Mathematics, Springer-Verlag, Berlin, 1995, Reprint of the 1972
  edition. \MR{MR1336823 (96c:53040)}

\bibitem{Kobayashi-survey}
Toshiyuki Kobayashi, \emph{Discontinuous groups for non-{R}iemannian
  homogeneous spaces}, Mathematics unlimited---2001 and beyond, Springer,
  Berlin, 2001, pp.~723--747. \MR{MR1852186 (2002f:53086)}

\bibitem{KobayashiYoshino}
Toshiyuki Kobayashi and Taro Yoshino, \emph{Compact {C}lifford-{K}lein forms of
  symmetric spaces---revisited}, Pure Appl. Math. Q. \textbf{1} (2005), no.~3,
  part 2, 591--663. \MR{MR2201328 (2007h:22013)}

\bibitem{KorevaarSchoen3}
Nicholas Korevaar and Richard Schoen, \emph{get title}, preprint 1999.

\bibitem{KorevaarSchoen2}
Nicholas~J. Korevaar and Richard~M. Schoen, \emph{Sobolev spaces and harmonic
  maps for metric space targets}, Comm. Anal. Geom. \textbf{1} (1993), no.~3-4,
  561--659. \MR{MR1266480 (95b:58043)}

\bibitem{KorevaarSchoen1}
\bysame, \emph{Global existence theorems for harmonic maps to non-locally
  compact spaces}, Comm. Anal. Geom. \textbf{5} (1997), no.~2, 333--387.
  \MR{MR1483983 (99b:58061)}

\bibitem{Kowalsky-thesis}
Nadine Kowalsky, \emph{Noncompact simple automorphism groups of {L}orentz
  manifolds and other geometric manifolds}, Ann. of Math. (2) \textbf{144}
  (1996), no.~3, 611--640. \MR{MR1426887 (98g:57059)}

\bibitem{Labourie-homogeneoussurvey}
F.~Labourie, \emph{Quelques r\'esultats r\'ecents sur les espaces localement
  homog\`enes compacts}, Manifolds and geometry (Pisa, 1993), Sympos. Math.,
  XXXVI, Cambridge Univ. Press, Cambridge, 1996, pp.~267--283. \MR{MR1410076
  (97h:53055)}

\bibitem{LabourieMozesZimmer}
F.~Labourie, S.~Mozes, and R.~J. Zimmer, \emph{On manifolds locally modelled on
  non-{R}iemannian homogeneous spaces}, Geom. Funct. Anal. \textbf{5} (1995),
  no.~6, 955--965. \MR{MR1361517 (97j:53053)}

\bibitem{Labourie-survey}
Fran{\c{c}}ois Labourie, \emph{Large groups actions on manifolds}, Proceedings
  of the International Congress of Mathematicians, Vol. II (Berlin, 1998), no.
  Extra Vol. II, 1998, pp.~371--380 (electronic). \MR{MR1648087 (99k:53069)}

\bibitem{LabourieZimmer}
Fran{\c{c}}ois Labourie and Robert~J. Zimmer, \emph{On the non-existence of
  cocompact lattices for {${\rm SL}(n)/{\rm SL}(m)$}}, Math. Res. Lett.
  \textbf{2} (1995), no.~1, 75--77. \MR{MR1312978 (96d:22014)}

\bibitem{LifchitzMorris-bgline}
Lucy Lifschitz and Dave~Witte Morris, \emph{Bounded generation and lattices
  that cannot act on the line}, Pure Appl. Math. Q. \textbf{4} (2008), no.~1,
  part 2, 99--126. \MR{MR2405997}

\bibitem{Lindenstrauss-survey}
Elon Lindenstrauss, \emph{Rigidity of multiparameter actions}, Israel J. Math.
  \textbf{149} (2005), 199--226, Probability in mathematics. \MR{MR2191215
  (2006j:37007)}

\bibitem{Livne-thesis}
Ron Livne, \emph{On certain covers of the universal elliptic curve},
  unpublished Harvard Ph.d Thesis.

\bibitem{Lubotzky-book}
Alexander Lubotzky, \emph{Discrete groups, expanding graphs and invariant
  measures}, Progress in Mathematics, vol. 125, Birkh\"auser Verlag, Basel,
  1994, With an appendix by Jonathan D. Rogawski. \MR{MR1308046 (96g:22018)}

\bibitem{Lubotzky-free}
\bysame, \emph{Free quotients and the first {B}etti number of some hyperbolic
  manifolds}, Transform. Groups \textbf{1} (1996), no.~1-2, 71--82.
  \MR{MR1390750 (97d:57016)}

\bibitem{LubotzkyMozesRaghunathan}
Alexander Lubotzky, Shahar Mozes, and M.~S. Raghunathan, \emph{The word and
  {R}iemannian metrics on lattices of semisimple groups}, Inst. Hautes \'Etudes
  Sci. Publ. Math. (2000), no.~91, 5--53 (2001). \MR{MR1828742 (2002e:22011)}

\bibitem{LubotzkyZimmer-canonical}
Alexander Lubotzky and Robert~J. Zimmer, \emph{A canonical arithmetic quotient
  for simple {L}ie group actions}, Lie groups and ergodic theory (Mumbai,
  1996), Tata Inst. Fund. Res. Stud. Math., vol.~14, Tata Inst. Fund. Res.,
  Bombay, 1998, pp.~131--142. \MR{MR1699362 (2000m:22013)}

\bibitem{LubotzkyZimmer-nontrivial}
\bysame, \emph{Arithmetic structure of fundamental groups and actions of
  semisimple {L}ie groups}, Topology \textbf{40} (2001), no.~4, 851--869.
  \MR{MR1851566 (2002f:22017)}

\bibitem{Mackey-virtual}
George~W. Mackey, \emph{Ergodic theory and virtual groups}, Math. Ann.
  \textbf{166} (1966), 187--207. \MR{MR0201562 (34 \#1444)}

\bibitem{Maleshich}
{\u{I}}ozhe Maleshich, \emph{The {H}ilbert-{S}mith conjecture for {H}\"older
  actions}, Uspekhi Mat. Nauk \textbf{52} (1997), no.~2(314), 173--174.
  \MR{MR1480156 (99d:57026)}

\bibitem{Mane}
Ricardo Ma{\~n}{\'e}, \emph{Quasi-{A}nosov diffeomorphisms and hyperbolic
  manifolds}, Trans. Amer. Math. Soc. \textbf{229} (1977), 351--370.
  \MR{MR0482849 (58 \#2894)}

\bibitem{MannSu}
L.~N. Mann and J.~C. Su, \emph{Actions of elementary {$p$}-groups on
  manifolds}, Trans. Amer. Math. Soc. \textbf{106} (1963), 115--126.
  \MR{MR0143840 (26 \#1390)}

\bibitem{Margulis-nonuniformone}
G.~A. Margulis, \emph{Arithmetic properties of discrete subgroups}, Uspehi Mat.
  Nauk \textbf{29} (1974), no.~1 (175), 49--98. \MR{MR0463353 (57 \#3306a)}

\bibitem{Margulis-ICM}
\bysame, \emph{Discrete groups of motions of manifolds of nonpositive
  curvature}, Proceedings of the International Congress of Mathematicians
  (Vancouver, B.C., 1974), Vol. 2, Canad. Math. Congress, Montreal, Que., 1975,
  pp.~21--34. \MR{MR0492072 (58 \#11226)}

\bibitem{Margulis-nonuniformtwo}
\bysame, \emph{Non-uniform lattices in semisimple algebraic groups}, Lie groups
  and their representations (Proc. Summer School on Group Representations of
  the Bolyai J\'anos Math. Soc., Budapest, 1971), Halsted, New York, 1975,
  pp.~371--553. \MR{MR0422499 (54 \#10486)}

\bibitem{Margulis-Inventsuperrigid}
\bysame, \emph{Arithmeticity of the irreducible lattices in the semisimple
  groups of rank greater than {$1$}}, Invent. Math. \textbf{76} (1984), no.~1,
  93--120. \MR{MR739627 (85j:22021)}

\bibitem{Margulis-book}
\bysame, \emph{Discrete subgroups of semisimple {L}ie groups}, Ergebnisse der
  Mathematik und ihrer Grenzgebiete (3) [Results in Mathematics and Related
  Areas (3)], vol.~17, Springer-Verlag, Berlin, 1991. \MR{MR1090825
  (92h:22021)}

\bibitem{Margulis-tits}
Gregory Margulis, \emph{Free subgroups of the homeomorphism group of the
  circle}, C. R. Acad. Sci. Paris S\'er. I Math. \textbf{331} (2000), no.~9,
  669--674. \MR{MR1797749 (2002b:37034)}

\bibitem{Margulis-millenium}
\bysame, \emph{Problems and conjectures in rigidity theory}, Mathematics:
  frontiers and perspectives, Amer. Math. Soc., Providence, RI, 2000,
  pp.~161--174. \MR{MR1754775 (2001d:22008)}

\bibitem{MargulisQian}
Gregory~A. Margulis and Nantian Qian, \emph{Rigidity of weakly hyperbolic
  actions of higher real rank semisimple {L}ie groups and their lattices},
  Ergodic Theory Dynam. Systems \textbf{21} (2001), no.~1, 121--164.
  \MR{MR1826664 (2003a:22019)}

\bibitem{Mather-commutatorI}
John~N. Mather, \emph{Commutators of diffeomorphisms}, Comment. Math. Helv.
  \textbf{49} (1974), 512--528. \MR{MR0356129 (50 \#8600)}

\bibitem{Mather-BAMS}
\bysame, \emph{Simplicity of certain groups of diffeomorphisms}, Bull. Amer.
  Math. Soc. \textbf{80} (1974), 271--273. \MR{MR0339268 (49 \#4028)}

\bibitem{Mather-commutatorII}
\bysame, \emph{Commutators of diffeomorphisms. {II}}, Comment. Math. Helv.
  \textbf{50} (1975), 33--40. \MR{MR0375382 (51 \#11576)}

\bibitem{Mather-ICM}
\bysame, \emph{Foliations and local homology of groups of diffeomorphisms},
  Proceedings of the {I}nternational {C}ongress of {M}athematicians
  ({V}ancouver, {B}. {C}., 1974), {V}ol. 2, Canad. Math. Congress, Montreal,
  Que., 1975, pp.~35--37. \MR{MR0431203 (55 \#4205)}

\bibitem{Matsushima-Murakami}
Yoz{\^o} Matsushima and Shingo Murakami, \emph{On vector bundle valued harmonic
  forms and automorphic forms on symmetric riemannian manifolds}, Ann. of Math.
  (2) \textbf{78} (1963), 365--416. \MR{MR0153028 (27 \#2997)}

\bibitem{Mok-Siu-Yeung}
Ngaiming Mok, Yum~Tong Siu, and Sai-Kee Yeung, \emph{Geometric superrigidity},
  Invent. Math. \textbf{113} (1993), no.~1, 57--83. \MR{MR1223224 (94h:53079)}

\bibitem{Monod-irreducible}
Nicolas Monod, \emph{Superrigidity for irreducible lattices and geometric
  splitting}, J. Amer. Math. Soc. \textbf{19} (2006), no.~4, 781--814
  (electronic). \MR{MR2219304 (2007b:22025)}

\bibitem{Morris-amenableline}
Dave~Witte Morris, \emph{Amenable groups that act on the line}, Algebr. Geom.
  Topol. \textbf{6} (2006), 2509--2518. \MR{MR2286034 (2008c:20078)}

\bibitem{MorrisZimmer-compactcocycles}
Dave~Witte Morris and Robert~J. Zimmer, \emph{Ergodic actions of semisimple
  {L}ie groups on compact principal bundles}, Geom. Dedicata \textbf{106}
  (2004), 11--27. \MR{MR2079831 (2005e:22016)}

\bibitem{Mostow-rigidity}
G.~D. Mostow, \emph{Quasi-conformal mappings in {$n$}-space and the rigidity of
  hyperbolic space forms}, Inst. Hautes \'Etudes Sci. Publ. Math. (1968),
  no.~34, 53--104. \MR{MR0236383 (38 \#4679)}

\bibitem{Mostow-book}
\bysame, \emph{Strong rigidity of locally symmetric spaces}, Princeton
  University Press, Princeton, N.J., 1973, Annals of Mathematics Studies, No.
  78. \MR{MR0385004 (52 \#5874)}

\bibitem{Mostow-personal}
G.D. Mostow, \emph{personal communication,}, 2000.

\bibitem{Navas-book}
Andr\'{e}s Navas, \emph{Groups of circle diffeomorphisms}, book in preprint
  form, 2008.

\bibitem{Navas-orderdynamics}
\bysame, \emph{On the dynamics of (left) orderable groups}, preprint.

\bibitem{Navas-Circlediff}
Andr{\'e}s Navas, \emph{Actions de groupes de {K}azhdan sur le cercle}, Ann.
  Sci. \'Ecole Norm. Sup. (4) \textbf{35} (2002), no.~5, 749--758.
  \MR{MR1951442 (2003j:58013)}

\bibitem{Navas-circleII}
\bysame, \emph{Quelques nouveaux ph\'enom\`enes de rang 1 pour les groupes de
  diff\'eomorphismes du cercle}, Comment. Math. Helv. \textbf{80} (2005),
  no.~2, 355--375. \MR{MR2142246 (2006j:57003)}

\bibitem{NevoZimmer-okalready}
Amos Nevo and Robert~J. Zimmer, \emph{Invariant rigid geometric structures and
  smooth projective factors}, preprint.

\bibitem{NevoZimmer-firstquotients}
\bysame, \emph{Homogenous projective factors for actions of semi-simple {L}ie
  groups}, Invent. Math. \textbf{138} (1999), no.~2, 229--252. \MR{MR1720183
  (2000h:22006)}

\bibitem{NevoZimmer-entropy}
\bysame, \emph{Rigidity of {F}urstenberg entropy for semisimple {L}ie group
  actions}, Ann. Sci. \'Ecole Norm. Sup. (4) \textbf{33} (2000), no.~3,
  321--343. \MR{MR1775184 (2001k:22009)}

\bibitem{NevoZimmer-survey}
\bysame, \emph{Actions of semisimple {L}ie groups with stationary measure},
  Rigidity in dynamics and geometry ({C}ambridge, 2000), Springer, Berlin,
  2002, pp.~321--343. \MR{MR1919409 (2003j:22029)}

\bibitem{NevoZimmer-intermediate}
\bysame, \emph{A generalization of the intermediate factors theorem}, J. Anal.
  Math. \textbf{86} (2002), 93--104. \MR{MR1894478 (2003f:22019)}

\bibitem{NevoZimmer-annals}
\bysame, \emph{A structure theorem for actions of semisimple {L}ie groups},
  Ann. of Math. (2) \textbf{156} (2002), no.~2, 565--594. \MR{MR1933077
  (2003i:22024)}

\bibitem{NevoZimmer-deRham}
\bysame, \emph{Entropy of stationary measures and bounded tangential de-{R}ham
  cohomology of semisimple {L}ie group actions}, Geom. Dedicata \textbf{115}
  (2005), 181--199. \MR{MR2180047 (2006k:22008)}

\bibitem{Ollivier-book}
Yann Ollivier, \emph{A {J}anuary 2005 invitation to random groups}, Ensaios
  Matem\'aticos [Mathematical Surveys], vol.~10, Sociedade Brasileira de
  Matem\'atica, Rio de Janeiro, 2005. \MR{MR2205306}

\bibitem{PlatonovRapinchuk}
Vladimir Platonov and Andrei Rapinchuk, \emph{Algebraic groups and number
  theory}, Pure and Applied Mathematics, vol. 139, Academic Press Inc., Boston,
  MA, 1994, Translated from the 1991 Russian original by Rachel Rowen.
  \MR{MR1278263 (95b:11039)}

\bibitem{Polterovich-book}
Leonid Polterovich, \emph{The geometry of the group of symplectic
  diffeomorphisms}, Lectures in Mathematics ETH Z\"urich, Birkh\"auser Verlag,
  Basel, 2001. \MR{MR1826128 (2002g:53157)}

\bibitem{Polterovich}
\bysame, \emph{Growth of maps, distortion in groups and symplectic geometry},
  Invent. Math. \textbf{150} (2002), no.~3, 655--686. \MR{MR1946555
  (2003i:53126)}

\bibitem{Popa-ICM}
Sorin Popa, \emph{Deformation and rigidity for group actions and von {N}eumann
  algebras}, International {C}ongress of {M}athematicians. {V}ol. {I}, Eur.
  Math. Soc., Z\"urich, 2007, pp.~445--477. \MR{MR2334200}

\bibitem{Py-thisvolume}
Pierre Py, \emph{Some remarks on area preserving actions of lattices}, in this
  volume.

\bibitem{Py-CRAS}
\bysame, \emph{Quasi-morphismes de {C}alabi et graphe de {R}eeb sur le tore},
  C. R. Math. Acad. Sci. Paris \textbf{343} (2006), no.~5, 323--328.
  \MR{MR2253051 (2007e:53116)}

\bibitem{Py-thesis}
\bysame, \emph{Quasi-morphismes et invariant de {C}alabi}, Ann. Sci. \'Ecole
  Norm. Sup. (4) \textbf{39} (2006), no.~1, 177--195. \MR{MR2224660
  (2007f:53116)}

\bibitem{CandelQuiroga-three}
R.~Quiroga-Barranco and A.~Candel, \emph{Rigid and finite type geometric
  structures}, Geom. Dedicata \textbf{106} (2004), 123--143. \MR{MR2079838
  (2005d:53068)}

\bibitem{Raghunathan-russian}
M.~Raghunathan, \emph{Diskretnye podgruppy grupp {L}i}, Izdat. ``Mir'', Moscow,
  1977, Translated from the English by O. V. \v Svarcman, Edited by \`E. B.
  Vinberg, With a supplement ``Arithmeticity of irreducible lattices in
  semisimple groups of rank greater than 1'' by G. A. Margulis. \MR{MR0507236
  (58 \#22394b)}

\bibitem{Raghunathan-1-vanishing}
M.~S. Raghunathan, \emph{On the first cohomology of discrete subgroups of
  semisimple {L}ie groups}, Amer. J. Math. \textbf{87} (1965), 103--139.
  \MR{MR0173730 (30 \#3940)}

\bibitem{Ratner}
Marina Ratner, \emph{On {R}aghunathan's measure conjecture}, Ann. of Math. (2)
  \textbf{134} (1991), no.~3, 545--607. \MR{MR1135878 (93a:22009)}

\bibitem{Rebelo-T2}
Julio~C. Rebelo, \emph{On nilpotent groups of real analytic diffeomorphisms of
  the torus}, C. R. Acad. Sci. Paris S\'er. I Math. \textbf{331} (2000), no.~4,
  317--322. \MR{MR1787192 (2001m:22041)}

\bibitem{RebeloSilva-burnside}
Julio~C. Rebelo and Ana~L. Silva, \emph{On the {B}urnside problem in {${\rm
  Diff}(M)$}}, Discrete Contin. Dyn. Syst. \textbf{17} (2007), no.~2, 423--439.
  \MR{MR2257443 (2007j:53107)}

\bibitem{Repovus}
Du{\u{s}}an Repov{\u{s}} and Evgenij {\u{S}}{\u{c}}epin, \emph{A proof of the
  {H}ilbert-{S}mith conjecture for actions by {L}ipschitz maps}, Math. Ann.
  \textbf{308} (1997), no.~2, 361--364. \MR{MR1464908 (99c:57066)}

\bibitem{Rieri}
I.~Mundet~I Rieri, \emph{Jordan's theorem for the diffeomorphism group of some
  manifolds}, preprint, 2008.

\bibitem{Robinson-book}
Derek J.~S. Robinson, \emph{A course in the theory of groups}, second ed.,
  Graduate Texts in Mathematics, vol.~80, Springer-Verlag, New York, 1996.
  \MR{MR1357169 (96f:20001)}

\bibitem{RodriguezHertz}
Federico Rodriguez~Hertz, \emph{Global rigidity of certain abelian actions by
  toral automorphisms}, J. Mod. Dyn. \textbf{1} (2007), no.~3, 425--442.
  \MR{MR2318497 (2008f:37063)}

\bibitem{Saper}
Leslie Saper, \emph{Tilings and finite energy retractions of locally symmetric
  spaces}, Comment. Math. Helv. \textbf{72} (1997), no.~2, 167--202.
  \MR{MR1470087 (99a:22019)}

\bibitem{Schmidt-thesis}
B.~Schmidt, \emph{Weakly hyperbolic actions of {K}azhdan groups on tori}, Geom.
  Funct. Anal. \textbf{16} (2006), no.~5, 1139--1156. \MR{MR2276535
  (2008b:37045)}

\bibitem{Selberg}
Atle Selberg, \emph{On discontinuous groups in higher-dimensional symmetric
  spaces}, Contributions to function theory (internat. Colloq. Function Theory,
  Bombay, 1960), Tata Institute of Fundamental Research, Bombay, 1960,
  pp.~147--164. \MR{MR0130324 (24 \#A188)}

\bibitem{Shah}
Nimish~A. Shah, \emph{Invariant measures and orbit closures on homogeneous
  spaces for actions of subgroups generated by unipotent elements}, Lie groups
  and ergodic theory (Mumbai, 1996), Tata Inst. Fund. Res. Stud. Math.,
  vol.~14, Tata Inst. Fund. Res., Bombay, 1998, pp.~229--271. \MR{MR1699367
  (2001a:22012)}

\bibitem{Shalom-Annals}
Yehuda Shalom, \emph{Rigidity, unitary representations of semisimple groups,
  and fundamental groups of manifolds with rank one transformation group}, Ann.
  of Math. (2) \textbf{152} (2000), no.~1, 113--182. \MR{MR1792293
  (2001m:22022)}

\bibitem{Shalom-ECM}
\bysame, \emph{Measurable group theory}, European Congress of Mathematics, Eur.
  Math. Soc., Z\"urich, 2005, pp.~391--423. \MR{MR2185757 (2006k:37007)}

\bibitem{Shalom-ICM}
\bysame, \emph{The algebraization of {K}azhdan's property ({T})}, International
  Congress of Mathematicians. Vol. II, Eur. Math. Soc., Z\"urich, 2006,
  pp.~1283--1310. \MR{MR2275645 (2008a:22003)}

\bibitem{Sharpe-book}
R.~W. Sharpe, \emph{Differential geometry}, Graduate Texts in Mathematics, vol.
  166, Springer-Verlag, New York, 1997, Cartan's generalization of Klein's
  Erlangen program, With a foreword by S. S. Chern. \MR{MR1453120 (98m:53033)}

\bibitem{Shenfield}
Daniel Shenfield, \emph{Semisimple representations of ${SL}(n, {Z}[{X}_1,
  \ldots, {X}_k])$}, to appear, {G}roups, {G}eometry and {D}ynamics.

\bibitem{Siu-YaleTalk}
Y.T. Siu, \emph{remarks in talk at margulis' 60th birthday conference,},
  February, 2006.

\bibitem{Siu-Mostow}
Yum~Tong Siu, \emph{The complex-analyticity of harmonic maps and the strong
  rigidity of compact {K}\"ahler manifolds}, Ann. of Math. (2) \textbf{112}
  (1980), no.~1, 73--111. \MR{MR584075 (81j:53061)}

\bibitem{Smale}
S.~Smale, \emph{Differentiable dynamical systems}, Bull. Amer. Math. Soc.
  \textbf{73} (1967), 747--817. \MR{MR0228014 (37 \#3598)}

\bibitem{Spatzier-alexandria}
R.~J. Spatzier, \emph{Harmonic analysis in rigidity theory}, Ergodic theory and
  its connections with harmonic analysis ({A}lexandria, 1993), London Math.
  Soc. Lecture Note Ser., vol. 205, Cambridge Univ. Press, Cambridge, 1995,
  pp.~153--205. \MR{MR1325698 (96c:22019)}

\bibitem{Spatzier-survey}
\bysame, \emph{An invitation to rigidity theory}, Modern dynamical systems and
  applications, Cambridge Univ. Press, Cambridge, 2004, pp.~211--231.
  \MR{MR2090772 (2006a:53041)}

\bibitem{SpatzierZimmer}
Ralf~J. Spatzier and Robert~J. Zimmer, \emph{Fundamental groups of negatively
  curved manifolds and actions of semisimple groups}, Topology \textbf{30}
  (1991), no.~4, 591--601. \MR{MR1133874 (92m:57047)}

\bibitem{Stein-transam}
Melanie Stein, \emph{Groups of piecewise linear homeomorphisms}, Trans. Amer.
  Math. Soc. \textbf{332} (1992), no.~2, 477--514. \MR{MR1094555 (92k:20075)}

\bibitem{Stuck}
Garrett Stuck, \emph{Minimal actions of semisimple groups}, Ergodic Theory
  Dynam. Systems \textbf{16} (1996), no.~4, 821--831. \MR{MR1406436
  (98a:57046)}

\bibitem{Thurston-BAMS}
William Thurston, \emph{Foliations and groups of diffeomorphisms}, Bull. Amer.
  Math. Soc. \textbf{80} (1974), 304--307. \MR{MR0339267 (49 \#4027)}

\bibitem{Thurston-Reeb}
William~P. Thurston, \emph{A generalization of the {R}eeb stability theorem},
  Topology \textbf{13} (1974), 347--352. \MR{MR0356087 (50 \#8558)}

\bibitem{Tits}
J.~Tits, \emph{Free subgroups in linear groups}, J. Algebra \textbf{20} (1972),
  250--270. \MR{MR0286898 (44 \#4105)}

\bibitem{BaderFrancesMelnick}
Charles~Frances Uri~Bader and Karin Melnick, \emph{An embedding theorem for
  automorphism groups of cartan geometries}, to appear GAFA.

\bibitem{Valette-PairPaper}
Alain Valette, \emph{Group pairs with property ({T}), from arithmetic
  lattices}, Geom. Dedicata \textbf{112} (2005), 183--196. \MR{MR2163898
  (2006d:22014)}

\bibitem{Weil-II}
Andr{\'e} Weil, \emph{On discrete subgroups of {L}ie groups}, Ann. of Math. (2)
  \textbf{72} (1960), 369--384. \MR{MR0137792 (25 \#1241)}

\bibitem{Weil-I}
\bysame, \emph{On discrete subgroups of {L}ie groups. {II}}, Ann. of Math. (2)
  \textbf{75} (1962), 578--602. \MR{MR0137793 (25 \#1242)}

\bibitem{Weil-remark}
\bysame, \emph{Remarks on the cohomology of groups}, Ann. of Math. (2)
  \textbf{80} (1964), 149--157. \MR{MR0169956 (30 \#199)}

\bibitem{Weinberger-thisvolume}
Shmuel Weinberger, \emph{Some trivial remarks on the ${C}^0$ zimmer program},
  in this volume.

\bibitem{Wilson}
John~S. Wilson, \emph{On exponential growth and uniformly exponential growth
  for groups}, Invent. Math. \textbf{155} (2004), no.~2, 287--303.
  \MR{MR2031429 (2004k:20085)}

\bibitem{Witte-circle}
Dave Witte, \emph{Arithmetic groups of higher {${\bf Q}$}-rank cannot act on
  {$1$}-manifolds}, Proc. Amer. Math. Soc. \textbf{122} (1994), no.~2,
  333--340. \MR{MR1198459 (95a:22014)}

\bibitem{Witte-measquotients}
\bysame, \emph{Measurable quotients of unipotent translations on homogeneous
  spaces}, Trans. Amer. Math. Soc. \textbf{345} (1994), no.~2, 577--594.
  \MR{MR1181187 (95a:22005)}

\bibitem{Witte-csr}
\bysame, \emph{Cocycle superrigidity for ergodic actions of non-semisimple
  {L}ie groups}, Lie groups and ergodic theory (Mumbai, 1996), Tata Inst. Fund.
  Res. Stud. Math., vol.~14, Tata Inst. Fund. Res., Bombay, 1998, pp.~367--386.
  \MR{MR1699372 (2000i:22008)}

\bibitem{WitteMorris-thisvolume}
Dave Witte~Morris, \emph{Can lattices in ${S}{L}(n, \mathbb{{R}})$ act on the
  circle?}, in this volume.

\bibitem{WitteMorris-book}
\bysame, \emph{An introduction to arithmetic groups.}, book in preprint form.

\bibitem{Yeung-MCG}
Sai-Kee Yeung, \emph{Representations of semisimple lattices in mapping class
  groups}, Int. Math. Res. Not. (2003), no.~31, 1677--1686. \MR{MR1981481
  (2004d:53047)}

\bibitem{Zeghib-lorentz1}
A.~Zeghib, \emph{Isometry groups and geodesic foliations of {L}orentz
  manifolds. {I}. {F}oundations of {L}orentz dynamics}, Geom. Funct. Anal.
  \textbf{9} (1999), no.~4, 775--822. \MR{MR1719606 (2001g:53059)}

\bibitem{Zeghib-lorentz2}
\bysame, \emph{Isometry groups and geodesic foliations of {L}orentz manifolds.
  {II}. {G}eometry of analytic {L}orentz manifolds with large isometry groups},
  Geom. Funct. Anal. \textbf{9} (1999), no.~4, 823--854. \MR{MR1719610
  (2001g:53060)}

\bibitem{Zeghib-CAG}
Abdelghani Zeghib, \emph{Le groupe affine d'une vari\'et\'e riemannienne
  compacte}, Comm. Anal. Geom. \textbf{5} (1997), no.~1, 199--211.
  \MR{MR1456311 (98g:53065)}

\bibitem{Zeghib-affine}
\bysame, \emph{Sur les actions affines des groupes discrets}, Ann. Inst.
  Fourier (Grenoble) \textbf{47} (1997), no.~2, 641--685. \MR{MR1450429
  (98d:57068)}

\bibitem{Zimmer-tcsr}
Robert~J. Zimmer, \emph{Topological superrigidity}, unpublished notes.

\bibitem{Zimmer-book}
\bysame, \emph{Ergodic theory and semisimple groups}, Monographs in
  Mathematics, vol.~81, Birkh\"auser Verlag, Basel, 1984. \MR{MR776417
  (86j:22014)}

\bibitem{Zimmer-Finitetype}
\bysame, \emph{Actions of lattices in semisimple groups preserving a
  {$G$}-structure of finite type}, Ergodic Theory Dynam. Systems \textbf{5}
  (1985), no.~2, 301--306. \MR{MR796757 (87g:22011)}

\bibitem{Zimmer-distal}
\bysame, \emph{Lattices in semisimple groups and distal geometric structures},
  Invent. Math. \textbf{80} (1985), no.~1, 123--137. \MR{MR784532 (86i:57056)}

\bibitem{Zimmer-affine}
\bysame, \emph{On connection-preserving actions of discrete linear groups},
  Ergodic Theory Dynam. Systems \textbf{6} (1986), no.~4, 639--644.
  \MR{MR873437 (88g:57045)}

\bibitem{Zimmer-ICM}
\bysame, \emph{Actions of semisimple groups and discrete subgroups},
  Proceedings of the International Congress of Mathematicians, Vol. 1, 2
  (Berkeley, Calif., 1986) (Providence, RI), Amer. Math. Soc., 1987,
  pp.~1247--1258. \MR{MR934329 (89j:22024)}

\bibitem{Zimmer-Mostow}
\bysame, \emph{Lattices in semisimple groups and invariant geometric structures
  on compact manifolds}, Discrete groups in geometry and analysis (New Haven,
  Conn., 1984), Progr. Math., vol.~67, Birkh\"auser Boston, Boston, MA, 1987,
  pp.~152--210. \MR{MR900826 (88i:22025)}

\bibitem{Zimmer-fundamentalgroup1}
\bysame, \emph{Representations of fundamental groups of manifolds with a
  semisimple transformation group}, J. Amer. Math. Soc. \textbf{2} (1989),
  no.~2, 201--213. \MR{MR973308 (90i:22021)}

\bibitem{Zimmer-algebraichull}
\bysame, \emph{On the algebraic hull of an automorphism group of a principal
  bundle}, Comment. Math. Helv. \textbf{65} (1990), no.~3, 375--387.
  \MR{MR1069815 (92f:57050)}

\bibitem{Zimmer-discretespectrum}
\bysame, \emph{Spectrum, entropy, and geometric structures for smooth actions
  of {K}azhdan groups}, Israel J. Math. \textbf{75} (1991), no.~1, 65--80.
  \MR{MR1147291 (93i:22014)}

\bibitem{Zimmer-ratner}
\bysame, \emph{Superrigidity, {R}atner's theorem, and fundamental groups},
  Israel J. Math. \textbf{74} (1991), no.~2-3, 199--207. \MR{MR1135234
  (93b:22019)}

\bibitem{Zimmer-Gromov}
\bysame, \emph{Automorphism groups and fundamental groups of geometric
  manifolds}, Differential geometry: Riemannian geometry (Los Angeles, CA,
  1990), Proc. Sympos. Pure Math., vol.~54, Amer. Math. Soc., Providence, RI,
  1993, pp.~693--710. \MR{MR1216656 (95a:58017)}

\bibitem{Zimmer-homogeneous}
\bysame, \emph{Discrete groups and non-{R}iemannian homogeneous spaces}, J.
  Amer. Math. Soc. \textbf{7} (1994), no.~1, 159--168. \MR{MR1207014
  (94e:22021)}

\bibitem{Zimmer-entropyquotient}
\bysame, \emph{Entropy and arithmetic quotients for simple automorphism groups
  of geometric manifolds}, Geom. Dedicata \textbf{107} (2004), 47--56.
  \MR{MR2110753 (2005i:22011)}

\bibitem{Zimmer-CBMS}
Robert~J. Zimmer and Dave Witte~Morris, \emph{Ergodic theory, groups and
  geometry}, American Mathematical Society, Providence, RI. To appear.

\end{thebibliography}

\end{document}